\documentclass[10pt]{amsart}

\usepackage{amsmath}
\usepackage{amsthm}
\usepackage{amsopn}
\usepackage{amssymb}
\usepackage[all]{xy}
\usepackage{lscape,xcolor}
\usepackage{graphicx}
\usepackage{hyperref}
\usepackage{mathtools}

\allowdisplaybreaks[1]

\newcommand{\sslash}{\mathbin{/\mkern-6mu/}}
\newcommand{\mmod}{\! \sslash \!}

\newcommand{\mc}[1]{\mathcal{#1}}
\newcommand{\ul}[1]{\underline{#1}}

\newcommand{\mr}[1]{\mathrm{#1}}

\newcommand{\mit}[1]{\mathit{#1}}

\newcommand{\abs}[1]{\lvert #1 \rvert}

\newcommand{\bra}[1]{\langle #1 \rangle}
\newcommand{\br}[1]{\overline{#1}}

\newcommand{\td}[1]{\widetilde{#1}}

\newcommand{\ZZ}{\mathbb{Z}}

\newcommand{\QQ}{\mathbb{Q}}

\newcommand{\FF}{\mathbb{F}}

\newcommand{\tmf}{\mathrm{tmf}}
\newcommand{\bo}{\mathrm{bo}}

\newcommand{\KO}{\mr{KO}}

\newcommand{\St}{\mr{St}}

\newcommand{\bou}{\ul{\bo}}

 \newtheorem{thm}[equation]{Theorem}
 \newtheorem{cor}[equation]{Corollary}
 \newtheorem{lem}[equation]{Lemma}
 \newtheorem{prop}[equation]{Proposition}
 
 \newtheorem{sublem}[equation]{Sublemma}
 
 \newtheorem*{thm*}{Theorem}
 \newtheorem*{cor*}{Corollary}
 \newtheorem*{lem*}{Lemma}
 \newtheorem*{prop*}{Proposition}

 \theoremstyle{definition}

 \newtheorem{rmk}[equation]{Remark}

 \newtheorem{question}[equation]{Question}
 
\newtheorem{tab}[equation]{Table}

\newtheorem*{defn*}{Definition}
\newtheorem*{ex*}{Example}
\newtheorem*{exs*}{Examples}
\newtheorem*{rmk*}{Remark}
\newtheorem*{claim*}{Claim}

\numberwithin{equation}{section}
\numberwithin{figure}{section}
\DeclareMathOperator{\Ext}{Ext}

\DeclareMathOperator*{\Tot}{Tot}

\newcommand{\xib}{\zeta}

\newcommand{\E}[2]{\prescript{#1}{#2}{E}}

\title{The $2$-primary Hurewicz image of $\tmf$}
\author{Mark Behrens}
\address{
Dept. of Mathematics \\
University of Notre Dame \\
Notre Dame, IN, U.S.A.
}

\author{Mark Mahowald}
\address{
Dept. of Mathematics \\
Northwestern University \\
Evanston, IL, U.S.A.
}

\author{J.D. Quigley}
\address{
Dept. of Mathematics \\
Cornell University \\
Ithaca, NY, U.S.A.
}

\begin{document}

\begin{abstract}
We determine the image of the 2-primary tmf-Hurewicz homomorphism, where tmf is the spectrum of topological modular forms.  We do this by lifting elements of $\tmf_*$ to the homotopy groups of the generalized Moore spectrum $M(8,v_1^8)$ using a modified form of the Adams spectral sequence and the tmf-resolution, and then proving the existence of a $v_2^{32}$-self map on $M(8,v_1^8)$ to generate 192-periodic families in the stable homotopy groups of spheres.   
\end{abstract}

\maketitle
\tableofcontents

\section{Introduction}\label{sec:intro}

The Hurewicz theorem implies that the Hurewicz homomorphism
$$ h: \pi_*(S^n) \rightarrow \td{H}_*(S^n; \ZZ)$$
is an isomorphism for $* = n$, implying the well known result that the $0$th stable stem is given by
$$ \pi_0^s \cong \ZZ. $$
In his paper \cite{Adams}, Adams studied the Hurewicz homomorphism for real K-theory
$$
h_{\KO}: \pi^s_* \rightarrow \pi_*\KO = \KO^{-*}(pt).
$$
The computation of the real K-theory of a point (the homotopy groups of the spectrum $\KO$ representing real K-theory) 
is a consequence of the Bott periodicity theorem \cite{Bott}: these groups are given by the following 8-fold periodic pattern.
\begin{center}
\begin{tabular}{c|cccccccc}\hline
$n \mod 8$ & 0 & 1 & 2 & 3 & 4 & 5 & 6 & 7
\\ \hline
$\pi_n\KO$ & $\ZZ$ & $\ZZ/2$ & $\ZZ/2$ & 0 & $\ZZ$ & 0 & 0 & 0 
\\ \hline
\end{tabular} 
\end{center}
The map $h_{\KO}$ is an isomorphism in degree $0$, and Adams showed that $h_{\KO}$ is surjective in degrees $* \equiv 1,2 \mod 8$.  He did this by constructing what is now known as a $v_1$-self map
$$ v_1^4: \Sigma^{8}M(2) \rightarrow M(2), $$
where $M(2)$ denotes the mod $2$ Moore spectrum, and considering the projections
$$ \mu_{8j +1 + \epsilon} \in \pi^s_{8j+1+\epsilon} $$
of the elements
\begin{equation}\label{eq:mun}
 \eta^{\epsilon} \cdot v_1^{4j}\td{\eta} \in \pi_{8j+2 + \epsilon} M(2)
\end{equation}
to the top cell of $M(2)$. Here $\td{\eta}$ denotes a lift of $\eta \in \pi_1^s$ to the top cell of $M(2)$ and $\epsilon \in \{0,1\}$.  Because we have 
$$ \pi_*^s \otimes \QQ = 0 $$
for $* > 0$, the homomorphism $h_{\KO}$ is necessarily trivial in positive degrees $* \equiv 0 \mod 4$.
 
Goerss, Hopkins, and Miller constructed the spectrum $\tmf$ of topological modular forms \cite{tmf} as a higher analog of the real $K$-theory spectrum.\footnote{Here, $\tmf$ denotes \emph{connective} topological modular forms.}  The homotopy groups of $\tmf$ are 576-periodic.
The goal of this paper is to determine the image of the $2$-local $\tmf$-Hurewicz homomorphism.  
$$ h_{\tmf}: \pi_*^s \rightarrow \pi_*\tmf_{(2)}. $$
The $3$-primary Hurewicz image has recently been determined by Belmont and Shimomura \cite{BelmontShimomura}. Since $\pi_*\tmf_{(p)}$ has no torsion for $p \ge 5$, the $p$-primary $\tmf$-Hurewicz image is trivial in positive degrees for these primes.
\emph{Henceforth, everything in this paper is implicitly $2$-local.}  
\begin{figure}
\includegraphics[angle = 90, origin=c, height =.7\textheight]{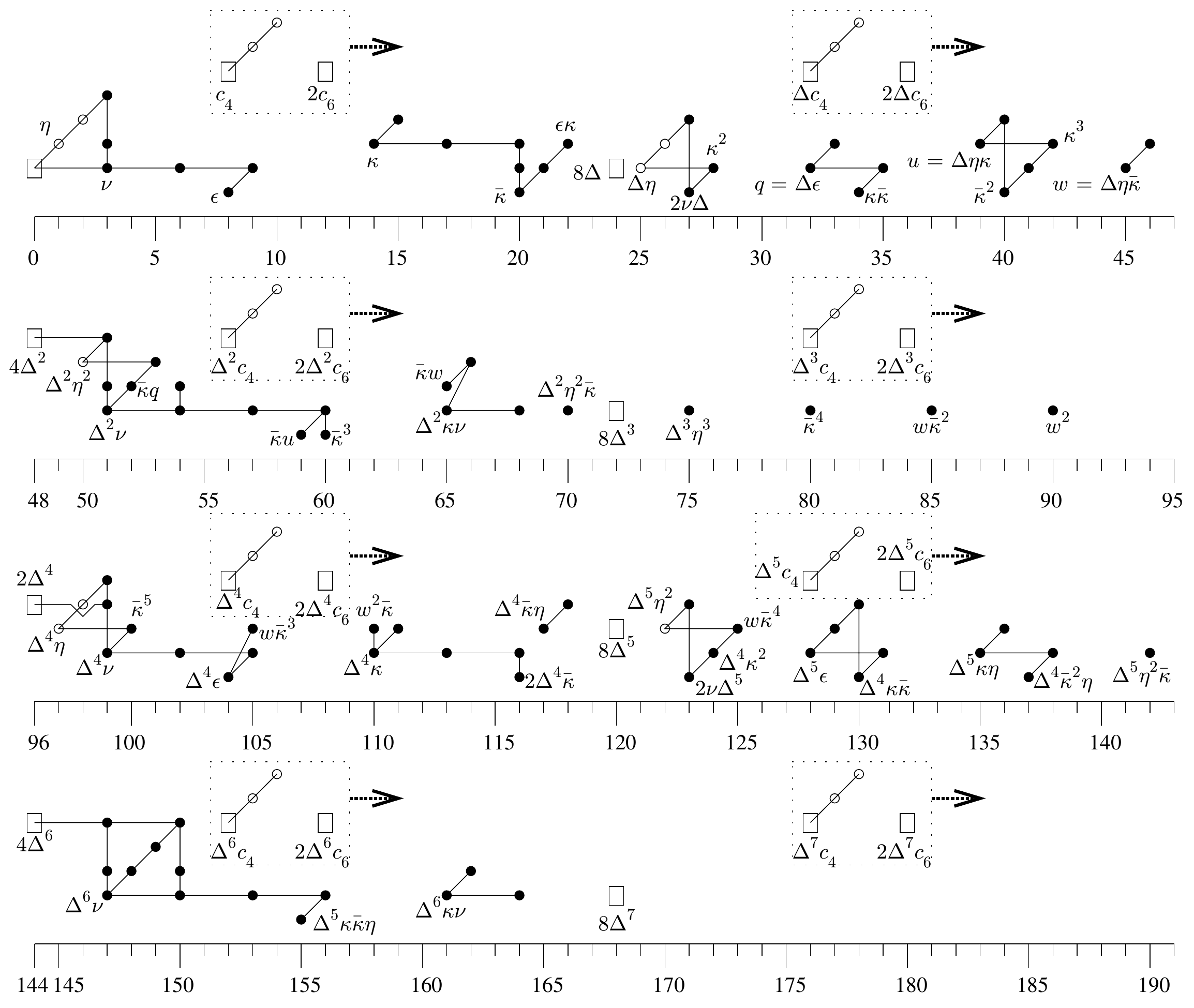}
\caption{The homotopy groups of $\tmf$}\label{fig:tmf2}
\end{figure}

$2$-locally, the homotopy groups of $\tmf$ are merely 192-periodic.  These homotopy groups were originally computed by Hopkins and Mahowald \cite[Ch.~15]{tmf} (see also \cite{Bau08}) using the descent spectral sequence
$$ \Ext^{s,t}_{\Gamma^{\mit{ell}}}(A^{\mit{ell}}, A^{\mit{ell}}) \Rightarrow \pi_{t-s}(\tmf) $$
where $(A^{\mit{ell}}, \Gamma^{\mit{ell}})$ is the elliptic curve Hopf algebroid.
These homotopy groups are displayed in Figure~\ref{fig:tmf2}. In this figure:
\begin{itemize}
\item A series of $i$ black dots joined by vertical lines corresponds to a factor of $\ZZ/2^i$ which is annihilated by some power of $c_4$.

\item An open circle corresponds to a factor of $\ZZ/2$ which is not annihilated by a power of $c_4$.

\item A box indicates a factor of $\ZZ_{(2)}$ which is not annihilated by a power of $c_4$.

\item The non-vertical lines indicate multiplication by $\eta$ and $\nu$.

\item A pattern with a dotted box around it and an arrow emanating from the right face indicates this pattern continues indefinitely to the right by $c_4$-multiplication (i.e. tensor the pattern with $\ZZ_{(2)}[c_4]$).

\item The vertical arrangement of the chart is arbitrary.
\end{itemize}
The homotopy groups $\pi_*\tmf$ are given by tensoring the pattern depicted in Figure~\ref{fig:tmf2} with $\ZZ_{(2)}[\Delta^8]$, where $\Delta^8 \in \pi_{192}\tmf$.  Our choice of names for generators in Figure~\ref{fig:tmf2} is motivated by the fact that the elements
$$ \eta, \nu, \epsilon, \kappa, \bar{\kappa}, q, u, w $$
in the stable stems map to the corresponding elements in $\pi_*\tmf$ under the $\tmf$-Hurewicz homomorphism.  The other indecomposable multiplicative generators are named based on the names of elements which detect them in the $E_2$-term of the descent spectral sequence.  There is thus some ambiguity in the naming of some of these elements coming from the filtration associated to the descent spectral sequence.

For definiteness we fix $c_4 \in \pi_8\tmf$ to be the unique element detected by $c_4$ in the descent spectral sequence of Adams filtration $4$.
Note that the $c_4$-torsion in $\pi_*\tmf$ does not have $c_4$-exponent 1.  Indeed, on $c_4$-torsion classes, multiplication by $c_4$ is equal to multiplication by $\epsilon$ \cite[9.5]{BrunerRognes}, so, for example, $c_4\kappa = \epsilon\kappa \ne 0$.  However, all $c_4$-torsion has $c_4$-exponent 2 \cite[Prop.~6.1]{BHHM2}, \cite[9.5]{BrunerRognes}.

The main theorem of this paper is the following.

\begin{thm}\label{thm:main}
The $\tmf$-Hurewicz image is the subgroup of $\pi_*\tmf$ generated by 
\begin{enumerate}
\item All the elements of $\pi_{\le 3}(\tmf)$,
\item The elements $c_4^i \eta$ and $c_4^i \eta^2$,
\item All the elements of $\pi_*\tmf$ annihilated by a power of $c_4$, \emph{except} those in $\pi_{24k+3}\tmf$.
\end{enumerate}
\end{thm}

\begin{rmk}
The reader will note from Figure~\ref{fig:tmf2} that the subgroup of $\pi_*(\tmf)$ generated by the elements of type (3) above form a self dual pattern centered in dimension 85.  This is discussed in \cite[Ch.~10]{BrunerRognes}.
\end{rmk}

Besides representing an advance in our understanding of $v_2$-periodic homotopy at the prime 2, Theorem~\ref{thm:main} also has applications to smooth structures on spheres, as explained in \cite{BHHM2}.  Specifically, Hill, Hopkins, and the first two authors consider the following question.

\begin{question}\label{ques:exotic}
In which dimensions $n$ do there exist exotic smooth structures on the $n$-sphere?
\end{question}

Such spheres with exotic smooth structures are called exotic spheres.
The work of Kervaire and Milnor \cite{KervaireMilnor} relates the existence of exotic spheres to the triviality of the Kervaire homomorphism
$$ \pi_{4k+2}^s \rightarrow \ZZ/2 $$
and the non-triviality of the cokernel of the $J$-homomorphism
$$ J: \pi_n SO \rightarrow \pi^s_{n}. $$
Specifically, they prove that exotic spheres exist in dimensions $n$ for which
\begin{description}
\item[$n = 4k$] $n \ge 8$ and there exists a non-trivial element of coker $J$, 
\item[$n = 4k+1$] there exists a non-trivial element of coker $J$, or there does not exist an element of Kervaire invariant 1 in dimension $n+1$,
\item[$n = 4k+2$] there exists a non-trivial element of coker $J$ with Kervaire invariant $0$,
\item[$n = 4k+3$] $n \ge 7$.
\end{description}
Combining this with the work of Moise \cite{Moise}, Browder \cite{Browder}, Barratt-Jones-Mahowald \cite{BJM}, Hill-Hopkins-Ravenel \cite{HHR}, and Wang-Xu \cite{WangXu}, Question~\ref{ques:exotic} has been answered completely for $n$ odd:
\begin{quote}
the only odd dimensions $n$ for which there do not exist exotic spheres are $n = 1$, 3, 5, and 61.
\end{quote}
For $n$ even, the case of $n = 4$ is unresolved.  For other even $n$, by the previous discussion, the question boils down to the existence of non-trivial elements of coker $J$ (with Kervaire invariant 0).  It is shown in \cite{BHHM2}:
\begin{quote}
the only even dimensions $4 \ne n < 140$ for which there do not exist exotic spheres are $n = 2$, 6, 12, and 56. 
\end{quote}
In the case of $n = 8k+2 \ge 10$, Adams' elements $\mu_{8k+2}$ with non-trivial $\KO$-Hurewicz image are not in the image of $J$ and have trivial Kervaire invariant.  It thus follows that:
\begin{quote}
there exist exotic spheres in all dimensions $n = 8k+2 \ge 10$. 
\end{quote}
As is explained in \cite{BHHM2}, many of the 192-periodic families of elements of Theorem~\ref{thm:main} also are not in the image of $J$ and have trivial Kervaire invariant.  Theorem~\ref{thm:main} therefore has the following corollary.\footnote{In fact, the $v_2^{32}$-self map of Theorem~\ref{thm:v232} which is used to construct the periodic families of Theorem~\ref{thm:main} also immediately implies the existence of some elements not in the image of the $J$-homomorphism which are in the kernel of the $\tmf$-Hurewicz homomorphism, such as the beta elements $\beta_{32k/8}$. However, we will not concern ourselves here with the few additional dimensions such considerations add to the list of Corollary~\ref{cor:exotic}.}

\begin{cor}\label{cor:exotic}
There exist exotic spheres in the following congruence classes of even dimensions $n \ge 8$ modulo 192:
\begin{align*}
& 2, 6, 8, 10, 14, 18, 20, 22, 26, 28, 32, 34, 40, 42, 46, 50, 52, 54, 58, 60, 66, 68, \\ 
& 70, 74, 80, 82, 
 90, 98, 100, 102, 104, 106, 110, 114, 116, 118, 122, 124, 128, \\
 & 130, 136, 138, 142, 146, 148,
150, 154, 156, 162, 164, 170, 178, 186.
\end{align*}
(This accounts for over half of the even dimensions.)
\end{cor}

We will prove Theorem~\ref{thm:main} by first showing (Theorem~\ref{thm:nothi}) that the subgroup of $\pi_*\tmf$ described by Theorem~\ref{thm:main} is contained in the Hurewicz image.  This will be a relatively straightforward consequence of some $v_1$-periodic computations.  The elements of Theorem~\ref{thm:main}(1) are already established to be in the Hurewicz image by the preceding discussion, and the elements (2) are in the Hurewicz image because they are the images of the elements $\mu_{8i+j}$. 
We are left to show that the elements of type (3) lift to $\pi_*^s$.  This is the main task of this paper.

In \cite{BrunerRognes}, Bruner and Rognes give a systematic and careful study of the Adams spectral sequence for tmf, and in particular they have independently established the Hurewicz image in many low-dimensional cases.  Specifically, they prove Theorem~\ref{thm:main} for degrees $* \le 101$ and also show that $w\bar{\kappa}^3$, $w^2\bar{\kappa}$, $w\bar{\kappa}^4$, $2\Delta^4\kappa\bar{\kappa}$, and $4\Delta^6 \nu^2$ (in dimensions 105, 110, 125, 130, and 150) are in the Hurewicz image.  Also, they use a different technique (Anderson duality) to prove that the Hurewicz image is contained in the subgroup of $\tmf_*$ described in Theorem~\ref{thm:main}. 

Our strategy to lift elements from $\pi_*\tmf$ to $\pi_*^s$ is to use the methods of \cite{BHHM2}.  We summarize that strategy here.
We recall the following from \cite[Prop. 6.1]{BHHM2}.

\begin{prop}[\cite{BHHM2}]\label{prop:torsion}
Every $c_4$-torsion element $x \in \pi_{*}\tmf$ is $8$-torsion and $c_4^2$-torsion.
\end{prop}

Let $M(2^i)$ denote the cofiber of $2^i$, and let $M(2^i,v_1^j)$ denote the cofiber of a $v_1$-self map (see \cite[Prop.~2.3]{DavisMahowald})
$$ v_1^j: \Sigma^{2j}M(2^i) \rightarrow M(2^i). $$

\begin{cor}
Every $c_4$-torsion element 
$x \in  \pi_{*}(\tmf)$ lifts to an element 
$$ \td{x} \in \tmf_{*+18} M(8, v_1^8) $$
so that the projection to the top cell maps $\td{x}$ to $x$.
\end{cor}

Given a $c_4$-torsion element 
$x \in  \pi_{< 192}(\tmf)$, Proposition~\ref{prop:torsion} implies it lifts to an element 
$$ \td{x} \in \tmf_* M(8, v_1^8) $$
so that the projection to the top cell maps $\td{x}$ to $x$.  We will then show that $\td{x}$ lifts to an element 
$$ \td{y} \in \pi_*M(8, v_1^8). $$ 
Then the image 
$$ y \in \pi_*^s $$
given by projecting $\td{y}$ to the top cell is an element whose image under the $\tmf$-Hurewicz homomorphism is $x$.  

Every $c_4$-torsion element $x' \in \pi_{\ge 192} \tmf$ is of the form $v_2^{32k} x$ for $x \in \pi_{<192}\tmf$.  We will prove the following theorem.

\begin{thm}\label{thm:v232}
There exists a $v_2^{32}$-self map
$$ v_2^{32}: \Sigma^{192} M(8,v_1^8) \rightarrow M(8,v_1^8). $$
\end{thm}

If $\td{x} \in \tmf_*M(8,v_1^8)$ is a lift of $x$, and $\td{y} \in \pi_*M(8,v_1^8)$ is a lift of $\td{x}$, as in the discussion above, then the resulting element
$$ v_2^{32k} \td{y} \in \pi_*M(8,v_1^8), $$
obtained by composing with the $k$-fold iterate of the $v_2^{32}$-self map, projects to an element $y' \in \pi_*^s$ which maps to $x'$ under the $\tmf$-Hurewicz homomorphism. 

As in \cite{BHHM2}, the analysis above rests on a systematic analysis of the homotopy groups $\pi_*M(8,v_1^8)$.  This will be based on computations using the \emph{modified Adams spectral sequence (MASS)}.  The $E_2$-term of the modified Adams spectral sequence will be analyzed in a region near its vanishing line by means of another spectral sequence, the \emph{algebraic tmf-resolution}.  

The work of \cite{BHHM2} was hampered by the fact that all of the algebraic tmf-resolution computations were performed on the level of the $E_1$-term of the algebraic tmf-resolution.  In this paper, we will show that the weight spectral sequence, used in the context of bo-resolutions by \cite{LellmannMahowald} and \cite{boass}, can be used to analyze the $E_2$-term of the algebraic tmf-resolution, greatly simplifying the computations.  

\subsection*{Conventions}

\begin{itemize}
\item Homology will be implicitly taken with mod 2 coefficients.

\item
We let $A_*$ denote the dual Steenrod algebra, $A\mmod A(2)_*$ denote the dual of the Hopf algebra quotient $A\mmod A(2)$, and for an $A_*$-comodule $M$ (or more generally an object of the stable homotopy category of $A_*$-comodules \cite{Hovey}) we let
$$ \Ext^{s,t}_{A_*}(M) $$
denote the group $\Ext^{s,t}_{A_*}(\FF_2, M)$.  

\item Given a Hopf algebroid $(B, \Gamma)$, and a comodule $M$, we will let $C^*_\Gamma(M)$ denote the associated normalized cobar complex.

\item For a spectrum $E$, we let $E_*$ denote its homotopy groups $\pi_*E$.
\end{itemize}

\subsection*{Outline of paper}

In Section~\ref{sec:preliminaries}, we recall the modified Adams spectral sequence (MASS), which takes the form
$$ \E{mass}{}_2^{*,*} = \Ext_{A_*}(H_*X \otimes H(8,v_1^8)) \Rightarrow \pi_*(X \wedge M(8,v_1^8)) $$
for a certain object $H(8,v_1^8)$ in the stable homotopy category of $A_*$-comodules.
We recall how the $E_2$-term of the MASS can be studied using the algebraic 
tmf-resolution, which is a spectral sequence that takes the form
$$ \E{\tmf}{alg}_1(M)^{*,*,*} \Rightarrow \Ext^{*,*}_{A_*}(M) $$
for any $M$ in the stable category of $A_*$-comodules.
We then recall how the $E_1$-term of the algebraic tmf-resolution decomposes as a sum of Ext groups involving tensor powers of bo-Brown-Gitler comodules, and also summarize an inductive method to compute these Ext groups.

In Section~\ref{sec:algtmfres}, we study the $d_1$ differential in the algebraic tmf-resolution for $\FF_2$, and introduce a tool, the weight spectral sequence (WSS) 
$$ \E{\tmf}{alg}_1 = \E{wss}{}_0 \Rightarrow \E{\tmf}{alg}_2, $$
which serves as an analog of the May spectral sequence, and converges to the $E_2$-term of the algebraic tmf-resolution.  The $E_0$-page of the $v_0$-localized weight spectral sequence is identified with the cobar complex of a primitively generated Hopf algebra, and this allows us to give ``names'' to the $v_0$-torsion-free classes of $\E{\tmf}{alg}_1$.  We include many charts of summands of $\E{\tmf}{alg}_1(\FF_2)$ corresponding to tensor powers of bo-Brown-Gitler comodules which illustrate this naming convention, and provide the essential data for the rest of the computations in this paper.  Finally, we study the $g$-local WSS\footnote{Here, $g \in \Ext^{4,24}_{A_*}(\FF_2)$ is the element corresponding to the element $h^4_{2,1}$ in the May spectral sequence which detects $\bar{\kappa}$ in the Adams spectral sequence for the sphere.} using recent work of Bhattacharya-Bobkova-Thomas \cite{BBT}, and show that many classes are killed in the $g$-local WSS by $d_1$-differentials.  This is the key fact we will use to systematically remove obstructions for lifting classes from $\tmf_*X$ to $\pi_*X$.

In Section~\ref{sec:M38} we study the structure of the MASS for $M(8,v_1^8)$.  We recall the structure of the MASS for $\tmf_*M(8,v_1^8)$, and we explain how to adapt the Ext charts of Section~\ref{sec:algtmfres} to give the corresponding computations of $\E{\tmf}{alg}_1(H(8,v_1^8))$. We then explain how to translate the computations of the $g$-localized algebraic tmf-resolution of Section~\ref{sec:algtmfres} to the case of $H(8,v_1^8)$.

Section~\ref{sec:v232} is dedicated to the proof of Theorem~\ref{thm:v232}.  We recall the work of Davis, Mahowald, and Rezk, who discovered topological attaching maps between the first two bo-Brown-Gitler spectra which comprise $\tmf \wedge \tmf$, which give extra differentials in the Adams spectral sequence of $\tmf \wedge \tmf$ that kill some $g$-torsion-free classes.  We then prove a technical lemma (Lemma~\ref{lem:technical}) which lifts differentials from the MASS for $\tmf^{s} \wedge M(8,v_1^8)$ to the MASS for $M(8,v_1^8)$.  We prove Theorem~\ref{thm:v232} by listing all elements in $\E{\tmf}{alg}_1(H(8,v_1^8))$ which could detect a non-trivial differential $d_r(v_2^{32})$ in the MASS for $M(8,v_1^8)$, and then we systematically eliminate these possibilities.  Most of these classes are $g$-torsion-free, and are eliminated in the WSS, or by using Lemma~\ref{lem:technical}.

In Section~\ref{sec:J}, we explain how $v_1$-periodic computations give an upper bound on the Hurewicz image.

Section~\ref{sec:tmfhi} is devoted to showing this upper bound is sharp, by producing lifts of the remaining elements of $\pi_*\tmf$ to the sphere.  We begin by identifying multiplicative generators of the Hurewicz image in dimensions less than 192, so that it suffices for us to lift these.  We then lift these elements by producing elements in the MASS for $M(8,v_1^8)$ which we show are permanent cycles, and detect elements of $\pi_*M(8,v_1^8)$ which project to the desired elements on the top cell.  These elements are then propagated to $v_2^{32}$-periodic families using the self-map, thus proving Theorem~\ref{thm:main} in all dimensions. 

\subsection*{Acknowledgments}

We are grateful to Bob Bruner and John Rognes for generously sharing their results on their study of the Adams spectral sequence of tmf, and also to Rognes for pointing out a redundancy in Section~\ref{sec:tmfhi}.  This project would have not been possible without the Ext computational software developed by Bob Bruner and Amelia Perry, and the detailed computations of the Adams spectral sequence of the sphere by Isaksen, Wang, and Xu.  The authors are especially grateful to Bob Bruner for providing them with a module definition file for $\br{A \mmod A(2)}$.    The first author would also like to express his appreciation to Agn\`es Beaudry, Prasit Bhattacharya, Dominic Culver, Kyle Ormsby, Nat Stapleton, Vesna Stojanoska, and Zhouli Xu, whose previous collaborative work on the tmf-resolution was essential for the results of this paper, as well as to Mike Hill and Mike Hopkins, whose collaboration with the first two authors was the genesis of this paper. The third author also wishes to thank the first two authors for the opportunity to contribute to this project. Finally, the authors wish to thank anonymous referees for important comments and corrections. The first author was supported by NSF grants DMS-1050466, DMS-1452111, DMS-1547292, DMS-1611786, and DMS-2005476 over the course of this work. The third author was partially supported by NSF grant DMS-1547292. 

\section{Preliminaries}\label{sec:preliminaries}

The techniques and methods of this paper closely follow those of \cite{BHHM2}.  In this section we recall some spectral sequences used in that paper.

\subsection*{The modified Adams spectral sequence}

Our computations of $\pi_*M(8,v_1^8)$ and $\tmf_*M(8,v_1^8)$ will be performed using the modified Adams spectral sequence (MASS).  We refer the reader to \cite[Sec. 6]{BHHM2} for a complete account of the construction of the MASS and summarize the form it takes here.

Let $\St_{A_*}$
denote Hovey's stable homotopy category of $A_*$-comodules \cite{Hovey}.  
For objects $M$ and $N$
of $\St_{A_*}$, we define groups
$$ \Ext^{s,t}_{A_*}(M,N) = \St_{A_*}(\Sigma^t M, N[s]) $$
as a group of maps in the stable homotopy category.  Here $\Sigma^t M$ denotes the
$t$-fold shift with respect to the internal grading of $M$, and $N[s]$ denotes
the $s$-fold shift with respect to the triangulated structure of
$\St_{A_*}$.  This
reduces to the usual definition of $\Ext_{A_*}$ when $M$ and $N$ are
$A_*$-comodules.  

Define $H(8)$ to be the cofiber of the map
\begin{equation}\label{eq:H3}
 \Sigma^3 \FF_2[-3] \xrightarrow{h_0^3} \FF_2
 \end{equation}
in the stable homotopy category of $A_*$-comodules.
Define $H(8,v_1^8) \in \St_{A_*}$ to be the cofiber
\begin{equation}\label{eq:H38}
\Sigma^{24}H(8)[-8] \xrightarrow{v_1^8} H(8) \rightarrow H(8,v_1^8).
\end{equation}
For a spectrum $X$, the MASS takes the form
$$ \E{mass}{}_2^{s,t}(M(8,v_1^8) \wedge X) = \Ext^{s,t}_{A_*}(H(8,v_1^8)\otimes H_*X) \Rightarrow \pi_{t-s} M(8,v_1^8) \wedge X. $$

Recall the following from \cite[Prop. 7.1]{BHHM2}.

\begin{prop}\label{prop:ring}
$M(8,v_1^8)$ is a weak homotopy ring spectrum.\footnote{By this, we mean a spectrum with a possibly non-associative product and a two sided unit in the stable homotopy category.}
\end{prop}

It follows that if $X$ is a ring spectrum, the MASS above is a spectral sequence of (non-associative) algebras. 

We recall the following key theorem of Mathew.

\begin{thm}[Mathew \cite{Mathew}]
We have
$$ H_*\tmf \cong A\mmod A(2)_* $$
as an algebra in $A_*$-comodules.
\end{thm}

Taking $X = \tmf \wedge Y$ for some $Y$, and applying a change of rings theorem, the MASS takes the form
$$ \E{mass}{}_2^{s,t}(\tmf \wedge M(8,v_1^8) \wedge Y) = \Ext^{s,t}_{A(2)_*}(H(8,v_1^8) \otimes H_* Y) \Rightarrow \tmf_{t-s}(M(8,v_1^8) \wedge Y). $$

\subsection*{The algebraic $\tmf$-resolution}

The $E_2$-page of the MASS for $M(8,v_1^8)$ will be analyzed using an algebraic analog of the $\tmf$-resolution (as in \cite[Sec.~6]{BHHM2}).  

The (topological) $\tmf$-resolution of a space $X$ is the Adams spectral sequence based on the spectrum $\tmf$:
$$ \E{\tmf}{}_1^{s,t} = \pi_{t} \tmf \wedge \br{\tmf}^{s} \wedge X \Rightarrow \pi_{t-s} X. $$
Here, $\br{\tmf}$ is the cofiber of the unit
$$ S \rightarrow \tmf \rightarrow \br{\tmf} $$
and $\br{\tmf}^s = \br{\tmf}^{\wedge s}$ denotes its $s$-fold smash power. 

The \emph{algebraic} $\tmf$-resolution is an algebraic analog.
Namely, let $M$ be an object of the stable homotopy category of $A_*$-comodules, 
and let $\br{A\mmod A(2)}_*$ denote the cokernel of the unit
$$ 0 \rightarrow \FF_2 \rightarrow A\mmod A(2)_* \rightarrow \br{A \mmod A(2)}_* \rightarrow 0 $$
(note that $H_* \br{\tmf} = \br{A\mmod A(2)}_*$).
The \emph{algebraic $\tmf$-resolution} of $M$ is a spectral sequence of the form 
$$ \E{\tmf}{alg}_1^{s,t,n}(M) = \Ext^{s,t}_{A(2)_*}(\br{A\mmod A(2)}_*^{\otimes n} \otimes M) \Rightarrow \Ext^{s+n,t}_{A_*}(M). $$

\subsection*{$\bo$-Brown-Gitler comodules}\label{sec:boi}

We recall some material on bo-Brown-Gitler comodules.  These are $A_*$-comodules which are the homology of the bo-Brown-Gitler spectra constructed by \cite{GoerssJonesMahowald}.
Mahowald used integral Brown-Gitler spectra to analyze the bo-resolution \cite{Mahowaldbo}.  The bo-Brown-Gitler comodules play a similar role in the algebraic tmf-resolution \cite{BHHM}, \cite{MahowaldRezk}, \cite{DavisMahowald},  \cite{BOSS}, \cite{BHHM2}.  

Endow the mod $2$ homology of the connective real $K$-theory spectrum 
$$ H_*(\bo) \cong A\mmod A(1)_* = \FF_2[\xib_1^4, \xib_2^2, \xib_3, \ldots] $$
with a multiplicative grading by declaring the \emph{weight} of $\xib_i$ to be 
\begin{equation}\label{eq:weight}
wt(\xib_i) = 2^{i-1}. 
\end{equation}
The $i$th \emph{$\bo$-Brown-Gitler} comodule is the subcomodule
$$ \bou_i = F_{4i}A\mmod A(1)_* \subset A \mmod A(1)_* $$
spanned by monomials of weight less than or equal to $4i$.  It is isomorphic as an $A_*$-comodule to the homology of the $i$th bo-Brown-Gitler spectrum $\bo_i$.

The analysis of the $E_1$-page of the algebraic $\tmf$-resolution is simplified via the decomposition of $A(2)_*$-comodules
$$ \br{A \mmod A(2)}_* \cong \bigoplus_{i > 0} \Sigma^{8i} \bou_i $$
of \cite[Cor.~5.5]{BHHM}.
We therefore have a decomposition of the $E_1$-page of the algebraic $\tmf$-resolution for $M$ given by
\begin{equation}\label{eq:E1decomp}
\E{\tmf}{alg}_1^{s,t,n}(M) \cong \bigoplus_{i_1, \ldots, i_n > 0}\Ext_{A(2)_*}^{s,t}(\Sigma^{8(i_1+\cdots+i_n)}\bou_{i_1} \otimes \cdots \otimes \bou_{i_n} \otimes M).
\end{equation}

For any $M$, the computation of
$$ \Ext_{A(2)_*}^{s,t}(\Sigma^{8(i_1+\cdots+i_n)}\bou_{i_1} \otimes \cdots \otimes \bou_{i_n} \otimes M) $$
can be inductively determined from $\Ext_{A(2)_*}(\bou_1^{\otimes k} \otimes M)$ by means of a set of exact sequences of $A(2)_*$-comodules which relate the $\bou_i$'s \cite[Sec.~7]{BHHM} (see also \cite{BOSS}):
\begin{gather}
0\to \Sigma^{8j} \ul{\bo}_j \to \ul{\bo}_{2j}\to A(2)\mmod A(1)_* \otimes \ul{\tmf}_{j-1}  \to \Sigma^{8j+9} \ul{\bo}_{j-1} \to 0 \label{eq:boSES1},
\\
0 \to \Sigma^{8j} \ul{\bo}_j \otimes \ul{\bo}_1 \to \ul{\bo}_{2j+1}\to A(2)\mmod A(1)_* \otimes \ul{\tmf}_{j-1} \to 0 \label{eq:boSES2}
\end{gather}
Here, $\ul{\tmf}_j$ is the $j$th $\tmf$-Brown-Gitler comodule --- it is the subcomodule of 
$$ H_*(\tmf) \cong A\mmod A(2)_* = \FF_2[\xib_1^8, \xib_2^4, \xib_3^2, \xib_4, \ldots] $$
spanned by monomials of weight less than or equal to $8j$.\footnote{
Technically speaking, as is addressed in \cite[Sec.~7]{BHHM}, the comodules $A(2)\mmod A(1)_* \otimes \ul{\tmf}_{j-1}$ in the above exact sequences have to be given a slightly different $A(2)_*$-comodule structure from the standard one arising from the tensor product.  However, this different comodule structure ends up being $\Ext$-isomorphic to the standard one.  As we are only interested in Ext groups, the reader can safely ignore this subtlety.}

The exact sequences (\ref{eq:boSES1}) and (\ref{eq:boSES2}) can be re-expressed as resolutions in the stable homotopy category of $A(2)_*$-comodules:
\begin{gather*}
\ul{\bo}_{2j}\to A(2)\mmod A(1)_* \otimes \ul{\tmf}_{j-1}  \to \Sigma^{8j+9} \ul{\bo}_{j-1} \to \Sigma^{8j} \ul{\bo}_j[2],
\\
\ul{\bo}_{2j+1}\to A(2)\mmod A(1)_* \otimes \ul{\tmf}_{j-1} \to \Sigma^{8j} \ul{\bo}_j \otimes \ul{\bo}_1[1] 
\end{gather*}
which give rise to spectral sequences
\begin{equation}\label{eq:boIss}
\begin{split}
E_1^{n,s,t} = 
\left\{
\begin{array}{ll}
\Ext^{s,t}_{A(1)_*}(\ul{\tmf}_{j-1} \otimes M), & n = 0, \\
\Ext^{s,t}_{A(2)_*}(\Sigma^{8j+9}\bou_{j-1} \otimes M[-1]), & n = 1, \\ 
\Ext^{s,t}_{A(2)_*}(\Sigma^{8j}\bou_j \otimes M), & n = 2,  \\
0, & n > 2 \\ 
\end{array}
\right\}
\Rightarrow \Ext^{s,t}_{A(2)_*}(\bou_{2j} \otimes M), \\
\\
E_1^{n,s,t} = \left\{
\begin{array}{ll}
\Ext^{s,t}_{A(1)_*}(\ul{\tmf}_{j-1} \otimes M), & n = 0, \\
\Ext^{s,t}_{A(2)_*}(\Sigma^{8j}\bou_j \otimes \bou_1 \otimes M), & n = 1, \\
0, & n > 1 \\ 
\end{array}
\right\}
\Rightarrow \Ext^{s,t}_{A(2)_*}(\bou_{2j+1} \otimes M). \\
\end{split}
\end{equation}
These spectral sequences have been observed to collapse in low degrees (see \cite{BOSS}) but it is not known if they collapse in general.
They inductively build $\Ext_{A(2)_*}(\bou_i \otimes M)$ out of $\Ext_{A(2)_*}(\bou_1^{\otimes k} \otimes M)$ and $\Ext_{A(1)_*}(\ul{\tmf}_j \otimes M)$.

%

\section{Analysis of the algebraic tmf-resolution}\label{sec:algtmfres}

In this section we will compute the $d_1$-differential in the algebraic tmf-resolution, and will introduce a tool, the \emph{weight spectral sequence (WSS)}, which is a variant of the May spectral sequence that converges to the $E_2$-page of the algebraic tmf-resolution.

\subsection*{The $d_1$ differential in the algebraic tmf-resolution}

Our approach to understanding the $d_1$-differential in the algebraic $\tmf$-resolution will be to compute it on $v_0$-torsion-free classes, and then infer its effect on $v_0$-torsion classes by means of linearity over $\Ext_{A_*}(\FF_2).$

Consider the algebraic $BP\bra{2}$ and algebraic $BP$-resolutions.  
\begin{gather*}
\E{BP\bra{2}}{alg}^{s,t,n} = \Ext^{s,t}_{E[2]_*}(\br{A\mmod E[2]}_*^{\otimes n}) \Rightarrow \Ext^{s+n,t}_{A_*}(\FF_2) \\
\E{BP}{alg}^{s,t,n} = \Ext^{s,t}_{E_*}(\br{A\mmod E}_*^{\otimes n}) \Rightarrow \Ext^{s+n,t}_{A_*}(\FF_2)
\end{gather*}
Here, $E[2] = E[Q_0, Q_1, Q_2]$ and $E = E[Q_0, Q_1, Q_2, \cdots]$ denote subalgebras of the Steenrod algebra, where $Q_i$ are the Milnor generators dual to $\xi_{i+1} \in A_*$.

The $d_1$-differential in the algebraic $\tmf$-resolution may be studied by means of the zig-zag
\begin{equation}\label{eq:zigzag}
\E{\tmf}{alg}_1^{*,*,*} \rightarrow \E{BP\bra{2}}{alg}_1^{*,*,*} \leftarrow 
\E{BP}{alg}_1^{*,*,*}.
\end{equation}
Note that 
$$ \E{BP}{alg}_1^{*,*,n} \cong \FF_2[v_0, v_1, v_2, \cdots ]\otimes \br{\FF_2[\xib_1^2, \xib^2_2, \cdots]}^{\otimes n} $$
where 
$\br{\FF_2[\xib_1^2, \xib_2^2, \cdots ]}$ denotes the cokernel of the unit
$$ \FF_2 \rightarrow \FF_2[\xib_1^2, \xib_2^2, \cdots ]. $$
The Adams spectral sequences
$$ \E{BP}{alg}_1^{n,*,*} = \E{ass}{*,*}_2(BP \wedge \br{BP}^{n}) \Rightarrow C^n_{BP_*BP}(BP_*) $$
collapse, where $C^*_{BP_*BP}$ is the normalized cobar complex for $BP_*BP$, and 
$$ \text{$\xib_i^2 \in \br{A\mmod E}_*$ detects $t_i \in BP_*BP$.} $$ 
We conclude:

\begin{lem}
The $d_1$ differential in the algebraic $BP$-resolution is the associated graded of the differential in the cobar complex for $BP_*BP$ with respect to Adams filtration.
\end{lem}

\subsection*{The weight spectral sequence}

Endow the normalized cobar complex 
$$ C^*(A_*,A_*,\FF_2) $$
with a decreasing filtration by weight by defining 
$$ wt(a_0[a_1|\cdots|a_s]) = wt(a_1)+\cdots+wt(a_s). $$
Applying $\Ext_{A_*}(\FF_2,-)$ to the resulting filtered $A_*$-comodule produces a variant of the May spectral sequence which we will call the \emph{modified May spectral sequence} (MMSS)\footnote{The authors of \cite{LSWX} construct a similar modified May spectral sequence, but with a slightly different filtration.}
\begin{equation}\label{eq:mmss}
 \E{mmss}{}^{w,s,t}_0 = C^*_{E^0 A_*}(\FF_2) \Rightarrow \Ext^{s,t}_{A_*}(\FF_2). 
 \end{equation}
Since $E^0A_*$ is primitively generated, we have
$$ \E{mmss}{}_1^{*,*} = \FF_2[h_{i,j} \: : \: i \ge 1, \: j \ge 0]. $$

The map $\tmf \rightarrow H$ induces an inclusion
$$ \Phi: H_*(\tmf \wedge \br{\tmf}^n) \hookrightarrow H_*(H \wedge \br{H}^n) \cong C^n(A_*,A_*,\FF_2). $$ 
Under this inclusion, the weight filtration restricts to a decreasing filtration on
$$ H_*(\tmf \wedge \br{\tmf}^n) \cong A\mmod A(2)_* \otimes \br{A \mmod A(2)}_*^{\otimes n}  $$
by $A_*$-subcomodules.
Because the weights of all of the generators of $A \mmod A(2)_*$ are divisible by 8, we actually work with weights divided by 8.  
Applying $\Ext_{A(2)_*}(\FF_2, -)$ and taking cohomology, we get 
the \emph{weight spectral sequence} (WSS):
$$ \E{wss}{}_0^{w,n,s,t} = \bigoplus_{i_1 + \cdots + i_n = w} \Ext_{A(2)_*}^{s,t}(\bou_{i_1} \otimes \cdots \otimes \bou_{i_n}) \Rightarrow \E{\tmf}{alg}^{n,s,t}_2. $$
The WSS serves as an analog of the May spectral sequence for the algebraic $\tmf$-resolution.

The map $\Phi$ above induces a map of spectral sequences
\begin{equation}\label{eq:Phi}
\xymatrix{
\E{wss}{}^{w,n,0,t}_0 \ar@{=>}[r] \ar[d]_{\Phi_*} & \E{tmf}{alg}_0^{n,0,t} 
\ar[d]^{\Phi_*} \\
\E{mmss}{}^{8w,n,t}_0 \ar[r] & \Ext^{n,t}_{A_*}(\FF_2)
}
\end{equation}

\subsection*{The $v_0$-localized algebraic $\tmf$-resolution}

Observe that we have
\begin{equation}\label{v0localExtA2}
v_0^{-1}\Ext_{A(2)_*}(\FF_2) = \FF_2[v_0^{\pm}, v_1^4, v_2^2].
\end{equation}
Note that $c_4, c_6 \in (\tmf_*)_\QQ$ are detected in the $v_0$-localized ASS by $v_1^4$ and $v_0^3v_2^2$, respectively.

We recall from \cite{BOSS} that
\begin{equation}\label{eq:v0localtmftmf}
v_0^{-1} \Ext_{A(2)_*}^{*,*}(A \mmod A(2)_*) = \FF_2[v_0^{\pm}, v_1^4, v_2^2][\xib_1^8, \xib^4_2]
\end{equation}
and that there is an isomorphism
\begin{equation}\label{eq:v0localboi}
 v_0^{-1}\Ext_{A(2)_*}(\bou_i) \cong \FF_2[v_0^{\pm}, v_1^4, v_2^2]\{\xib_1^{8i'}\xib_2^{4 i''} \}_{i = i' + i''}.
\end{equation}

We will now compute the localized $E_1$-page $v_0^{-1} \E{wss}{}_1$.
The following is immediate from the computation of the cobar differential (modulo terms of higher Adams filtration) on the elements $\xib_1^8$ and $\xib_2^4$, using (\ref{eq:v0localtmftmf}), (\ref{eq:v0localboi}), and (\ref{eq:zigzag}).

\begin{prop}
There is an isomorphism of differential graded algebras
$$ v_0^{-1}\E{wss}{}^{*,n,*,*}_0 \cong \FF_2[v_0^\pm, v_1^4, v_2^2]\otimes C^n_{\FF_2[\xib^8_1, \xib_2^4]} $$
where $\FF_2[\xib_1^8,\xib_2^4]$ is regarded as a primitively generated Hopf algebra.
\end{prop}

\begin{cor}
There is an isomorphism
$$ v_0^{-1}\E{wss}{}_1 = \FF_2[v_0^\pm, v_1^4, v_2^2]\otimes \FF_2[h_{1,3}, h_{1,4}, \ldots, h_{2,2}, h_{2,3}, \ldots ] $$ 
\end{cor}

\subsection*{Charts}

For the convenience of the reader we include some charts of $\Ext_{A(2)_*}(\bou_1^k)$ for $0 \le k \le 3$ as well as $\Ext_{A(2)_*}(\bou_2)$.

\begin{figure}
\includegraphics[angle = 90, origin=c, height =.7\textheight]{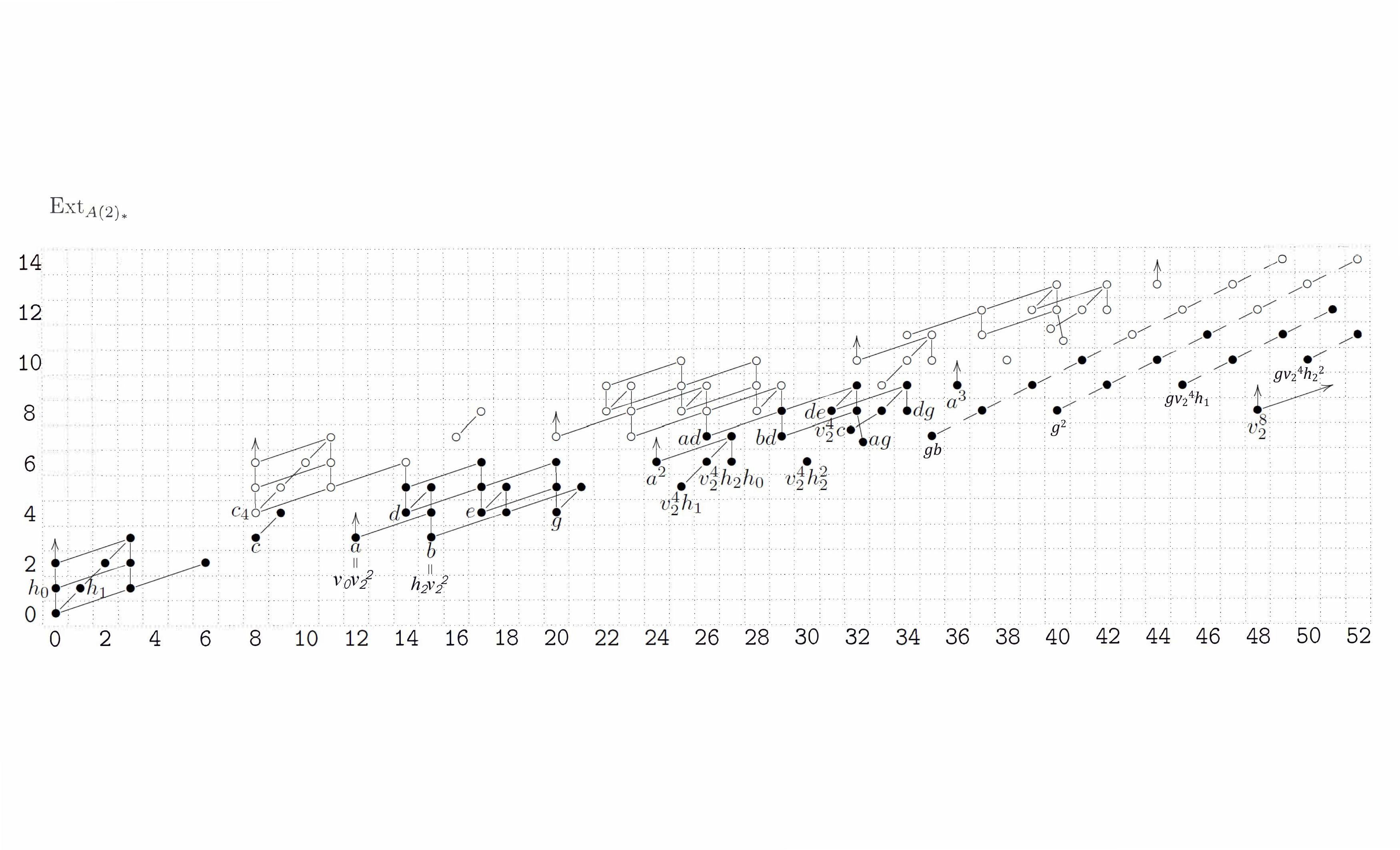}
\caption{$\Ext_{A(2)_*}(\FF_2)$.}\label{fig:ExtA2}
\end{figure}

\begin{figure}
\includegraphics[angle = 90, origin=c, height =.7\textheight]{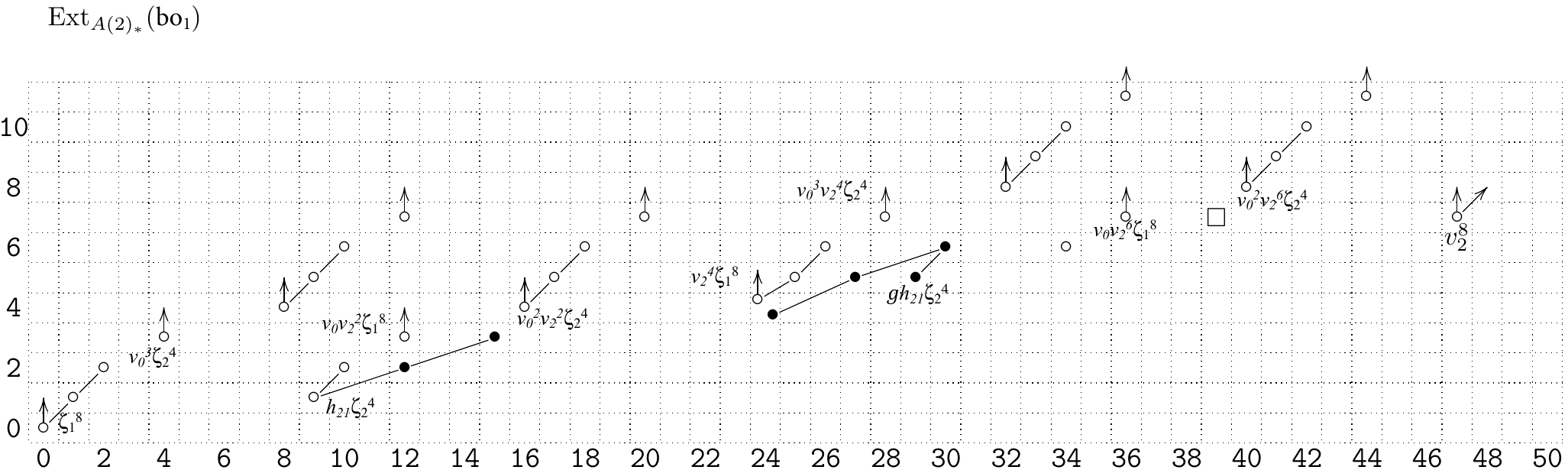}
\caption{$\Ext_{A(2)_*}(\bou_1)$.}\label{fig:bo1}
\end{figure}

\begin{figure}
\includegraphics[angle = 90, origin=c, height =.7\textheight]{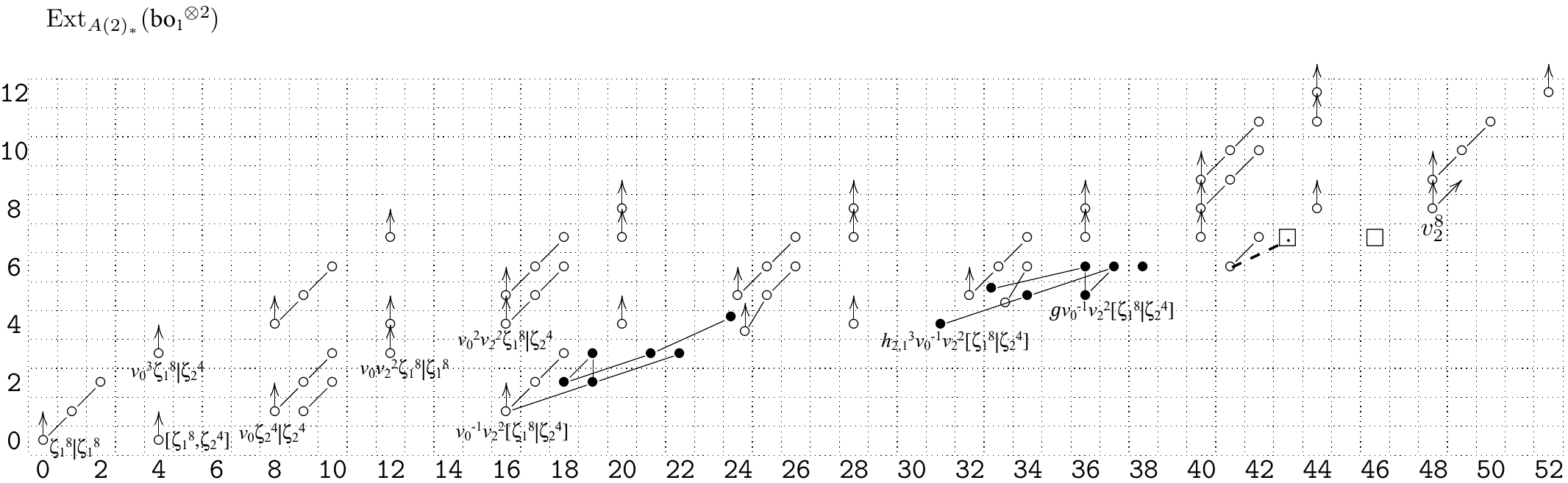}
\caption{$\Ext_{A(2)_*}(\bou_1^{\otimes 2})$.}\label{fig:bo1bo1}
\end{figure}

\begin{figure}
\includegraphics[angle = 90, origin=c, height =.7\textheight]{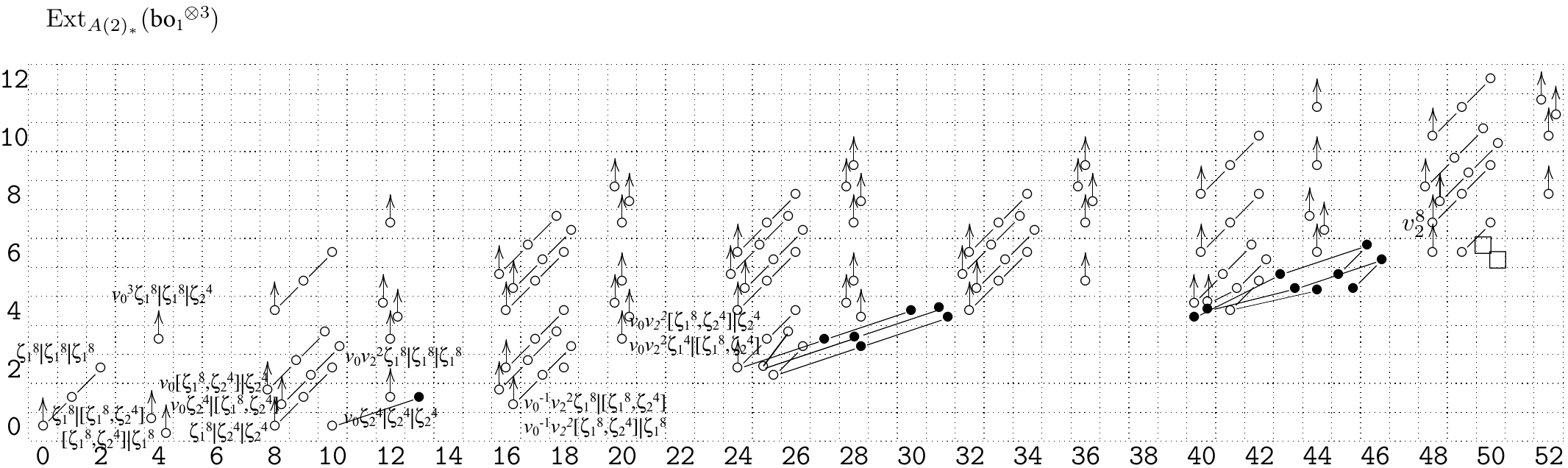}
\caption{$\Ext_{A(2)_*}(\bou_1^{\otimes 3})$.}\label{fig:bo1bo1bo1}
\end{figure}

\begin{figure}
\includegraphics[angle = 0, origin=c, height =\textheight]{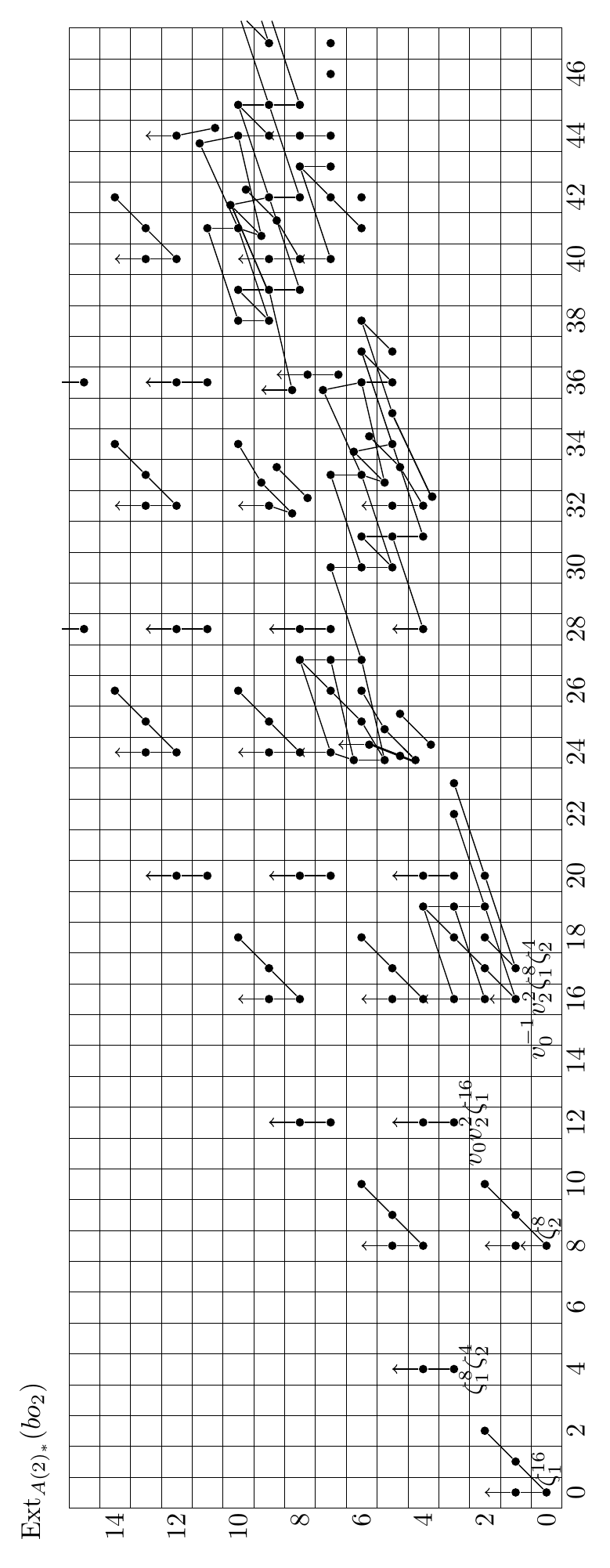}
\caption{$\Ext_{A(2)_*}(\bou_2)$.}\label{fig:bo2}
\end{figure}

$\Ext_{A(2)_*}(\FF_2):$ (Figure \ref{fig:ExtA2})

All of the elements are
$c_4 = v_1^4$-periodic, and $v_2^8$-periodic.  
Exactly one $v_1^4$ multiple
of each element is displayed with the $\bullet$ replaced by a $\circ$. 
Observe the wedge pattern beginning in $t-s = 35$.
This pattern is infinite, propagated horizontally by $h_{2,1}$-multiplication 
and vertically by $v_1$-multiplication.  Here, $h_{2,1}$ is the name of the
generator in the May spectral sequence of bidegree $(t-s,s) = (5,1)$, and
$h_{2,1}^4 = g$.

$\Ext_{A(2)_*}(\bou_1^{\otimes k}),\text{ for }k = 1,2,3:$ (Figures \ref{fig:bo1}, \ref{fig:bo1bo1}, \ref{fig:bo1bo1bo1})

Every element is $v_2^8$-periodic.
However, unlike $\Ext_{A(2)_*}(\FF_2)$, not every element of these Ext groups is
$v_1^4$-periodic.  Rather, it is the case that either an element $x \in
\Ext_{A(2)_*}(\bou_1^{\otimes k})$ satisfies 
$v_1^4x = 0$, or it is $v_1^4$-periodic.
Each of the $v_1^4$-periodic 
elements fit into families which look like shifted and truncated copies of
$\Ext_{A(1)_*}(\FF_2)$, and are labeled with a $\circ$.  
We have only
included the beginning of these $v_1^4$-periodic patterns in the chart.  
The other
generators are labeled with a $\bullet$.  A $\Box$ indicates a
polynomial algebra $\FF_2[h_{2,1}]$.

$\Ext_{A(2)_*}(\bou_2):$ (Figure~\ref{fig:bo2})

Via the spectral sequence (\ref{eq:boIss}), this Ext chart is assembled out of 
$\Ext_{A(1)_*}(\FF_2)$, $\Ext_{A(2)_*}(\Sigma^8\bou_1)$, and $\Ext_{A(2)_*}(\Sigma^{17}\FF_2[-1])$.

\subsection*{$\pmb{h_{2,1}}$-towers}

Our computations of the MASS for $M(8,v_1^8)$ will rely on a detailed understanding of this spectral sequence near its vanishing line.  Since $M(8,v_1^8)$ is a type 2 complex, the Hopkins-Smith Periodicity Theorem \cite{HopkinsSmith} implies that the $E_\infty$-page of this MASS has a vanishing line of slope $1/\abs{v_2} = 1/6$.  However, $g = h_{2,1}^4$ is not nilpotent in the modified Ext groups $\Ext_{A_*}(H(8,v_1^8))$, and $h_{2,1}$-multiplication has slope $1/5$.  The goal of this subsection is to show that many of the $h_{2,1}$-towers in the $E_1$-page of the algebraic tmf-resolution actually kill each other off by the $E_2$-page of the algebraic tmf-resolution.  We will then identify specific $h_{2,1}$-periodic elements of $\Ext_{A_*}(\FF_2)$ that some of these remaining $h_{2,1}$-towers detect.

Consider the quotient Hopf algebra $C_* := \FF_2[\zeta_2]/(\zeta_2^4)$ of $A(2)_*$, with
$$ \Ext^{*,*}_{C_*}(\FF_2) = \FF_2[v_1, h_{2,1}]. $$

\begin{lem}\label{lem:glocal}
Let $C(v_2^8)$ be the cofiber of the map
$$ v_2^8: \Sigma^{56}\FF_2[-8] \rightarrow \FF_2 $$
in the stable homotopy category $\St_{A(2)_*}$.
For any $M \in \St_{A(2)_*}$ there is an isomorphism 
$$ g^{-1}\Ext_{A(2)_*}(M \otimes C(v_2^8)) \cong h^{-1}_{2,1}\Ext_{C_*}(M). $$  
\end{lem}

\begin{proof}
Since the element $v_2^8 \in \Ext_{A(2)_*}(\FF_2)$ maps to zero in $\Ext_{C_*}(\FF_2)$, it follows that there is a factorization
$$
\xymatrix{
\FF_2 \ar[r] \ar[d] &
A(2)\mmod C_* \\
C(v_2^8) \ar@{.>}[ur]
}
$$
in $\St_{A(2)_*}$.
Explicit computation reveals
$$ g^{-1}\Ext_{A(2)_*}(\FF_2) = \FF_2[v_2^8,v_1, h^{\pm}_{2,1}] $$
and it follows that the map
$$ g^{-1} C(v_2^8) \rightarrow g^{-1} A(2)\mmod C_* $$
induces an isomorphism on $\Ext_{A(2)_*}$, and is hence an equivalence.  The result follows.
\end{proof}

\begin{cor}\label{cor:P21ss}
For any $M \in \St_{A(2)_*}$, there is a $v_2^8$-Bockstein spectral sequence
$$ h^{-1}_{2,1}\Ext_{C_*}(M)\otimes \FF_2[v_2^8] \Rightarrow g^{-1}\Ext_{A(2)_*}(M). $$
\end{cor}

Bhattacharya, Bobkova, and Thomas \cite{BBT} computed the $P_2^1$-Margolis homology of the tmf-resolution, and in the process computed the structure of $A \mmod A(2)^{\otimes n}_*$ as $C_*$-comodules.  From this one can read off the Ext groups
$$ h_{2,1}^{-1}\Ext_{C_*}(A \mmod A(2)^{\otimes n}_*), $$
which in turn determines the $g$-local algebraic $\tmf$-resolution by Corollary~\ref{cor:P21ss} (the spectral sequence in this corollary will collapse in the cases we consider it).

To state the results of \cite{BBT} we will need to introduce some notation.
The coaction of $\FF_2[\zeta_2]/\zeta_2^4$ is encoded in the dual action of the algebra $E[Q_1, P_2^1]$ on $A\mmod A(2)^{\otimes n}_*$.
Define elements
\begin{align*}
x_{i,j} & = 1\otimes  \cdots \otimes 1 \otimes \underbrace{\zeta_{i+3}}_{j} \otimes 1 \otimes \cdots \otimes 1,  \\
t_{i,j} & = 1\otimes \cdots \otimes 1 \otimes \underbrace{\zeta^{4}_{i+1}}_{j} \otimes 1 \otimes \cdots \otimes 1
\end{align*}
in $A\mmod A(2)_*^{\otimes n}$.  
The weight filtration on $A \mmod A(2)_*$ induces a \emph{multi-weight filtration} on $A \mmod A(2)^{\otimes n}_*$ indexed by $n$-tuples of weights.
The generators $x_{i,j}$ and $t_{i,j}$ have multi-weight
$$ (0, \ldots ,0, \underbrace{2^{i+2}}_{j}, 0, \ldots, 0). $$
For sets of multi-indices
\begin{align*} 
I & = \{ (i_1, j_1), \ldots, (i_k, j_k) \}, \\
I' & = \{ (i'_1, j'_1), \ldots, (i'_{k'}, j'_{k'}) \}
\end{align*}
with $I \cap I' = \emptyset$,
let 
$$ x_I t_{I'} \in A\mmod A(2)_* $$
denote the corresponding monomial. 
The action of the algebra $E[Q_1, P_2^1]$ on the $\FF_2$-submodule of $A\mmod A(2)_*^{\otimes n}$ spanned by such monomials is given by
\begin{align*}
Q_1(x_{I}t_{I'}) & = \sum_\ell x_{I-\{(i_\ell, j_\ell)\}} t_{I' \cup \{(i_\ell, j_\ell)\}}, \\
P_2^1(x_I t_{I'}) & = \sum_{\ell < \ell'} x_{I-\{(i_\ell,j_\ell),(i_{\ell'}, j_{\ell'})\}} t_{I' \cup \{(i_\ell,j_\ell),(i_{\ell'}, j_{\ell'})\}}. 
\end{align*}
For an \emph{ordered} set 
$$ J = ((i_1, j_1), \ldots, (i_k, j_k)) $$  
of multi-indices, let
$$ \abs{J} := k $$
denote the number of pairs of indices it contains.
Define linearly independent sets of elements
$$ \mc{T}_J \subset A \mmod A(2)^{\otimes n}_* $$
inductively as follows.  Define
$$ \mc{T}_{(i,j)} = \{x_{i,j}\}. $$
For $J$ as above with $\abs{J}$ odd, define
\begin{align*}
\mc{T}_{J, (i,j)} & = \{ z \cdot x_{i,j} \}_{z \in \mc{T}_J}, \\
\mc{T}_{J, (i,j), (i',j')} & = \{ Q_1(z \cdot x_{i,j})x_{i',j'} \}_{z \in \mc{T}_J} \cup \{ Q_1(z \cdot x_{i',j'})x_{i,j} \}_{z \in \mc{T}_J}.
\end{align*} 
Let 
$$ N_J \subset A\mmod A(2)_*^{\otimes n} $$ 
denote the $\FF_2$-subspace with basis
$$ Q_1 \mc{T}_J := \{ Q_1(z) \}_{z \in \mc{T}_J}. $$
While the set $\mc{T}_J$ depends on the ordering of $J$, the subspace $N_J$ does not.

Finally, for a set of pairs of indices
$$ J = \{(i_1,j_1), \cdots , (i_k,j_k)\} $$ 
as before, define
$$ x_Jt_J := x_{i_1,j_1}t_{i_1,j_1} \cdot \cdots \cdot x_{i_k,j_k}t_{i_k,j_k}. $$

The following is the main theorem of \cite{BBT}\footnote{The main theorem of \cite{BBT} is a computation of $P_2^1$-Margolis homology, but the actual content of the paper is a decomposition of $A\mmod A(2)_*$ in the stable module category of $E[Q_1, P_2^1]$.} 

\begin{thm}[Bhattacharya-Bobkova-Thomas]\label{thm:BBT}
As modules over $\FF_2[h^\pm_{2,1}, v_1]$, we have
\begin{multline*} 
h_{2,1}^{-1}\Ext^{*,*}_{E[Q_1,P_2^1]}(A \mmod A(2)_*^{\otimes n}) = \\
\FF_2[h_{2,1}^{\pm}] \otimes \bigg( \FF_2[v_1]\{x_{J'}t_{J'}\}_{J'} \oplus 
\bigoplus_{\abs{J} \: \mr{odd}} N_J\{x_{J'}t_{J'}\}_{J \cap J' = \emptyset} 
\\
\oplus \bigoplus_{\abs{J} \ne 0 \: \mr{even}}
\FF_2[v_1]/v_1^2 \otimes N_J\{x_{J'}t_{J'}\}_{J \cap J' = \emptyset} \bigg)
\end{multline*}
where $J$ and $J'$ range over the subsets of
$$ \{ (i,j) \: : \: 1 \le i, 1 \le j \le n \} $$ 
and $v_1$ acts trivially on $N_J$ for $\abs{J}$ odd.  The summand 
$$ h_{2,1}^{-1}\Ext^{*,*}_{E[Q_1,P_2^1]}(\bou_{i_1} \otimes \cdots \otimes \bou_{i_n}) $$
is spanned by those monomials of multiweight $(8i_1, \ldots, 8i_n)$.
\end{thm}

In light of Lemma~\ref{lem:glocal} and Corollary~\ref{cor:P21ss}, we may refer to elements of the $g$-local algebraic tmf-resolution as $v_2^{8j}z$, where $z$ is an element of the $h_{2,1}$-localized Ext groups described in the theorem above.

\begin{lem}\label{lem:dxt}
The WSS $d_0$-differential on the element
$$ x_{1,1}t_{1,1} \in g^{-1}\Ext^{*,*}_{A(2)_*}(\bou_2) $$
is given by
$$ d^{wss}_0(x_{1,1}t_{1,1}) = Q_1(x_{1,1}x_{1,2}) \in \Ext_{A(2)_*}(\bou_1^{\otimes 2}). $$
\end{lem}

\begin{proof}
We use the map of spectral sequences
$$ \E{wss}{}_0 \rightarrow g^{-1} \E{wss}{}_0. $$
By explicit computation of $g^{-1}\Ext_{A(2)_*}(\bou_2)$, under the map
$$ \Ext_{A(2)_*}(\bou_2) \rightarrow g^{-1} \Ext_{A(2)_*}(\bou_2) $$
we have
$$ v_0^{-1}v_2^2 \zeta_1^8\zeta_2^4 \mapsto h_{2,1}x_{1,1}t_{1,1}. $$
In the WSS we have 
\begin{equation}\label{eq:dzeta18zeta24}
 d_0^{wss}(v_0^{-1}v_2^2 \zeta_1^8\zeta_2^4) = v_0^{-1}v_2^2 [\zeta_1^8,\zeta_2^4]. 
 \end{equation}
Again, by explicit computation of $g$-local Ext groups, under the map
$$ \Ext_{A(2)_*}(\bou_1^{\otimes 2}) \rightarrow g^{-1} \Ext_{A(2)_*}(\bou_1^{\otimes 2}) $$
we have
$$ v_0^{-1}v_2^2 [\zeta_1^8,\zeta_2^4] \mapsto h_{2,1}Q_1(x_{1,1}x_{1,2}). $$
The result follows.
\end{proof}

\begin{prop}\label{prop:h21death}
In $g^{-1}\E{wss}{}_0$, all of the $h_{2,1}$-towers coming from $\Ext_{A(2)_*}(\bou_1^{\otimes k})$, for $k \ge 2$, either support non-trivial $d_0$-differentials, or are the target of $d_0$-differentials.
\end{prop}

\begin{proof}
By Lemma~\ref{lem:glocal} and Theorem~\ref{thm:BBT}, the $h_{2,1}$-towers coming from 
$$\Ext_{A(2)_*}(\bou_1^{\otimes k})$$ 
are supported by the elements
$\mc{T}_{\{(1,1), \ldots, (1,k)\}}$.
By Lemma~\ref{lem:dxt}, the WSS $d_0$ induces a surjection for $k = 2$
$$ d_0^{wss}: \FF_2[h_{2,1}^{\pm}, v_1, v_2^8]\{x_{1,1}t_{1,1}\} \twoheadrightarrow \FF_2[h_{2,1}^{\pm}, v_1, v_2^8]/v_1^2 \otimes N_{\{(1,1), (1,2)\}}. $$
For $k > 2$, observe that
$$ \mc{T}_{(1,1), \ldots, (1,k)} = Q_1(x_{1,1}x_{1,2})\mc{T}_{(1,3), \ldots, (1,k)} \cup Q_1(x_{1,2}x_{1,3}) \mc{T}_{\{(1,1), (1,4), \ldots, (1,k) \}}. $$
For $k>2$ even the WSS $d_0$ gives isomorphisms
\begin{multline*}
d_0^{wss}: \FF_2[h_{2,1}^{\pm}, v_1, v_2^8]/v_1^2 \otimes x_{1,1}t_{1,1}N_{\{(1,2), \ldots, (1,k-1)\}} \\
\xrightarrow{\cong} \FF_2[h_{2,1}^{\pm}, v_1, v_2^8]/v_1^2 \otimes Q_1(x_{1,1}x_{1,2})N_{\{(1,3), \ldots, (1,k)\}},
\end{multline*}
\begin{multline*}
d_0^{wss}: \FF_2[h_{2,1}^{\pm}, v_1, v_2^8]/v_1^2 \otimes x_{1,2}t_{1,2}N_{\{(1,1),(1,3), \ldots, (1,k-1)\}} \\
\xrightarrow{\cong} \FF_2[h_{2,1}^{\pm}, v_1, v_2^8]/v_1^2 \otimes Q_1(x_{1,2}x_{1,3})N_{\{(1,1),(1,4),\ldots, (1,k)\}},
\end{multline*}
and for $k > 2$ odd the WSS $d_0$ gives isomorphisms
\begin{multline*}
d_0^{wss}: \FF_2[h_{2,1}^{\pm}, v_2^8] \otimes x_{1,1}t_{1,1}N_{\{(1,2), \ldots, (1,k-1)\}} \\ 
\xrightarrow{\cong} \FF_2[h_{2,1}^{\pm}, v_2^8] \otimes Q_1(x_{1,1}x_{1,2})N_{\{(1,3), \ldots, (1,k)\}},
\end{multline*}
\begin{multline*}
d_0^{wss}: \FF_2[h_{2,1}^{\pm}, v_2^8] \otimes x_{1,2}t_{1,2}N_{\{(1,1),(1,3), \ldots, (1,k-1)\}} \\ \xrightarrow{\cong} \FF_2[h_{2,1}^{\pm}, v_2^8] \otimes Q_1(x_{1,2}x_{1,3})N_{\{(1,1),(1,4),\ldots, (1,k)\}}.
\end{multline*}
\end{proof}

We shall denote the elements of the Mahowald-Tangora wedge \cite{MahowaldTangora} in $\Ext_{A_*}(\FF_2)$ by\footnote{This notation is slightly misleading, as there are a few wedge elements for which the $P$ operator does not take the element we are denoting $v_1^i x$ to the element we are denoting $v_1^{i+4}x$, but we justify this notation by the fact that the wedge elements map to elements with such names in $\Ext_{A(2)_*}(\FF_2)$. }
$$ v_1^i h_{2,1}^j g^2, \quad i \ge 0, j \ge 0. $$
Recall that the Mahowald operator
$$ M = \bra{g_2, h_0^3, -} $$
leads to an infinite collection of wedges
$$ M^k(v_1^i h_{2,1}^j g^2) \in \Ext_{A_*}(\FF_2) $$
with non-zero image in 
$$\Ext_{B_*}(\FF_2) = \Ext_{A(2)_*}(\FF_2)[v_3] $$
where $B_*$ is the quotient algebra
\begin{equation}\label{eq:Bstar}
 B_* := \FF_2[\zeta_1, \zeta_2, \zeta_3, \zeta_4]/(\zeta_1^8, \zeta_2^4, \zeta_3^2, \zeta_4^2)
 \end{equation}
of $A_*$ \cite{MargolisPriddyTangora},\cite{IsaksenMO}. 
The existence of the element $\Delta^2 g^2 \in \Ext_{A_*}(\FF_2)$ gives elements  
$$ \Delta^{2m} M^k(v_1^i h_{2,1}^{j+8m} g^2) \in \Ext_{A_*}(\FF_2). $$
These elements are all linearly independent, since they project to linearly independent elements of $\Ext_{B_*}(\FF_2)$.

The following proposition gives the elements of $\Ext_{A(2)_*}$ that some of the remaining $h_{2,1}$ towers in $\Ext_{A(2)_*}$ detect in the algebraic tmf-resolution.

\begin{prop}\label{prop:h21detection}
The following table lists, for $i \ge 0$, $m \ge 0$, and $j \ge 4$ an $A(2)_*$-comodule $M$, an $h_{2,1}$-tower in $g^{-1}\Ext_{A(2)_*}(M)$, the corresponding $h_{2,1}$-tower in $\Ext_{A(2)_*}(M)$, and an $h_{2,1}$-tower in $\Ext_{A_*}(\FF_2)$ that it detects in the algebraic tmf-resolution (assuming the latter is non-zero). 
\begin{center}
\begin{tabular}{c|c|c|c}
$M$ & $g^{-1}\Ext_{A(2)_*}(M)$ & $\Ext_{A(2)_*}(M)$ & $\Ext_{A_*}(\FF_2)$ \\
\hline
$\FF_2$ & $\Delta^{2m} v_1^i h^{j+8m+8}_{2,1}$ & 
$\Delta^{2m} v_1^i h^{j+8m}_{2,1} g^2$ &$\Delta^{2m} v_1^i h_{2,1}^{j+8m} g^2$ \\
&&&\\
$\bou_1$ & $\Delta^{2m} h^{j+8m+4}_{2,1} Q_1(x_{1,1})$ & $\Delta^{2m} h_{2,1}^{j+8m+4}\zeta_2^4$ & $\Delta^{2m}h_{2,1}^{j+8m} n$ \\
&&&\\
$\bou_2$ & $\Delta^{2m}h_{2,1}^{j+8m+6} Q_1(x_{2,1})$ & $\Delta^{2m}h^{j+8m+1}_{2,1} g(h_{2,1}v_0^{-2}v_2^2\zeta_1^{16})$ & $\Delta^{2m}h_{2,1}^{j+8m} Q_2$ \\
 & $\Delta^{2m}v_1^{i+2} h_{2,1}^{j+8m+11}x_{1,1}t_{1,1}$ & $\Delta^{2m} v_1^{i+2}h^{j+8m+2}_{2,1}g^2(v_0^{-1}v_2^2\zeta_1^8\zeta_2^4)$ & $\Delta^{2m}v_1^i h_{2,1}^{j+8m} Mg^2$ \\
\hline
\end{tabular}
\end{center}
(Note that the notation $Q_2$ in the above table refers to the name of the generator of $\Ext^{7,57+7}_{A_*}(\FF_2)$, and \emph{not} the Milnor generator $Q_2 \in A$.) 
\end{prop}

\begin{proof}
The classes corresponding to $\Delta^{2m} v_1^i h^k_{2,1}$ are clear, because they are in the image of the map
$$ \Ext_{A_*}(\FF_2) \rightarrow \Ext_{A(2)_*}(\FF_2). $$
In the case of the classes corresponding to $\Delta^{2m}h^k_{2,1} n$, $\Delta^{2m}h^k_{2,1} Q_2$,
we consider the $h^j_{2,1}$ multiples of $n, Q_2 \in \Ext_{A_*}(\FF_2)$ for $j \ge 4$:
 \begin{align*}
 & gn, \: gt, \: rn, \: mn, g^2 n, \cdots, \\
 & gQ_2, \: gC_0, \: rQ_2, \: mQ_2, \: g^2 Q_2 \cdots.
 \end{align*}
It suffices to show that 
$$ n,\: t, \: Q_2, \: C_0$$
are detected in the algebraic tmf-resolution by
\begin{equation}\label{eq:ntQ2C0}
 h_{2,1}^4\zeta_2^4 + \alpha_1, \: h^5_{2,1}\zeta_2^4+\alpha_2, \: h^6_{2,1}v_0^{-2}v_2^2\zeta_1^{16} + \alpha_3, \: h^{7}_{2,1} v_0^{-2}v_2^2\zeta_1^{16} + \alpha_4
\end{equation}
where $g\alpha_i = r\alpha_i = m\alpha_i = 0$.

Examination of a computer calculation of $\Ext_{A_*}(\br{A\mmod A(2)}^{\otimes 2}_*)$ reveals that none of the elements $n,t, Q_2, C_0$ are in the image of the map
\begin{equation}\label{eq:AmodA22map}
 \Ext^{*,*}_{A_*}(\br{A \mmod A(2)}^{\otimes 2}_*) \rightarrow \Ext^{*+2,*}_{A_*}(\FF_2). 
 \end{equation}
Since the elements $n, t, Q_2,$ and $C_0$ map to zero in $\Ext_{A(2)_*}(\FF_2)$, they must therefore be detected on the 1-line of the algebraic tmf-resolution.  Examination of the relevant Ext charts reveals the only possibility is for the elements to be detected by classes of the form (\ref{eq:ntQ2C0}).

If we consider the class $Mg \in \Ext_{A_*}(\FF_2)$, one can both check that it is not in the image of (\ref{eq:AmodA22map}), and that the only class in $\Ext_{A(2)_*}(\br{A \mmod A(2)}_*)$ which can detect it is the class 
$$ e_0^2(v_0^{-1}v_2^2\zeta_1^8\zeta_2^4) \in \Ext_{A(2)_*}(\bou_2). $$ 
It follows from the multiplicative structure of the wedge, and the fact that
$$ g e_0^2 = v_1^2 h^2_{2,1}g^2, $$
that the elements $v_1^i h^j_{2,1}Mg^2 \in \Ext_{A_*}(\FF_2)$ are detected by 
$$ v_1^{i+2} h^{j+2}_{2,1} g^2(v_0^{-1}v_2^2\zeta_1^8\zeta_2^4) \in \Ext_{A(2)_*}(\bou_2) $$ 
for $i \ge 0$ and $j \ge 4$.
\end{proof}

\section{The MASS for $M(8,v_1^8)$}\label{sec:M38}

In this and following sections, we shall use the notation
$$ x[k] $$
to denote an element of $\Ext_{A(2)_*}(M \otimes H(8,v_1^8))$ detected by an element 
$$ x\in \Ext_{A(2)_*}(M)$$
on the $k$-cell of $H(8,v_1^8)$ ($k \in \{0,1,17,18\}$).

\subsection*{The MASS for $\pmb{\tmf_*M(8,v_1^8)}$}

\begin{figure}
\includegraphics[angle = 90, origin=c, height =.7\textheight]{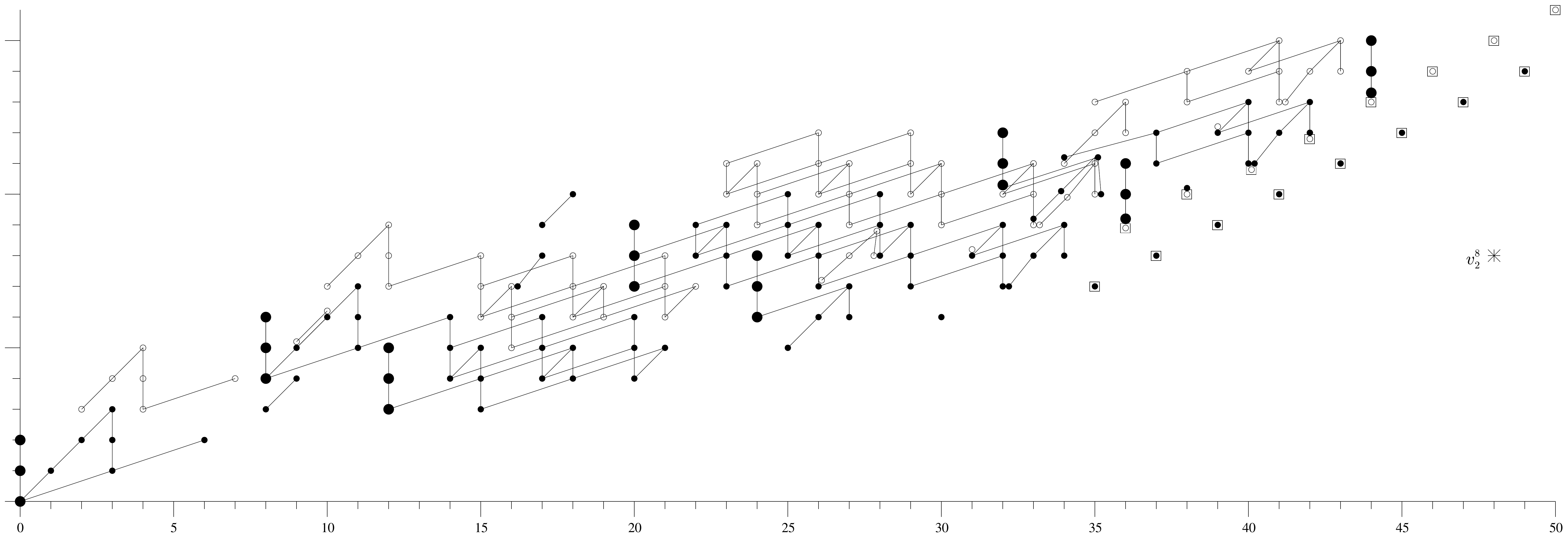}
\caption{The groups $\Ext_{A(2)_*}(H(8,v_1^8))$.}\label{fig:ExtH38}
\end{figure}

The computation of $\Ext_{A(2)_*}(H(8,v_1^8))$ is depicted in Figure~\ref{fig:ExtH38}.  
In this figure, solid dots correspond to classes carried by the ``$0$-cell'' of $H(8,v_1^8)$, and open circles correspond to classes carried by the ``$1$-cell'' of $H(8,v_1^8)$.  The large solid circles correspond to $h_0$-torsion free classes of $\Ext_{A(2)_*}(\FF_2)$ on the $0$-cell of $H(8,v_1^8)$.
The classes with solid boxes around them support $h_{2,1}$ towers.  Everything is $v_2^8$-periodic.
  

\begin{figure}
\includegraphics[angle = 90, origin=c, height =.7\textheight]{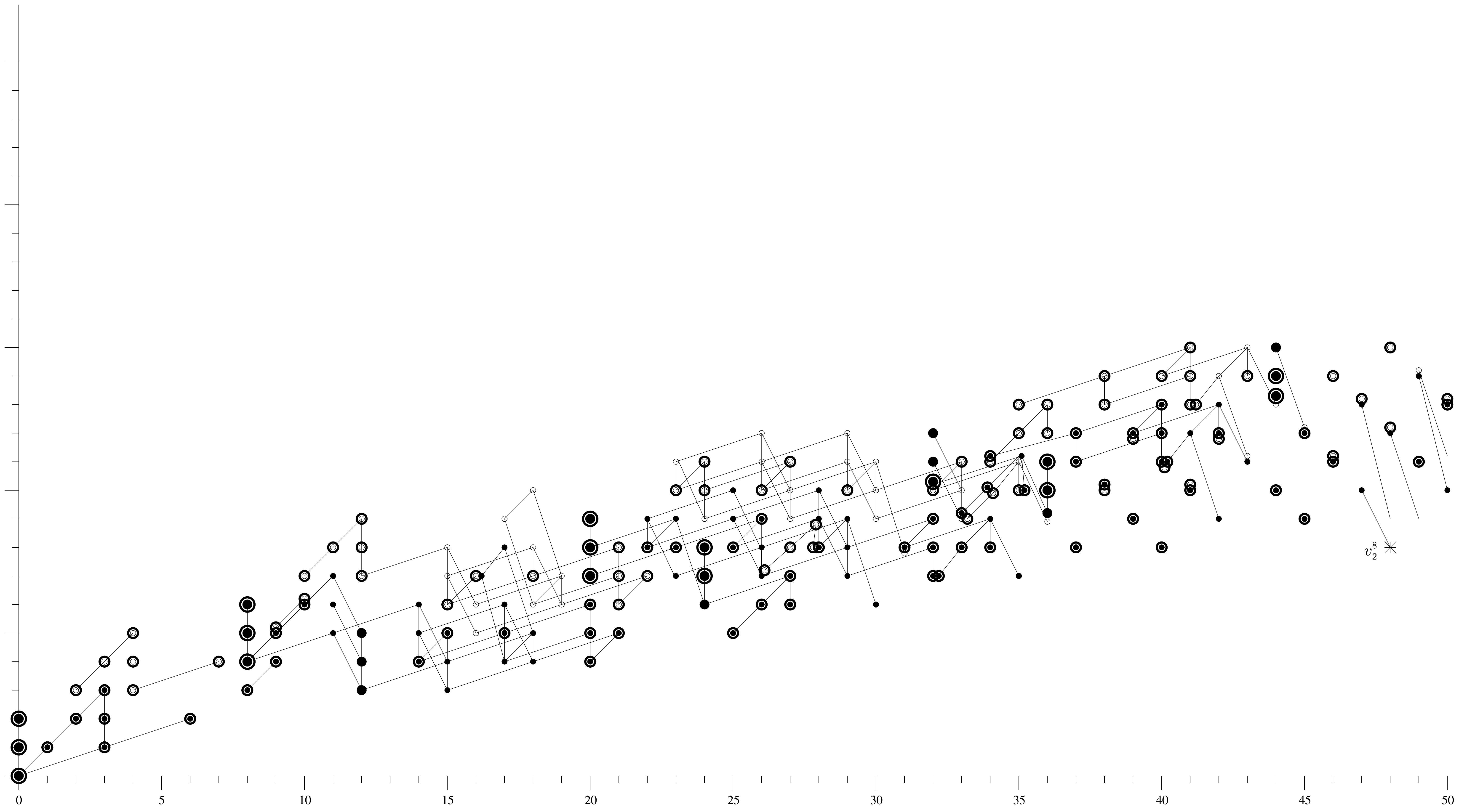}
\caption{The MASS for $\tmf \wedge M(8,v_1^8)$.}\label{fig:MASSM38}
\end{figure}

Figure~\ref{fig:MASSM38} depicts the differentials in the MASS for $\tmf \wedge M(8,v_1^8)$ through the same range; the complete computation of this MASS can be similarly accomplished.  An explanation of how to determine these differentials can be found in \cite{BHHM2}.  



\subsection*{The algebraic tmf-resolution for $\pmb{H(8,v_1^8)}$}

The following lemma explains that, in our $H(8,v_1^8)$ computations, we may disregard terms coming from $\Ext_{A(1)_*}$ in the sequence of spectral sequences (\ref{eq:boIss}).

\begin{lem}[Lemma~8.8 of \cite{BHHM2}]
In the algebraic $\tmf$-resolution for $M = H(8,v_1^8)$, the terms $$\Ext_{A(1)_*}(\text{something})$$ in (\ref{eq:boIss}) do not contribute to $\Ext^{s,t}_{A_*}(H(8,v_1^8))$ if
$$ s > \frac{1}{7}(t-s)+\frac{51}{7}. $$
\end{lem}

\begin{figure}
\includegraphics[angle = 90, origin=c, height =.7\textheight]{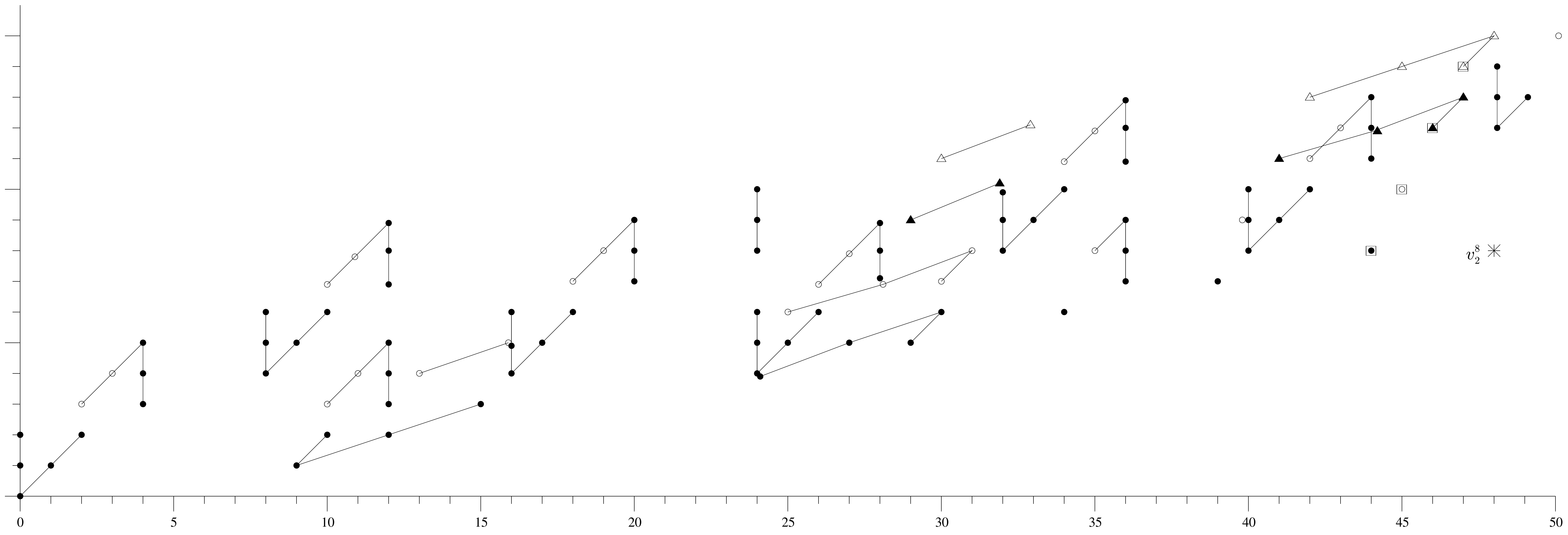}
\caption{$\Ext_{A(2)_*}(\bou_1 \otimes H(8,v_1^8))$.}\label{fig:Extbo1H38}
\end{figure}

For $n > 0$ and $i_1, \ldots, i_n > 0$, the terms
$$ \Ext_{A(2)_*}^{s,t}(\bou_{i_1} \otimes \cdots \otimes \bou_{i_n} \otimes H(8,v_1^{8})) $$
that comprise the terms in the algebraic $\tmf$-resolution for $H(8,v_1^8)$
are in some sense less complicated than $\Ext_{A(2)_*}(H(8,v_1^{8}))$.

Most of the features of these computations can already be seen in the computation of $\Ext_{A(2)_*}(\bou_1 \otimes H(8,v_1^8))$, which is displayed in Figure~\ref{fig:Extbo1H38}.  This computation was performed by taking the computation of $\Ext_{A(2)_*}(\bou_1)$ (see, for example, \cite{BHHM}) and running the long exact sequences in $\Ext$ associated to the cofiber sequences
\begin{gather*}
\Sigma^3 \bou_1 [-3] \xrightarrow{h_0^3} \bou_1 \rightarrow \bou_1 \otimes H(8), \\
\Sigma^{24}\bou_1 \otimes H(8)[-8] \xrightarrow{v_1^8} \bou_1 \otimes H(8) \rightarrow \bou_1 \otimes H(8,v_1^8). 
\end{gather*}
In Figure~\ref{fig:Extbo1H38}, as before, solid dots represent generators carried by the 0-cell of $H(8,v_1^8)$ and open circles are carried by the 1-cell.  Unlike the case of $\Ext_{A(2)_*}(H(8))$, there is $v_1^8$-torsion in $\Ext_{A(2)_*}(\bou_1 \otimes H(8))$.  This results in classes in $ \Ext_{A(2)_*}(\bou_1 \otimes H(8,v_1^8)) $ 
carried by the 17-cell and the 18-cell of $H(8,v_1^8)$, which are represented by solid triangles and open triangles, respectively.  A box around a generator indicates that that generator actually carries a copy of $\FF_2[h_{2,1}]$.  As before, everything is $v_2^8$-periodic.

One can similarly compute
$$ \Ext_{A(2)_*}(\bou_1^{\otimes k} \otimes H(8,v_1^8)) $$
for larger values of $k$ by applying the same method to the corresponding computations of 
$$ \Ext_{A(2)_*}(\bou_1^{\otimes k}) $$
in \cite{BHHM}.  We do not bother to record the complete results of these computations for small values of $k$, but will freely use them in what follows.  The spectral sequences (\ref{eq:boIss}) imply these computations control $\Ext_{A(2)_*}(\bou_I)$.


\subsection*{$\pmb{h_{2,1}}$ towers in the algebraic tmf-resolution for $\pmb{H(8,v_1^8)}$}

Theorem~\ref{thm:BBT} has the following implication for the $g$-local algebraic $\tmf$-resolution of $H(8,v_1^8)$:
\begin{multline*} 
h_{2,1}^{-1}\Ext^{*,*}_{E[Q_1,P_2^1]}(A \mmod A(2)_*^{\otimes n} \otimes H(8,v_1^8)) = \\
\FF_2[h_{2,1}^{\pm}] \otimes \bigg( \FF_2[v_1]/v_1^8\otimes H(8)\{x_{J'}t_{J'}\}_{J'} \oplus 
\bigoplus_{\abs{J} \: \mr{odd}} N_J \otimes H(8,v_1^8)\{x_{J'}t_{J'}\}_{J \cap J' = \emptyset} \\
\oplus \bigoplus_{\abs{J} \ne 0 \: \mr{even}}
\FF_2[v_1]/v_1^2 \otimes N_J \otimes H(8,v_1^8)\{x_{J'}t_{J'}\}_{J \cap J' = \emptyset} \bigg)
\end{multline*}
where $J$ and $J'$ range over the subsets of
$$ \{ (i,j) \: : \: 1 \le i, 1 \le j \le n \}. $$ 

This leads to the following twist in the analog of Proposition~\ref{prop:h21death}.

\begin{prop}\label{prop:h21deathH38}
In $g^{-1}\E{wss}{}_0(H(8,v_1^8))$, all of the $h_{2,1}$-towers coming from $$\Ext_{A(2)_*}(\bou_1^{\otimes k} \otimes H(8,v_1^8))$$ 
for $k \ge 3$ are either the source of a non-trivial $d_0$-differential, or are the target of a $d_0$-differential.  For $k = 2$, the $h_{2,1}$ towers
$$ v_1^\epsilon h^j_{2,1}Q_1(x_{1,1}x_{1,2})[n] $$
are  killed for $\epsilon \in \{0,1\}$ and $n \in \{0,1\}$ (but the corresponding towers with $n \in \{17,18\}$ are \emph{not} killed).
\end{prop}

\begin{proof}
Everything is identical to the proof of \ref{prop:h21death}, except that the differentials
$$ d_0^{wss}: \FF_2[v_1, h^{\pm}_{2,1}]/v_1^8 \{x_{1,1}t_{1,1}\} \otimes H(8) \rightarrow \FF_2[v_1, h_{2,1}^{\pm}]/v_1^2 \{ Q_1(x_{1,1}x_{1,2}) \}\otimes H(8,v_1^8) $$
now have non-trivial kernel and cokernel.
\end{proof}

We now give elements of $\Ext_{A_*}(H(8,v_1^8))$ which these remaining $h_{2,1}$-towers detect in the algebraic tmf-resolution.
Note that, as pointed out in \cite{MargolisPriddyTangora}, the Mahowald operator satisfies 
$$ h_0^3 M(x) = 0 $$
which implies that for any $x \in \Ext_{A_*}(\FF_2)$, there exists a lift 
$$ M(x)[1] \in \Ext_{A_*}(H(8)) $$
and thus an element $M(x)[1] \in \Ext_{A_*}(H(8,v_1^8))$.  Furthermore, the element $\Delta^2 = v_2^8$ exists in $\Ext_{A_*}(H(8,v_1^8))$ (see Lemma~\ref{lem:v2^8} below).  
We conclude that for $0 \le i \le 7$, $j,k,l \ge 0$, and $\epsilon \in \{0,1\}$
the wedge elements
$$ v_1^i h^j_{2,1}\Delta^{2k}M^l g^2 [\epsilon] \in \Ext_{A_*}(H(8,v_1^8)) $$
exist, and we see they are linearly independent by mapping to $\Ext_{B_*}(H(8,v_1^8))$ (where $B_*$ is defined in (\ref{eq:Bstar})).

\begin{prop}\label{prop:h21detectionH38}
The following table lists, for $m \ge 0$, $0 \le i \le 7$, $0 \le i' \le 5$, $j \ge 4$, $k \in \{0,1,17,18\}$, and $\epsilon, \epsilon' \in \{0,1\}$ an $A(2)_*$-comodule $M$, an $h_{2,1}$-tower in $g^{-1}\Ext_{A(2)_*}(M \otimes H(8,v_1^8))$, the corresponding $h_{2,1}$-tower in $\Ext_{A(2)_*}(M \otimes H(8,v_1^8))$, and an $h_{2,1}$-tower in $\Ext_{A_*}(H(8,v_1^8))$ that it detects in the algebraic tmf-resolution. 
\begin{center}
{\scriptsize \begin{tabular}{c|c|c|c}
$M$ & $g^{-1}\Ext_{A(2)_*}(M \otimes H(8,v_1^8))$ & $\Ext_{A(2)_*}(M \otimes H(8,v_1^8))$ & $\Ext_{A_*}(H(8,v_1^8))$ \\
\hline
$\FF_2$ & $\Delta^{2m}v_1^i h^{j+8}_{2,1}[\epsilon]$ & 
$\Delta^{2m}v_1^i h^{j}_{2,1} g^2[\epsilon]$ &$\Delta^{2m}v_1^i h_{2,1}^{j} g^2[\epsilon]$ \\
&&&\\
$\bou_1$ & $\Delta^{2m}h^{j+4}_{2,1} Q_1(x_{1,1})[k]$ & $\Delta^{2m}h_{2,1}^{j+4}\zeta_2^4[k]$ & $\Delta^{2m}h_{2,1}^{j} n[k]$ \\
&&&\\
$\bou_2$ & $\Delta^{2m}h_{2,1}^{j+6} Q_1(x_{2,1})[k]$ & $\Delta^{2m}h^{j+1}_{2,1} g(h_{2,1}v_0^{-2}v_2^2\zeta_1^{16})[k]$ & $\Delta^{2m}h_{2,1}^{j} Q_2[k]$ \\
 & $\Delta^{2m}v_1^{i'+2} h_{2,1}^{j+11}x_{1,1}t_{1,1}[\epsilon]$ & $\Delta^{2m}v_1^{i'+2}h^{j+2}_{2,1}g^2(v_0^{-1}v_2^2\zeta_1^8\zeta_2^4)[\epsilon]$ & $\Delta^{2m}v_1^{i'} h_{2,1}^{j} Mg^2[\epsilon]$ \\
&&&\\
$\bou_1^{\otimes 2}$ & $v_1^{\epsilon'} \Delta^{2m}h_{2,1}^{j+11}Q_1(x_{1,1}x_{1,2})[17+\epsilon]$ & $\Delta^{2m}v_1^{\epsilon'}h^{j+2}_{2,1}g^2(v_0^{-1}v_2^2[\zeta_1^8,\zeta_2^4])[17+\epsilon]$ & $\Delta^{2m}v_1^{6+\epsilon'} h_{2,1}^{j} Mg^2[\epsilon]$ \\
\hline
\end{tabular} }
\end{center}
\end{prop}

\begin{proof}
The cases of
\begin{align*}
& \Delta^{2m} v_1^i h^j_{2,1} g^2[\epsilon], \\
& \Delta^{2m} h_{2,1}^j n[\epsilon], \\
& \Delta^{2m} h_{2,1}^j Q_2[\epsilon], \\
& \Delta^{2m} v_1^{i'} h^j_{2,1} Mg^2[\epsilon]
\end{align*}
follow immediately from Proposition~\ref{prop:h21detection} since all of these elements are annihilated by $v_0^3$.

The elements
\begin{equation*}
\begin{split}
h^{j+4}_{2,1} \zeta_2^4 \in \Ext_{A(2)_*}(\bou_1), \\
h^{j+6}_{2,1} \zeta_1^{16} \in \Ext_{A(2)_*}(\bou_2)
\end{split}
\end{equation*}
lift to elements
\begin{equation}\label{eq:nQ2lifts}
\begin{split}
h^{j+4}_{2,1} \zeta_2^4[17 + \epsilon] \in \Ext_{A(2)_*}(\bou_1 \otimes H(8,v_1^8)), \\
h^{j+6}_{2,1} \zeta_1^{16}[17+\epsilon] \in \Ext_{A(2)_*}(\bou_2 \otimes H(8,v_1^8).
\end{split}
\end{equation}
One can explicitly check that the lifts (\ref{eq:nQ2lifts}) are permanent cycles in the algebraic tmf-resolution.  Therefore they detect the desired elements
$$ h_{2,1}^j n[17 + \epsilon], \: h^j_{2,1} Q_2[17+\epsilon] \in 
\Ext_{A_*}(H(8,v_1^8)). $$
Applying Case (5) of the Geometric Boundary Theorem \cite[Lem.A.4.1]{goodehp}
to the triangle
$$ H(8,v_1^8)[-1] \rightarrow \Sigma^{24} H(8)[-8] \xrightarrow{v_1^8} H(8) \rightarrow H(8,v_1^8) $$
and the differential
$$ d_1(v_1^{\epsilon'}h^{j+2}_{2,1}g^2(v_0^{-1}v_2^2\zeta_1^8\zeta_2^4))
= v_1^{\epsilon'}h^{j+2}_{2,1}g^2(v_0^{-1}v_2^2[\zeta_1^8,\zeta_2^4]) $$
in the algebraic tmf-resolution for $\Sigma^{24}H(8)[-8]$  (\ref{eq:dzeta18zeta24}), we find that
the images of the elements
$$ v_1^{8 + \epsilon'} h_{2,1}^j M(g^2)[\epsilon] \in \Ext_{A_*}(H(8)) $$
under the map
$$ \Ext_{A_*}(H(8)) \rightarrow \Ext_{A_*}(H(8,v_1^8)) $$
are detected by the elements
$$ v_1^{\epsilon'}h^{j+2}_{2,1}g^2(v_0^{-1}v_2^2[\zeta_1^8,\zeta_2^4])[17+\epsilon] $$
in the algebraic tmf-resolution for $H(8,v_1^8)$.
\end{proof}

\section{The $v_2^{32}$ self-map on $M(8,v_1^8)$}\label{sec:v232}

We now endeavor to prove Theorem~\ref{thm:v232}.  We first recall the following lemma.

\begin{lem}[Lem.~7.6 of \cite{BHHM2}]\label{lem:v2^8}
The element
$$ v_2^8 \in \Ext^{8,48+8}_{A(2)_*}(H(8,v_1^8)) $$
is a permanent cycle in the algebraic $\tmf$-resolution, and gives rise to an element
$$ v_2^8 \in \Ext^{8,48+8}_{A_*}(H(8,v_1^8)). $$
\end{lem}

It follows from the Leibniz rule that $v_2^{32}$ persists to the $E_4$-page of the MASS for $M(8,v_1^8)$.  Our task will then be reduced to showing that $d_r(v_2^{32}) = 0$ for $r \ge 4$.  We will do this by identifying the potential targets of such a differential, and show that they either the source or target of shorter differentials.  This will necessitate lifting certain differentials from the MASS for $\tmf \wedge \br{\tmf}^n \wedge M(8,v_1^8)$ to the MASS for $M(8,v_1^8)$.

As explained in \cite[Sec.~7.4]{BOSS}, work of the second author, Davis, and Rezk \cite{MahowaldRezk},\cite{DavisMahowald} implies that the algebraic map
$$ \Ext_{A(2)}(\Sigma^8 \bou_1 \oplus \Sigma^{16}\bou_2) \rightarrow \Ext_{A(2)_*}(\br{A\mmod A(2)}_*) $$
realizes to a map
\begin{equation}\label{eq:bo1bo2}
\tmf \wedge \br{\tmf}_2  \rightarrow \tmf \wedge \br{\tmf} 
\end{equation}
where $\tmf \wedge \br{\tmf}_2$ is a spectrum built out of $\tmf \wedge \Sigma^8 \bo_1$ and $\tmf \wedge \Sigma^{16} \bo_2$.
They furthermore show that there is a map
\begin{equation}\label{eq:tmfinbo2}
 \Sigma^{32} \tmf \rightarrow \tmf \wedge \br{\tmf}_2 
 \end{equation}
which geometrically realizes the inclusion of the direct summand (\ref{eq:boIss})
$$ \Ext_{A(2)_*}(\Sigma^{33}\FF_2[-1]) \hookrightarrow \Ext_{A(2)_*}(\Sigma^{16} \bou_2) \subset \Ext_{A(2)_*}(\Sigma^8 \bou_1 \oplus \Sigma^{16} \bou_2). $$
The attaching map from $\tmf \wedge \bo_2$ to $\tmf \wedge \bo_1$ in the spectrum $\tmf\wedge \br{\tmf}_2$ induces $d_3$-differentials from the $h_{2,1}$-towers in $\bo_2$ to the $h_{2,1}$-towers in $\bo_1$ in the ASS for $\tmf \wedge \br{\tmf}$ under the map (\ref{eq:bo1bo2}).  Furthermore, there are differentials in the ASS's for $\tmf \wedge \bo_1$, $\tmf \wedge \bo_2$, and $\tmf$, which induce differentials in the ASS for $\tmf \wedge \br{\tmf}$ under the maps (\ref{eq:bo1bo2}) and (\ref{eq:tmfinbo2}).  We wish to study when these differentials (and more generally differentials in the ASS for $\tmf \wedge \br{\tmf}^n$) lift via the tmf-resolution to differentials in the ASS for the sphere.

To this end we consider the partial totalizations
$$ T^n := {\Tot}^n (\tmf^{\bullet+1}) $$
of the cosimplicial $\tmf$-resolution of the sphere, so that we have
$$ S \simeq \varprojlim_n T^n $$
and fiber sequences
$$ \Sigma^{-n}\tmf \wedge \br{\tmf}^{n} \rightarrow T^n \rightarrow T^{n-1}. $$
The spectrum $T^n$ is a ring spectrum, and in particular has a unit
$$ S \rightarrow T^n.$$
We let 
\begin{equation}\label{eq:Ful}
 \ul{T}^n = {\Tot}^n (A\mmod A(2)_*^{\otimes \bullet + 1})
\end{equation}
denote the corresponding construction in the stable homotopy category of $A_*$-comodules.
There is a MASS
$$ \Ext^{*,*}_{A_*}(\ul{T}^n \otimes H(8,v_1^8)) \Rightarrow T^n_*M(8,v_1^8) $$
and the algebraic tmf-resolution for $H(8,v_1^8)$ truncates to give an algebraic tmf-resolution 
$$ \bigoplus_{i = 0}^n  \Ext^{*,*}_{A(2)_*}(\br{A \mmod A(2)}_*^{\otimes i} \otimes H(8,v_1^8)) \Rightarrow \Ext_{A_*}(\ul{T}^n \otimes H(8,v_1^8)). $$

The following lemma will be our key to lifting the desired differentials.

\begin{lem}\label{lem:technical}
Suppose $x$ is an element of $\Ext_{A_*}(H(8,v_1^8))$ which is detected in the $n$-line of the algebraic tmf-resolution for $H(8,v_1^8)$ by an element
$$ x' \in \Ext_{A(2)_*}(\br{A \mmod A(2)}_*^{\otimes n} \otimes H(8,v_1^8)). $$
Furthermore, suppose that in the MASS for $\tmf \wedge \br{\tmf}^n \wedge M(8,v_1^8)$, there is a differential
$$ d^{mass}_r(x') = y' $$ 
and that for $2 \le r' < r$ we have
$$ d^{mass}_{r'}(x) = 0 $$
in the MASS for the $M(8,v_1^8)$.  Then either of the following is true:
\begin{enumerate}
\item The differential 
$$ d^{mass}_r(x) $$ 
in the ASS for $M(8,v_1^8)$ is detected by $y'$ in the algebraic tmf-resolution, or 
\item The element $y'$ is the target of a differential in the algebraic tmf-resolution for $H(8,v_1^8)$, or in the algebraic tmf-resolution for $\ul{T}^n \otimes H(8,v_1^8)$ the element $y'$ detects an element of $\Ext_{A_*}(\ul{T}^n \otimes H(8,v_1^8))$ which is zero in $\E{mass}{}_r(T^n \wedge M(8,v_1^8))$.
\end{enumerate}
\end{lem} 

\begin{proof}
Consider the maps of algebraic tmf-resolutions and MASS's induced from the zig-zag
$$ M(8,v_1^8) \xrightarrow{\alpha} T^n \wedge M(8,v_1^8) \xleftarrow{\beta} \Sigma^{-n} \tmf \wedge \br{\tmf}^n \wedge M(8,v_1^8). $$
Define
$$ \br{x} := \alpha_*(x) \in \Ext_{A_*}(\ul{T}^n \otimes H(8,v_1^8))$$ 
Then $\br{x}$ is detected by $x'$, regarded as an element of the algebraic tmf-resolution for $T^n \wedge M(8,v_1^8)$.  In particular, this means that 
$$ \br{x} = \beta_*(x') $$ 
Therefore, the differential
$$ d_r^{mass}(x') = y' $$
in the MASS for $\tmf \wedge \br{\tmf}^n \wedge M(8,v_1^8)$ maps to a differential
$$ d_r^{mass}(\br{x}) = \br{y} := \beta_*(y') $$
in the MASS for $T^n \wedge M(8,v_1^8)$.  In particular, either (Case 1) $\br{y}$ is nonzero in $\E{mass}{}_r(T^n \wedge M(8,v_1^8)$ and is detected by $y'$ in the algebraic tmf-resolution for $\ul{T}^n \otimes H(8,v_1^8)$, or (Case 2) either $\br{y} = 0$ in $\E{mass}{}_r(T^n \wedge M(8,v_1^8))$ or $y'$ is killed in the algebraic tmf-resolution for $\ul{T}^n \otimes H(8,v_1^8)$.   If the latter is true, then $y'$ is killed in the algebraic tmf-resolution for $H(8,v_1^8)$, since the algebraic tmf-resolution for $\ul{T}^n \otimes H(8,v_1^8)$ is a truncation of the algebraic tmf-resolution for $H(8,v_1^8)$.

If we are in Case (2), we are done.  If we are in Case (1), consider the differential
$$ y := d^{mass}_r(x) $$
in the MASS for $M(8,v_1^8)$ (which is defined by hypothesis).  
We must have 
$$ \alpha_*(y) = \br{y}. $$
Therefore, $d_r^{mass}(x)$ is detected by $y'$ in the algebraic tmf-resolution.
\end{proof}

\begin{rmk}\label{rmk:technical}
We will primarily be applying Lemma~\ref{lem:technical} to the following two cases:
\begin{description}
\item[Case 1] $\pmb{x = \Delta^{2m} h^{j}_{2,1}Q_2[k]}$.  Suppose that we can prove 
$$ d^{ass}_2(\Delta^{2m} h^{j}_{2,1}Q_2[k]) = 0 $$ 
in the MASS for $M(8,v_1^8)$.  The element $\Delta^{2m} h^{j}_{2,1}Q_2[k]$ is detected by 
$$ \Delta^{2m} h^{j+1}_{2,1} g(h_{2,1}v_0^{-2}v_2^2\zeta_1^{16})[k] \in \Ext_{A(2)_*}(\bou_2 \otimes H(8,v_1^8)) $$
in the algebraic tmf-resolution, and it is proven in \cite{BOSS} that in the ASS for $\tmf \wedge \br{\tmf}$ there is a differential
\begin{multline*}
d^{ass}_3(\Delta^{2m} h^{j+1}_{2,1} g(h_{2,1}v_0^{-2}v_2^2\zeta_1^{16})) = \\ \Delta^{2m} h^{j+4}_{2,1} g(h_{2,1}\zeta_2^4) + \epsilon(m) \Delta^{2m-4}h^{j+20}_{2,1} g(h_{2,1}v_0^{-2}v_2^2\zeta_1^{16}) 
\end{multline*}
where
$$ 
\epsilon(m) = 
\begin{cases}
1, & m \equiv 2 \mod 4, \\
0, & \rm{otherwise}.  
\end{cases}
$$
Lifting this differential to $\tmf \wedge \br{\tmf} \wedge M(8,v_1^8)$,  Lemma~\ref{lem:technical} implies that either the target of the differential $d^{ass}_3(\Delta^{2m} h^{j}_{2,1}Q_2[k])$ in the MASS for $M(8,v_1^8)$ is detected by
$$ \Delta^{2m} h^{j+4}_{2,1} g(h_{2,1}\zeta_2^4)[k]+\epsilon(m) \Delta^{2m-4} h^{j+20}_{2,1} g(h_{2,1}v_0^{-2}v_2^2\zeta_1^{16})[k] $$
in the algebraic tmf-resolution, or 
$$ \Delta^{2m} h^{j+4}_{2,1} g(h_{2,1}\zeta_2^4)[k]+\epsilon(m) \Delta^{2m-4} h^{j+20}_{2,1} g(h_{2,1}v_0^{-2}v_2^2\zeta_1^{16})[k] $$ 
is the target of a differential in the algebraic tmf-resolution or detects an element of $\Ext_{A_*}(\ul{T}^1 \otimes H(8,v_1^8))$ which is zero on the $E_3$-page of the MASS for $T^1 \wedge M(8,v_1^8)$.
\vspace{10pt}

\item[Case 2] $\pmb{x = M \Delta^2 v_1^i h_{2,1}^{j+8}[\epsilon]}$ for $\epsilon \in \{0,1\}$ and $0 \le i \le 4$.  The element $M \Delta^2 v_1^i h_{2,1}^{j+8}[\epsilon]$ is detected by 
$$ \Delta^2 v_1^{i+2} h^{j+10}_{2,1}(v_0^{-1}v_2^2\zeta_1^8\zeta_2^4)[\epsilon] $$
in the algebraic tmf-resolution for $H(8,v_1^8)$ , and the map (\ref{eq:tmfinbo2}) implies there is a differential
$$ d^{mass}_2(\Delta^2 v_1^{i+2} h^{j+10}_{2,1}(v_0^{-1}v_2^2\zeta_1^8\zeta_2^4)[\epsilon]) = 
v_1^{i+3} h^{j+19}_{2,1}(v_0^{-1}v_2^2\zeta_1^8\zeta_2^4)[\epsilon] $$
in the MASS for $\tmf \wedge \br{\tmf} \wedge M(8,v_1^8)$.  Then Lemma~\ref{lem:technical} implies that either $d^{mass}_2(M \Delta^2 v_1^i h_{2,1}^{j+8}[\epsilon])$ is detected by
$$ v_1^{i+3} h^{j+19}_{2,1}(v_0^{-1}v_2^2\zeta_1^8\zeta_2^4)[\epsilon] $$
in the algebraic tmf-resolution, or $v_1^{i+3} h^{j+19}_{2,1}(v_0^{-1}v_2^2\zeta_1^8\zeta_2^4)[\epsilon]$ is killed in the tmf-resolution for $H(8,v_1^8)$ or it detects an element which is zero in the $E_2$-term of the MASS for $T^1 \wedge M(8,v_1^8)$.  However, the element
$$ M v^{i+1}_1 h_{2,1}^{j+17}[\epsilon] \in \Ext_{A_*}(H(8,v_1^8)) $$
is non-zero, and is detected by $v_1^{i+3} h^{j+19}_{2,1}(v_0^{-1}v_2^2\zeta_1^8\zeta_2^4)[\epsilon]$ in the algebraic tmf-resolution for $H(8,v_1^8)$.  We conclude that $v_1^{i+3} h^{j+19}_{2,1}(v_0^{-1}v_2^2\zeta_1^8\zeta_2^4)[\epsilon]$ is not killed in the algebraic tmf-resolution for $H(8,v_1^8)$.  Since the algebraic tmf-resolution for $\ul{T}^1\otimes H(8,v_1^8)$ is a truncation of the algebraic tmf-resolution for $H(8,v_1^8)$, we conclude that $v_1^{i+3} h^{j+19}_{2,1}(v_0^{-1}v_2^2\zeta_1^8\zeta_2^4)[\epsilon]$ detects a non-trivial element of the $E_2$-page of the MASS for $T^1 \wedge M(8,v_1^8)$.  We conclude that 
$$ d^{mass}_2(M \Delta^2 v_1^i h_{2,1}^{j+8}[\epsilon])$$
is non-trivial in the MASS for $M(8,v_1^8)$, and is detected in the algebraic tmf-resolution by $v_1^{i+3} h^{j+19}_{2,1}(v_0^{-1}v_2^2\zeta_1^8\zeta_2^4)[\epsilon]$.
\end{description}
\end{rmk}

\begin{proof}[Proof of Theorem~\ref{thm:v232}]
By Proposition~\ref{prop:ring}, it suffices to prove that
$$ v_2^{32} \in \Ext_{A_*}(H(8,v_1^8)) $$
is a permanent cycle in the MASS.  Furthermore, since $v_2^8 \in \E{mass}{}_2(M(8,v_1^8))$, the Leibniz rule implies that $v_2^{32} \in \E{mass}{}_4(M(8,v_1^8))$.  We therefore are left with eliminating possible targets of $d^{mass}_r(v_2^{32})$ for $r \ge 4$.

Suppose that $d_r(v_2^{32})$ is non-trivial for $r \ge 4$.  We successively consider terms in the algebraic tmf-resolution which could detect $d_r(v_2^{32})$, and then eliminate these possibilities one by one.

The only terms in the algebraic tmf-resolution $E_1$-page which can contribute to $\Ext^{s,191+s}_{A_*}(H(8,v_1^8))$ for $s \ge 36$ are 
\begin{itemize}
\item $\Ext_{A(2)_*}(\bou^{\otimes s}_1)$ for $0 \le s \le 6$, and 
\item $\Ext_{A(2)_*}(\bou^{\otimes s}_1 \otimes \bou_2)$ for $0 \le s \le 2$.
\end{itemize}  
Furthermore, $\bou_1^{\otimes s}$ only contributes $h_{2,1}$-towers in this range for $s = 5,6$.
We list these contributions below, except we   
do not list elements in $h_{2,1}$-towers coming from $\bou_1^{\otimes s}$ for $s \ge 2$ which are zero in the WSS $E_1$-term (see Proposition~\ref{prop:h21deathH38}).  Also, since $v_2^{32}$ is a permanent cycle in the MASS for $\tmf \wedge M(8,v_1^8)$, we can disregard any terms coming from $\Ext_{A(2)_*}(\FF_2)$ (the zero-line of the algebraic tmf-resolution).  Finally, we do not include any terms which can be eliminated through the application of Case 2 of Remark~\ref{rmk:technical}.

\begin{minipage}{\textwidth}
\begin{tab}\label{tab:v232targets}
{\bf List of potential targets of $\pmb{d^{mass}_r(v_2^{32})}$ for $\pmb{r \ge 4}$.} 
\end{tab}
\begin{center}
\begin{tabular}{c|c}
\hline
${\bou_1}$ & $h_{2,1}^{31}g(h_{2,1}\zeta_2^4)[0]$ \\
& $h_{2,1}^{18} \Delta^2 g(h_{2,1}\zeta_2^4)[17]$ \\
& \\
$\bou_2$ 
& $h^5_{2,1}\Delta^{4}g(h_{2,1}\zeta_1^{16})[18]$ \\
& $v_1^2 h_{2,1}^{31}(v_0^{-1}v_2^2\zeta_1^8\zeta_2^4)[1]$ \\
& \\
$\bou^{\otimes 2}_1$
& $h^5_{2,1}\Delta^4 v_1 g(v_0^{-1}v_2^2[\zeta_1^8,\zeta_2^4])[18]$ \\
& $h_{2,1}^{15}\Delta^2 g(v_0^{-1}v_2^2[\zeta_1^8,\zeta_2^4])[18]$ \\
& \\
$\bou_1 \otimes \bou_2$ & $v_1 h^{21}_{2,1} g(v_0^{-1}v_2^2 [\zeta_2^4,\zeta_2^{8}])[18]$ \\
& $v_1 h^{21}_{2,1} g(v_0^{-1}v_2^2 [\zeta_2^8,\zeta_2^{4}])[18]$ \\
& \\
$\bou_1^3$ 
& $v_1^4 \Delta^6 h_1 (v_2^2 \zeta_1^8|\zeta_1^8|\zeta_2^4)[1]$ \\
& \\
$\bou_1^{\otimes 2} \otimes \bou_2$
& $v_1 h_{2,1}^{18} g(v_0^{-2}v_2^4[\zeta_1^8,\zeta_2^4]|\zeta_1^8\zeta_2^4)[18]$ \\
& $v_1 h_{2,1}^{18} g(v_0^{-2}v_2^4(\zeta_1^8|\zeta_1^8\zeta_2^4|\zeta_2^4+\zeta_2^4|\zeta_1^8\zeta_2^4|\zeta_1^8))[18]$ \\
& $v_1 h_{2,1}^{18} g(v_0^{-2}v_2^4\zeta_1^8\zeta_2^4|[\zeta_1^8,\zeta_2^4])[18]$ \\
& \\
$\bou_1^4$ 
& $v_1^4 \Delta^6 h_1^2 \zeta_1^8|\zeta_1^8|\zeta_2^4|\zeta_2^4[1]$ \\
\hline
\end{tabular}   
\end{center}
\end{minipage}

We now eliminate these possibilities one by one.  We will consider the terms in the order of \emph{reverse} algebraic tmf filtration.
\begin{description}
\item[$\pmb{\bou_1^{\otimes 4}}$] In the modified May spectral sequence (\ref{eq:mmss}) there is a differential
$$ d^{mmss}_8(b_{2,2}h^2_3) = h_3^5 $$
which lifts under the map $\Phi_*$ of \ref{eq:Phi} to a non-trivial differential
$$ d_1^{wss}([\zeta^8_1|\zeta^8_1|\zeta_2^4|\zeta_2^4]) = [\zeta_1^8|\zeta_1^8|\zeta_1^8|\zeta_1^8|\zeta_1^8] $$ 
in the WSS for $\FF_2$, 
and this implies a non-trivial differential
$$ d_1^{wss}(v_1^4 \Delta^6 h_1^2[\zeta^8_1|\zeta^8_1|\zeta_2^4|\zeta_2^4][1])
= v_1^4 \Delta^6 h_1^2 [\zeta_1^8|\zeta_1^8|\zeta_1^8|\zeta_1^8|\zeta_1^8][1] $$
in the WSS for $H(8,v_1^8)$.
\vspace{10pt}

\item[$\pmb{\bou_1^{\otimes 2} \otimes \bou_2}$]  In the cobar complex for $\FF_2[\zeta_1^8, \zeta_2^4]$ we find
$$
 d([\zeta_1^8,\zeta_2^4]|\zeta_1^8\zeta_2^4) \quad \rm{and} \quad d(\zeta_1^8|\zeta_1^8\zeta_2^4|\zeta_2^4 + \zeta_2^4|\zeta_1^8\zeta_2^4|\zeta_1^8)
$$
are linearly independent, and
$$ d([\zeta_1^8,\zeta^4_2]|\zeta_1^8\zeta_2^4+\zeta_1^8\zeta_2^4|[\zeta_1^8,\zeta_2^4]) = 0 $$
However
$$ d(\zeta_1^8\zeta_2^4|\zeta_1^8\zeta_2^4) = [\zeta_1^8,\zeta^4_2]|\zeta_1^8\zeta_2^4+\zeta_1^8\zeta_2^4|[\zeta_1^8,\zeta_2^4] $$
The elements are thus eliminated by multiplying the computations above with $v_1^{-2}v_2^4h^{22}_{2,1}$ and lifting them to the top cell of $H(8,v_1^8)$.
\vspace{10pt}

\item[$\pmb{\bou_1^{\otimes 3}}$]
Note that
$$ \Ext^{10,10+48}_{A_*}(\FF_2) = 0. $$
We conclude that the class
$$ v_1^4c_0h_1(v_0^{-1}v_2^2\zeta_1^8\zeta_2^4) \in \Ext_{A(2)_*}(\bou_2) $$
must either support or be the target of a differential in the algebraic tmf-resolution, for otherwise it would give a non-zero element of $\Ext^{10,10+48}_{A_*}(\FF_2)$.  However, by examination, there are no classes in $\Ext_{A(2)_*}(\FF_2)$ which can kill $v_1^4c_0h_1(v_0^{-1}v_2^2\zeta_1^8\zeta_2^4)$ in the algebraic tmf-resolution, so there must be a non-trivial differential
$$ d_r(v_1^4c_0h_1(v_0^{-1}v_2^2\zeta_1^8\zeta_2^4)) $$
in the algebraic tmf-resolution for $\FF_2$.  Since the target of this differential must be $h_1$-torsion, there is only one possibility:
$$ d_2(v_1^4c_0h_1(v_0^{-1}v_2^2\zeta_1^8\zeta_2^4)) = v_1^4 h_1^2 v_2^2 \zeta_1^8|\zeta_1^8|\zeta_2^4. $$
It follows that we have
$$ d_2(v_1^4c_0(v_0^{-1}v_2^2\zeta_1^8\zeta_2^4)) = v_1^4 h_1 v_2^2 \zeta_1^8|\zeta_1^8|\zeta_2^4. $$
This differential lifts to a differential
$$ d_2(v_1^4c_0(v_0^{-1}v_2^2\zeta_1^8\zeta_2^4)[1]) = v_1^4 h_1 v_2^2 \zeta_1^8|\zeta_1^8|\zeta_2^4[1] $$
in the algebraic tmf-resolution for $H(8,v_1^8)$.  Multiplying by $\Delta^6$, we have
$$ d_2(\Delta^6 v_1^4c_0(v_0^{-1}v_2^2\zeta_1^8\zeta_2^4)[1]) = \Delta^6 v_1^4 h_1 v_2^2 \zeta_1^8|\zeta_1^8|\zeta_2^4[1]. $$

\item[$\pmb{\bou_1 \otimes \bou_2}$]
There is a differential
$$ d^{wss}_0(\zeta_2^{12}) = [\zeta_2^4,\zeta_2^8] $$
in the WSS for $\FF_2$ which lifts to a differential
$$ d^{wss}_0(v_1 h_{2,1}^{21} g(v_0^{-1}v_2^2\zeta_2^{12})) = 
v_1 h_{2,1}^{21} g(v_0^{-1}v_2^2[\zeta_2^4,\zeta_2^8]). $$
We therefore only have to consider one of the two potential elements.
In the modified May spectral sequence (\ref{eq:mmss}), there is a differential
$$ d_8^{mmss}(h_{2,3}) = h_{1,3}h_{1,4} $$
which lifts to a differential
$$ d^{wss}_1(\zeta_2^8) = \zeta_1^8|\zeta_1^{16}. $$
using the map $\Phi_*$ of (\ref{eq:mmss}), and gives a differential
$$ d^{wss}_1(\zeta_2^4|\zeta_2^8) = \zeta_2^4|\zeta_1^8|\zeta_1^{16}. $$
The elements 
$$ v_1 g (v_0^{-1}v_2^2\zeta_2^4|\zeta_2^8) \in \Ext_{A(2)_*}(\bou_1 \otimes \bou_2) $$
and
$$ v_1 g (v_0^{-1}v_2^2\zeta_2^4|\zeta_1^8|\zeta_1^{16}) \in \Ext_{A(2)_*}(\bou_1^{\otimes 2} \otimes \bou_2) $$
support $h_{2,1}$-towers which are non-trivial in $\E{wss}{}_1$.  Therefore we have a non-trivial differential
$$ d_1^{wss}(v_1 h_{2,1}^{21} g(v_0^{-1}v_2^2\zeta_2^4|\zeta_2^8))
= v_1 h_{2,1}^{21} g(v_0^{-1}v_2^2\zeta_2^4|\zeta_1^8|\zeta_1^{16}). $$
This differential lifts to the top cell of $H(8,v_1^8)$ to give
$$ d_1^{wss}(v_1 h_{2,1}^{21} g(v_0^{-1}v_2^2\zeta_2^4|\zeta_2^8)[18])
= v_1 h_{2,1}^{21} g(v_0^{-1}v_2^2\zeta_2^4|\zeta_1^8|\zeta_1^{16})[18] $$
in the WSS for $H(8,v_1^8)$.
\vspace{10pt}

\item[$\pmb{\bou_1^{\otimes 2}}$] 
The element
$$ h_{2,1}^5\Delta^4 v_1 g(v_0^{-1}v_2^2[\zeta_1^8,\zeta_2^4])[18] $$
detects the element 
$$ \Delta^4 \cdot MP\Delta h_0^2e_0[18] $$
in the algebraic tmf-resolution for $H(8,v_1^8)$. 
Regarding this element as an element in the MASS for $\tmf \wedge \bo_1^2$, there is a non-trivial differential
$$ d^{mass}_3(h_{2,1}^5\Delta^4 v_1 g(v_0^{-1}v_2^2[\zeta_1^8,\zeta_2^4])[18]) = h^{24}_{2,1}v_1 g(v_0^{-1}v_2^2[\zeta_1^8,\zeta_2^4])[18]. $$
By applying $(-)^{\wedge_\tmf 2}$ to the map of tmf modules (\ref{eq:bo1bo2}),
we may consider the composite
\begin{equation}\label{eq:bo1^2map}
\tmf \wedge \bo_1^2 \hookrightarrow (\tmf \wedge \br{\tmf}_2)^{\wedge_\tmf 2}
\rightarrow \tmf \wedge \br{\tmf}^2. 
\end{equation}
The differential above maps to a non-trivial differential between elements of the same name in the MASS for $\tmf \wedge \br{\tmf}^2$.  We wish to apply Lemma~\ref{lem:technical}.  We must have 
$$ d^{mass}_2(\Delta^4 \cdot MP\Delta h_0^2e_0[18]) = 0 $$
in the MASS for $M(8,v_1^8)$, since there are no elements in the algebraic tmf-resolution for $H(8,v_1^8)$ which could detect a target for this differential.  Thus Lemma~\ref{lem:technical} implies that either
$$ d_3^{mass}(\Delta^4 \cdot MP\Delta h_0^2e_0[18]) $$
is non-trivial and detected by $h^{24}_{2,1}v_1 g(v_0^{-1}v_2^2[\zeta_1^8,\zeta_2^4])[18]$, or 
$$ h^{24}_{2,1}v_1 g(v_0^{-1}v_2^2[\zeta_1^8,\zeta_2^4])[18] $$
is killed in the algebraic tmf-resolution for $H(8,v_1^8)$, or detects an element which is killed in the MASS for $T^2 \wedge M(8,v_1^8)$.  The only such possibility is for 
$$ \Delta^2 h_{2,1}^{23} \zeta_2^4[17] $$
to detect the source of a $d_2$-differential in the MASS for $T^2 \wedge M(8,v_1^8)$ to do such a killing.
Projecting onto the top Moore space of $M(8,v_1^8)$, this would imply 
$$ \Delta^2 h_{2,1}^{23} \zeta_2^4 $$
detects an element in the algebraic tmf-resolution for the sphere which supports a non-trivial $d_2$-differential in the ASS for the sphere.  However, 
$\Delta^2 h_{2,1}^{23} \zeta_2^4$
detects
$$ \Delta^2 g^5 \cdot \Delta h_2 c_1 $$
in the ASS for the sphere, and there is a differential
\begin{align*}
d_2^{ass}(\Delta^2 g^5 \cdot \Delta h_2 c_1) 
& = d_2^{ass}(\Delta^2 g^2) \cdot g^3 \cdot \Delta h_2 c_1 \\
& = \Delta^2 h_2^2 g^2 e_0 \cdot g^3 \cdot \Delta h_2 c_1.
\end{align*}
However $\Delta^2 h_2^2 e_0 \cdot \Delta h_2 c_1 = 0$ in $\Ext_{A_*}(\FF_2)$ \cite{Bruner}, so this $d_2^{ass}$ is zero.
$$ \quad $$
We now turn our attention to the other potential target coming from $\bou_1^{\otimes 2}$:
$$ h_{2,1}^{15} \Delta^2 g (v_0^{-1}v_2^2[\zeta_1^8,\zeta_2^4])[18]. $$
This element detects
$$ \Delta^2g^2 v_1^6 h_{2,1} Mg^3[0] $$
in the algebraic tmf-resolution for $M(8,v_1^8)$.
However, in the ASS for the sphere, $v_1^6 h_{2,1} g^3$ is a $d_2$-cycle, and so there is a differential
\begin{align*}
d^{ass}_2(\Delta^2 g^2 \cdot v_1^6 h_{2,1} g^3) 
& = d^{ass}_2(\Delta^2 g^2) \cdot v_1^6 h_{2,1} g^3 \\ 
& = \Delta^2 h_2^2 g^2 e_0 \cdot v_1^6 h_{2,1} g^3 \\
& =  v_1^7 h_{2,1}^{22} g^2.
\end{align*}
Applying $M(-) = \bra{-,h_0^3,g_2}$, and mapping under the inclusion of the bottom cell of $M(8,v_1^8)$, we get a non-trivial differential
$$ d^{mass}_2(\Delta^2 g^2 \cdot v_1^6 h_{2,1} Mg^3[0]) = v_1^7 h_{2,1}^{22} M g^2[0]. $$

\item[$\pmb{\bou_1}$] 
The element
$$h_{2,1}^{31}g(h_{2,1}\zeta_2^4)$$
detects 
$$ g^8 n \in \Ext_{A_*}(\FF_2) $$
in the algebraic tmf-resolution for $\FF_2$ (Prop.~\ref{prop:h21detection}).  
This element can be eliminated by Case (1) of Remark~\ref{rmk:technical}, but we can also handle it manually using low dimensional calculations in the ASS for the sphere.
There is a differential
$$ d_3 (mQ_2) = g^3n $$
in the ASS for the sphere \cite{IWX2}, from which it follows that $g^8 n$ is zero on the $E_4$-page of the ASS of the sphere, and hence $g^8 n[0]$ is zero on the $E_4$-page of the MASS for $M(8,v_1^8)$. 
$$ \quad $$
For the the element
$$ h_{2,1}^{18} \Delta^2 g(h_{2,1}\zeta_2^4)[17] $$
we wish to employ Case (1) of Remark~\ref{rmk:technical}, using the differential
$$ d^{mass}_3(h^{15}_{2,1}\Delta^2 g(h_{2,1}v_0^{-2}v_2^2\zeta_1^{16})[17]) = h_{2,1}^{18} \Delta^2 g(h_{2,1}\zeta_2^4)[17] $$
in the MASS for $\tmf \wedge \br{\tmf} \wedge M(8,v_1^8)$.
Note that 
$$ h^{15}_{2,1}\Delta^2 g(h_{2,1}v_0^{-2}v_2^2\zeta_1^{16})[17] $$
detects the element
$$ C'' \cdot \Delta^2 g^2[17] $$
in the algebraic tmf-resolution.
Observe that  we have \cite{IsaksenChart}, \cite{Bruner}
\begin{align*}
d_2(C'' \cdot \Delta^2g^2) 
& = C'' \cdot d_2(\Delta^2 g^2) \\
& = g^2 \cdot C'' \Delta h_2^2 e_0 \\
& = g^2 \cdot 0 = 0.
\end{align*} 
It follows that $d_2(C'' \cdot \Delta^2g^2[17])$ is in the image of the map
$$ \Ext_{A_*}(H(8)) \rightarrow \Ext_{A_*}(H(8,v_1^8)) $$
but a check of the algebraic tmf-resolution for $H(8,v_1^8)$ reveals there are no possible targets in this bidegree.
We therefore have
$$ d_2(C'' \cdot \Delta^2 g^2[17]) = 0. $$
Therefore the hypotheses of Lemma~\ref{lem:technical} are satisfied.
It follows that 
$$ h_{2,1}^{18} \Delta^2 g(h_{2,1}\zeta_2^4)[17] $$
is either killed in the algebraic tmf-resolution for $H(8,v_1^8)$, 
or detects an element in the MASS which is killed by $d_3(C'' \cdot \Delta^2g^2[17])$,
 or detects an element which killed by a $d_2$-differential in the MASS for $T^1 \wedge M(8,v_1^8)$.  We just need to eliminate this last possibility.
 
 Any possible source for such a $d_2$-differential would necessarily be detected on the $0$-line of the algebraic tmf-resolution, and would not support a non-trivial $d_2$ in the MASS for $\tmf \wedge M(8,v_1^8)$.  The only such possibility is
$$ \Delta^4 h_{21}^{19}[1]. $$
However, we can express this element as the Hurewicz image of 
the element 
$$ gm \cdot \Delta^4 \cdot g^2[1] $$
in the MASS for $M(8,v_1^8)$.  This element is therefore necessarily a $d_2$-cycle, since it is a product of $d_2$-cycles.
\vspace{10pt}

\item[$\pmb{\bou_2}$] We begin with the element 
$$ h_{2,1}^5 \Delta^4 g(h_{2,1} v_0^{-2}v_2^2 \zeta_1^{16})[18] $$
which detects the element
$$ \Delta^4 gQ_2[18] $$
in the MASS for $M(8,v_1^8)$.  We are in Case (1) of Remark~\ref{rmk:technical}.
An elementary check using the charts of \cite{IsaksenChart} reveals that the element $gQ_2$ in the ASS for the sphere lifts to a $d_2$-cycle
$$ gQ_2[18] $$
supported by the top cell of $H(8,v_1^8)$.  Since $\Delta^4$ is a $d_2$-cycle in the MASS for $M(8,v_1^8)$, we deduce that
$$ \Delta^4 gQ_2[18] $$
is a $d_2$-cycle.
We therefore deduce that either 
$$ d_3^{mass}(\Delta^4gQ_2[18]) $$
is detected by 
$$ \Delta^4 h_{2,1}^{8} g(h_{2,1} \zeta_2^4)[18] + h_{2,1}^{24} g(h_{2,1} v_0^{-2}v_2^2 \zeta_1^{16}) $$
in the algebraic tmf-resolution for $H(8,v_1^8)$ or 
$$ \Delta^4 h_{2,1}^{8} g(h_{2,1} \zeta_2^4)[18] + h_{2,1}^{24} g(h_{2,1} v_0^{-2}v_2^2 \zeta_1^{16}) $$
is killed in the algebraic tmf-resolution for $H(8,v_1^8)$, or detects an element which is killed in the MASS for $T^1 \wedge M(8,v_1^8)$.  The only possible sources of such algebraic tmf-resolution differentials are wedge elements coming from $\Ext_{A(2)_*}(H(8,v_1^8))$, and we know these all must be permanent cycles in the algebraic tmf-resolution because they detect the corresponding wedge elements of $\Ext_{A_*}(H(8,v_1^8))$.  The only elements of the algebraic tmf-resolution which can detect an element which could support a $d_2$-differential killing 
$$ \Delta^4 h_{2,1}^{8} g(h_{2,1} \zeta_2^4)[18] + h_{2,1}^{24} g(h_{2,1} v_0^{-2}v_2^2\zeta_1^{16}) $$
in the MASS for $T^1 \wedge M(8,v_1^8)$ are the elements
\begin{equation}\label{eq:bo2elts}
\Delta^2 v_1^6 h_{2,1}^{23} [0] \quad \mr{and} \quad \Delta^2 v_1^3 h_{2,1}^{24} [1].
\end{equation}
However, using the map of spectral sequences
$$ \E{mass}{}_2^{*,*}(T^1 \wedge M(8,v_1^8)) \rightarrow \E{mass}{}_2^{*,*}(\tmf \wedge M(8,v_1^8)) $$
we can eliminate these possibilities on the basis that the elements (\ref{eq:bo2elts}) support non-trivial $d_2$ differentials in the MASS for $M(8,v_1^8)$.
$$ \quad $$
We are left with eliminating 
$$ v_1^2 h_{2,1}^{31}(v_0^{-1}v_2^2\zeta_1^8\zeta_2^4)[1] $$
as possibly detecting $d_5^{mass}(v_2^{32})$ in the MASS for $M(8,v_1^8)$.  This is the trickiest obstruction to eliminate.  In the MASS for $\tmf \wedge \br{\tmf} \wedge M(8,v_1^8)$ there is a differential
$$ d^{mass}_2(\Delta^2 v_1 h_{2,1}^{22}(v_0^{-1}v_2^2\zeta_1^8\zeta_2^4)[1]) = 
v_1^2 h_{2,1}^{31}(v_0^{-1}v_2^2\zeta_1^8\zeta_2^4)[1]. $$
The \emph{problem} is that in the WSS for $H(8,v_1^8)$ there is a non-trivial differential
$$ d^{wss}_0(\Delta^2 v_1 h_{2,1}^{22}(v_0^{-1}v_2^2\zeta_1^8\zeta_2^4)[1])
= \Delta^2 v_1 h_{2,1}^{22}(v_0^{-1}v_2^2[\zeta_1^8,\zeta_2^4])[1]. $$

\begin{sublem}\label{sublem:v232T1}
The element $v_2^{32}$ is a permanent cycle in the MASS for $T^1 \wedge M(8,v_1^8)$.
\end{sublem}

\begin{proof}[Proof of sublemma]
The elements of the algebraic tmf-resolution which could possibly detect the target of a differential 
$$ d_r^{mass}(v_2^{32}), \quad r \ge 4, $$
in the MASS for $T^1 \wedge M(8,v_1^8)$ consist of those terms in Table~\ref{tab:v232targets} coming from $\bou_1$ and $\bou_2$.

Using (\ref{eq:tmfinbo2}) there is a map
$$ \Sigma^{31}\tmf \wedge M(8,v_1^8) \rightarrow \Sigma^{-1} \tmf \wedge \br{\tmf} \rightarrow T^1 $$
and we therefore have a differential
$$ d^{mass}_2(\Delta^2 v_1 h_{2,1}^{22}(v_0^{-1}v_2^2\zeta_1^8\zeta_2^4)[1]) = 
v_1^2 h_{2,1}^{31}(v_0^{-1}v_2^2\zeta_1^8\zeta_2^4)[1] $$
in the MASS for $T^1 \wedge M(8,v_1^8)$.
Therefore $v_1^2 h_{2,1}^{31}(v_0^{-1}v_2^2\zeta_1^8\zeta_2^4)[1]$ cannot be the target of a differential $d_5^{mass}(v_2^{32})$ in the MASS for $T^1 \wedge M(8,v_1^8)$.

Our previous arguments eliminate all the other possibilities.
\end{proof}

Suppose now for the purpose of generating a contradiction that the differential
$$ d_5^{mass}(v_2^{32}) $$ 
in the MASS for $M(8,v_1^8)$ is non-trivial and detected by $v_1^2 h_{2,1}^{31}(v_0^{-1}v_2^2\zeta_1^8\zeta_2^4)[1]$ in the algebraic tmf-resolution for $H(8,v_1^8)$.  
Consider the fiber sequence
$$ \Sigma^{-2} \br{\tmf}^{2} \wedge M(8,v_1^8) \rightarrow M(8,v_1^8) \rightarrow T^1 \wedge M(8,v_1^8) \xrightarrow{\partial} \Sigma^{-1} \br{\tmf}^2. $$
We have proven that $v_2^{32}$ exists in $\pi_{192}T^1 \wedge M(8,v_1^8)$, and because our assumption implies that $v_2^{32}$ does not lift to $\pi_{192}M(8,v_1^8)$, we must have
$$ 0 \ne \partial(v_2^{32}) \in \pi_{191}\Sigma^{-2} \br{\tmf}^2 \wedge M(8,v_1^8).
$$
\begin{sublem}
There exists a choice of $v_2^{32} \in \pi_{192}T^1 \wedge M(8,v_1^8)$ so that
$\partial(v_2^{32})$ has modified Adams filtration $34$.
\end{sublem}
\begin{proof}[Proof of sublemma]
Let $X^{\bra{k}}$ denote the $k$th modified Adams cover of $X$ - so that the MASS for $X^{\bra{k}}$ is the truncation of the MASS for $X$ obtained by only considering terms in $\E{mass}{}^{s,t}_2(X)$ for $s \ge k$, and let $X_{\bra{k}}$ denote the cofiber
$$ X^{\bra{k+1}} \rightarrow X \rightarrow X_{\bra{k}} $$ 
Then we have fiber sequences
$$ M(8,v_1^8)_{\bra{k}} \rightarrow (T^1 \wedge M(8,v_1^8))_{\bra{k}} 
\rightarrow (\Sigma^{-1}\br{\tmf}^2 \wedge M(8,v_1^8))_{\bra{k-2}}. $$
Define $\td{M}_{\bra{k}}$ to be the homotopy pullback
$$
\xymatrix{
\td{M}_{\bra{k}} \ar[r] \ar[d] &
T^1 \wedge M(8,v_1^8) \ar[d] \\
M(8,v_1^8)_{\bra{k}} \ar[r] &
(T^1 \wedge M(8,v_1^8))_{\bra{k}}
}
$$
Then the algebraic tmf-resolution for $\td{M}_{\bra{k}}$ is the truncation of the algebraic tmf-resolution for $M(8,v_1^8)$ obtained by omitting, for $n \ge 2$ all terms of
$$ \Ext_{A(2)_*}(\bou_{i_1} \otimes \cdots \otimes \bou_{i_n} \otimes H(8,v_1^8)) $$
of cohomological degree greater than $k-n$.
It follows from the map of algebraic tmf-resolutions and MASS's associated to the map
$$ M(8,v_1^8) \rightarrow \td{M}_{\bra{k}} $$
that there is a differential
$$ d^{mass}_5(v_2^{32}) = v_1^2 h_{2,1}^{31}(v_0^{-1}v_2^2\zeta_1^8\zeta_2^4)[1] $$
in the MASS for $\td{M}_{\bra{k}}$.  This differential is non-trivial in the MASS for $\td{M}_{\bra{36}}$, because it is non-trivial in the MASS for $M(8,v_1^8)$, and any intervening differentials killing the target in the algebraic tmf-resolution or MASS for $\td{M}_{\bra{36}}$ would lift to $M(8,v_1^8)$ because the spectral sequences are isomorphic in the relevant range.
The same is not true in the case of $\td{M}_{\bra{35}}$, where 
$$ d^{wss}_0(\Delta^2 v_1 h_{2,1}^{22}(v_0^{-1}v_2^2\zeta_1^8\zeta_2^4)[1])
= 0 $$
and therefore $\Delta^2 v_1 h_{2,1}^{22}(v_0^{-1}v_2^2\zeta_1^8\zeta_2^4)[1]$ persists to the $E_2$-term of the MASS
$$ d^{mass}_2(\Delta^2 v_1 h_{2,1}^{22}(v_0^{-1}v_2^2\zeta_1^8\zeta_2^4)[1]) = 
v_1^2 h_{2,1}^{31}(v_0^{-1}v_2^2\zeta_1^8\zeta_2^4)[1]. $$
Therefore the proof of Sublemma~\ref{sublem:v232T1} goes through with $T^1 \wedge M(8,v_1^8)$ replaced with $\td{M}_{\bra{35}}$ to show that there exists an element 
$$ \td{v_2}^{32} \in \pi_{192}\td{M}_{\bra{35}} $$
which is detected by $v_2^{32}$ in the MASS.
Consider the diagram
$$
\xymatrix{
& & \Sigma^{-1} \br{\tmf}^2 \ar[d] \\
\td{M}_{\bra{36}} \ar[d] \ar[r] & 
T^1 \wedge M(8,v_1^8) \ar[r]_-{\partial'} \ar[ru]^{\partial} \ar@{=}[d] & 
(\Sigma^{-1} \br{\tmf}^2 \wedge M(8,v_1^8))_{\bra{34}} \ar[d] \\
\td{M}_{\bra{35}} \ar[r] & 
T^1 \wedge M(8,v_1^8) \ar[r]_-{\partial''} &
(\Sigma^{-1} \br{\tmf}^2 \wedge M(8,v_1^8))_{\bra{33}} 
}
$$
where the rows are cofiber sequences.
The element $\td{v_2^{32}} \in \pi_{192}\td{M}_{\bra{35}}$ maps to an element
$v_2^{32} \in T^1 \wedge M(8,v_1^8)$ with
$$ \partial''(v_2^{32}) = 0. $$
However, since $d_5^{mass}(v_2^{32})$ is non-trivial in the MASS for $\td{M}_{\bra{36}}$, the element $v_2^{32} \in \pi_{192}T^1 \wedge M(8,v_1^8)$ cannot lift to $\td{M}_{\bra{36}}$, and therefore
$$ \partial'(v_2^{32}) \ne 0. $$
It follows that $\partial(v_2^{32})$ has modified Adams filtration $34$.  
\end{proof}

However we have

\begin{sublem}
There are no elements of $\pi_{191}\Sigma^{-2}\br{\tmf}^2 \wedge M(8,v_1^8)$ of modified Adams filtration $34$.
\end{sublem}

\begin{proof}[Proof of sublemma]
The only possible elements in the algebraic tmf-resolution for $\br{\tmf}^{2} \wedge M(8,v_1^8)$ which could contribute to modified Adams filtration 34 in this degree are
\begin{equation}\label{eq:lastdude}
\Delta^2 v_1 h_{2,1}^{22}(v_0^{-1}v_2^2[\zeta_1^8,\zeta_2^4])[1] \in \Ext_{A(2)_*}(\bou_1^{\otimes 2} \otimes H(8,v_1^8))
\end{equation}
and the elements of Table~\ref{tab:v232targets} of algebraic tmf filtration greater than 1 in the appropriate modified Adams filtration.  However, the previous arguments eliminate all of the candidates coming from Table~\ref{tab:v232targets}, so we are left with eliminating (\ref{eq:lastdude}).   We wish to lift the differential
$$ d_3^{mass}(\Delta^6 v_1 h_{2,1}^{3}(v_0^{-1}v_2^2[\zeta_1^8,\zeta_2^4])[1])
= \Delta^2 v_1 h_{2,1}^{22}(v_0^{-1}v_2^2[\zeta_1^8,\zeta_2^4])[1] $$
in the MASS for $\tmf \wedge \br{\tmf}^2 \wedge M(8,v_1^8)$ to a differential in the MASS for $\br{\tmf}^2 \wedge M(8,v_1^8)$.  We therefore must argue that 
$$ d_2^{mass}(\Delta^6 v_1 h_{2,1}^{3}(v_0^{-1}v_2^2[\zeta_1^8,\zeta_2^4])[1]) = 0 $$
in the MASS for $\br{\tmf}^2 \wedge M(8,v_1^8)$.   
We will therefore argue there are no elements in the algebraic tmf-resolution for $\br{\tmf}^2 \wedge M(8,v_1^8)$ which could detect the target of such a $d_2$.
Ignoring any possibilities which are eliminated by Proposition~\ref{prop:h21deathH38}, the only possibilities are
\begin{equation*}
\begin{split}
\Delta^6 v_1^4 h_1 v_0^{-1}v_2^2\zeta_1^8|[\zeta_1^8,\zeta_2^4][1], \\
\Delta^6 v_1^4 h_1 v_0^{-1}v_2^2[\zeta_1^8,\zeta_2^4]|\zeta_1^8[1], \\
\Delta^6 v_1^4 h_0^2 [\zeta_1^8,\zeta_2^4]|\zeta_1^8|\zeta_1^8|\zeta_1^8[0], \\
\Delta^6 v_1^4 h_0^2 \zeta_1^8|[\zeta_1^8,\zeta_2^4]|\zeta_1^8|\zeta_1^8[0], \\
\Delta^6 v_1^4 h_0^2 \zeta_1^8|\zeta_1^8|[\zeta_1^8,\zeta_2^4]|\zeta_1^8[0], \\
\Delta^6 v_1^4 h_0^2 \zeta_1^8|\zeta_1^8|\zeta_1^8|[\zeta_1^8,\zeta_2^4][0].
\end{split}
\end{equation*}
However, these are killed by the respective WSS differentials:
\begin{equation*}
\begin{split}
d_0^{wss}\Delta^6 v_1^4 h_1 v_0^{-1}v_2^2\zeta_1^8|\zeta_1^8\zeta_2^4[1], \\
d_0^{wss}\Delta^6 v_1^4 h_1 v_0^{-1}v_2^2\zeta_1^8\zeta_2^4|\zeta_1^8[1], \\
d_0^{wss}\Delta^6 v_1^4 h_0^2
\zeta_1^8\zeta_2^4|\zeta_1^8|\zeta_1^8|\zeta_1^8[0], \\
d_0^{wss}\Delta^6 v_1^4 h_0^2 \zeta_1^8|\zeta_1^8\zeta_2^4|\zeta_1^8|\zeta_1^8[0], \\
d_0^{wss}\Delta^6 v_1^4 h_0^2 \zeta_1^8|\zeta_1^8|\zeta_1^8\zeta_2^4|\zeta_1^8[0], \\
d_0^{wss}\Delta^6 v_1^4 h_0^2 \zeta_1^8|\zeta_1^8|\zeta_1^8|\zeta_1^8\zeta_2^4[0].
\end{split}
\end{equation*}
\end{proof}
Thus we have arrived at a contradiction, as we have produced an element of modified Adams filtration 34, and subsequently showed no such elements exist. We conclude that our supposition, that the differential $d_5^{mass}(v^{32}_2)$ in the MASS for $M(8,v_1^8)$ is non-trivial and detected by $v_1^2 h_{2,1}^{31}(v_0^{-1}v_2^2\zeta_1^8\zeta_2^4)[1]$ in the algebraic tmf-resolution, is false. 

\end{description}
\end{proof}

\section{Determination of elements not in the tmf Hurewicz image}\label{sec:J}

\begin{thm}\label{thm:nothi}
The elements of $\tmf_*$ not in the subgroup described in Theorem~\ref{thm:main} are not in the Hurewicz image.
\end{thm}

We first recall some well known K-theory computations.  Recall that $\pi_*\KO$ is given by the following $v_1^4$-periodic pattern:
\begin{center}
\includegraphics[width=0.5\linewidth]{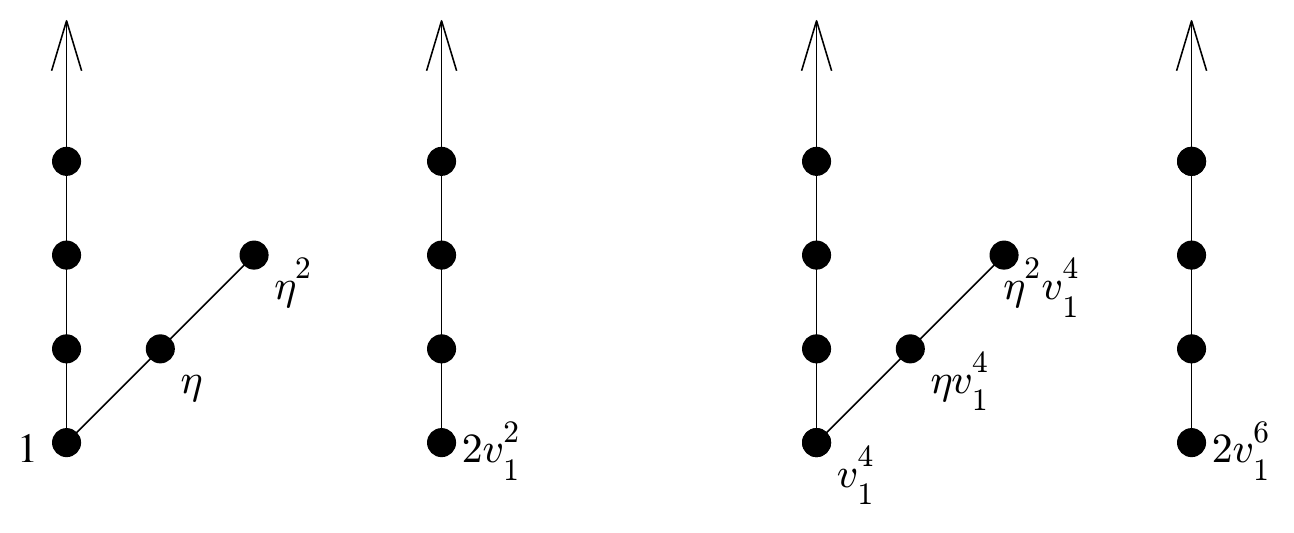}
\end{center}
Let
$$ M(2^\infty) := \varinjlim_i M(2^i) $$
denote the Moore spectrum for $\ZZ/2^{\infty}$.  Consider the following diagram
of cofiber sequences:
\begin{equation}\label{eq:KOM2}
\xymatrix{
\Sigma^{-1}\KO \wedge M(2) \ar[r]^-p \ar[d]_{\cdot 2^{-1}} &
\KO \ar[r]^{\cdot 2} \ar@{=}[d] &
\KO \ar[r]^-{\overline{(\cdot)}} \ar[d]^{\cdot 2^{-1}} &
\KO \wedge M(2) \ar[d]^{\cdot 2^{-1}} 
\\
\Sigma^{-1}\KO \wedge M(2^\infty) \ar[r]_-p  &
\KO \ar[r] &
\KO_\QQ \ar[r]_-{\overline{(\cdot)}} &
\KO \wedge M(2^\infty)     
}
\end{equation}
The groups $\KO_*M(2)$ are well-known to be given by the following $v^4_1$-periodic pattern:
\begin{center}
\includegraphics[width=0.5\linewidth]{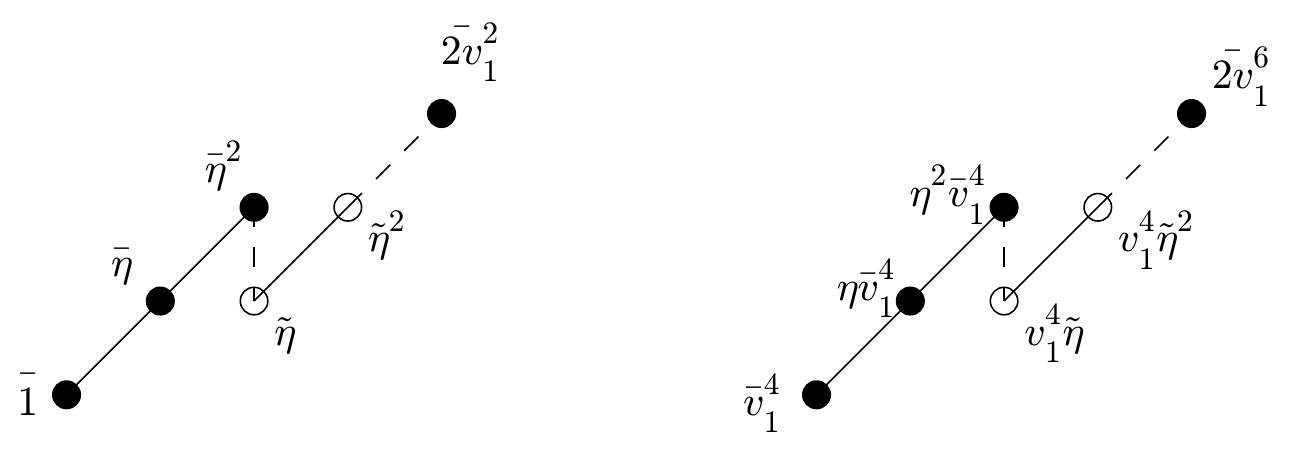}
\end{center}
where we denote lifts of elements of $\KO_*$ along the map $p$ of Diagram~(\ref{eq:KOM2}) with a tilde, and the images of the map $\overline{(\cdot)}$ with a bar.
It then follows easily from the map of long exact sequences coming from the above diagram that $\KO_*M(2^\infty)$ is given by the $v_1^4$-periodic pattern
\begin{center}
\includegraphics[width=0.5\linewidth]{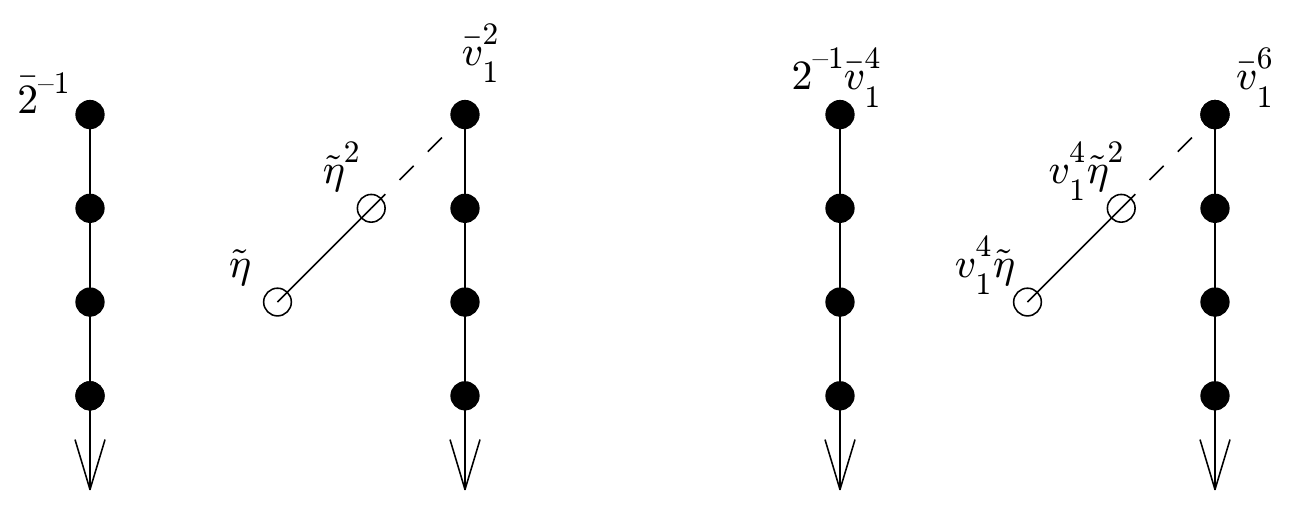}
\end{center}
where again we denote lifts over the map $p$ with a tilde, and images under the map $\overline{(\cdot)}$ with a bar.  The infinite sequences of dots going down represent the elements $2^{-i}$ in $\ZZ/2^{\infty} = \QQ/\ZZ_{(2)}$.

\begin{proof}[Proof of Theorem~\ref{thm:nothi}]
Recall that we have an equivalence \cite[Cor.~3]{Laures} 
$$ c_4^{-1}\tmf \simeq \KO[j^{-1}] $$
where $j^{-1} = \Delta/c_4^3$.  Applying $\pi_0$ to this equivalence, we have a commutative diagram
$$
\xymatrix{
S \ar[r] \ar[d] & 
\KO \ar@{^{(}->}[dr] 
\\
\tmf \ar[r] &
c_4^{-1}\tmf \ar[r]_\simeq & 
\KO[j^{-1}].
}
$$
Consider the following diagram
$$
\xymatrix{
\pi_*S \ar[d]_h &
\pi_{*+1} M(2^\infty) \ar[l]_p \ar[d]_h \ar[r] &
\KO_{*+1} M(2^\infty) \ar@{^{(}->}[dr]^i 
\\
\tmf_* \ar[d]_L &
\tmf_{*+1}M(2^\infty) \ar[l] \ar[r]^{L} &
c_4^{-1} \tmf_{*+1}M(2^\infty) \ar[dll]^{p'} \ar@{=}[r] &
\KO_{*+1}M(2^\infty)[j^{-1}].
\\
c_4^{-1}\tmf_*
}
$$
Suppose that $x \in \tmf_{>0}$ has non-trivial image in $L(x) \in c_4^{-1}\tmf_*$, and suppose that $x = h(y)$.  Since $y$ is torsion, it lifts over ${p}$ to an element 
$$ \td{y} \in \pi_{*+1}M(2^\infty) $$
The commutativity of the diagram, implies that
$$ 0 \ne L(x) \in \mr{Im}(p' \circ i) $$
and this implies that
$$ L(x) \in \{c_4^k\eta^l \: : \: k \ge 0, l \in \{1,2\}\}. $$
Now consider elements of the form
$$ x = \alpha\Delta^k\nu \in \tmf_* $$
with $\alpha \not\equiv 0 \mod 8$.  Suppose that $x = h(y)$.  Lift $y$ to an element 
$$ \td{y} \in \pi_{*+1}M(2^\infty). $$
Then we have
$$ Lh(\td{y}) = \frac{\br{\alpha\Delta^k v_1^2}}{8} =  \frac{\br{\alpha v_1^{12k+2}}}{4} j^{-k} \ne 0. $$
But the commutativity of the diagram implies that $Lh(\td{y})$ is in the image of $i$, which implies that $k = 0$. 
\end{proof}

\section{Lifting the remaining elements of $\tmf_*$ to $\pi_*^s$.}\label{sec:tmfhi}

\subsection*{Multiplicative generators of the Hurewicz image below the $192$-stem}

In this section, we determine a set of elements which multiplicatively generate the tmf-Hurewicz image the below the $192$-stem. The results in this section drastically reduce the number of classes which we must lift in the sequel. 

\begin{lem}\label{Lem:Eoo}
The Hurewicz map $S \to \tmf$ is a map of ring spectra. In particular, it preserves multiplication. 
\end{lem}

\begin{cor}
Suppose $\alpha = \beta \gamma$ is a product of elements $\beta, \gamma \in \pi_*(\tmf)$ with lifts $\tilde{\beta}, \tilde{\gamma} \in \pi_*(S)$. Then $\tilde{\beta} \tilde{\gamma} \in \pi_*(S)$ must be a lift of $\alpha$.
\end{cor}
 

With this in mind, it suffices to find a subset of the Hurewicz image which generates the entire Hurewicz image up to the $192$-stem under products. Our desired generating subset is given in Corollary \ref{Cor:Gens}.  We will obtain our generating set by listing generators in lemmas and then recording their products in corollaries, until we have exhausted the tmf-Hurewicz image up to stem 192. 

\begin{lem}
The classes $\eta \in \pi_1(\tmf)$, $\nu \in \pi_3(\tmf)$,
$\epsilon \in \pi_8(\tmf)$, $\kappa \in \pi_{14}(\tmf)$, $\bar{\kappa} \in \pi_{20}(\tmf)$, $u \in \pi_{39}(\tmf)$, and $w \in \pi_{45}(\tmf)$
are in the Hurewicz image.
\end{lem}

\begin{proof}
The elements $\eta$, $\nu$, $\epsilon$, $\kappa$, $\bar{\kappa}$, $u$, and $w$ are all well known elements of $\pi_*^s$, detected in the Adams spectral sequence by $h_1$, $h_2$, $c_0$, $d_0$, $g$, $\Delta h_1 d_0$, and $\Delta h_1 g$
\cite[Table 8]{Isaksen}.  These elements have non-trivial images under the map of Adams spectral sequences induced by the unit map $S \to \tmf$.  The lemma is therefore somewhat tautological, as the corresponding elements in $\tmf$ were \emph{defined} in Section~\ref{sec:intro} to be the Hurewicz images of these elements. 
\end{proof}

\begin{lem}
The class $q \in \pi_{32}(\tmf)$ is in the Hurewicz image.
\end{lem}

\begin{proof}
See the proof of Lemma~\ref{Lem:quwLift}(1).
\end{proof}

\begin{cor}
The classes $\eta^2 \in \pi_2(\tmf)$, $\nu^2 \in \pi_6(\tmf)$, $\nu^3 = \epsilon \eta \in \pi_9(\tmf)$, $\kappa \eta \in \pi_{15}(\tmf)$, $\kappa \nu \in \pi_{17}(\tmf)$,  $\bar{\kappa}\eta \in \pi_{21}(\tmf)$, $\bar{\kappa} \eta^2 = \kappa \epsilon \in \pi_{22}(\tmf)$, $\bar{\kappa} \epsilon = \kappa^2 \in \pi_{28}(\tmf)$, $q \eta \in \pi_{33}(\tmf)$, $\bar{\kappa} \kappa \in \pi_{34}(\tmf)$, $\bar{\kappa} \kappa \eta \in \pi_{35}(\tmf)$, $\bar{\kappa}^2 \in \pi_{40}(\tmf)$, $\bar{\kappa}^2 \eta \in \pi_{41}(\tmf)$, $\bar{\kappa}^2 \eta^2 = \kappa^3 \in \pi_{42}(\tmf)$, $w \eta \in \pi_{46}(\tmf)$,
$\bar{\kappa} q \in \pi_{52}(\tmf)$, $\bar{\kappa} q \eta \in \pi_{53}(\tmf)$,
 $\bar{\kappa}^2 \kappa \in \pi_{54}(\tmf)$, $\bar{\kappa} u \in \pi_{59}(\tmf)$, $\bar{\kappa}^3 \in \pi_{60}(\tmf)$, $\bar{\kappa} w \in \pi_{65}(\tmf)$, $\bar{\kappa} w \eta \in \pi_{66}(\tmf)$, $\bar{\kappa}^4 \in \pi_{80}(\tmf)$, $\bar{\kappa}^2 w \in \pi_{85}(\tmf)$, $w^2 \in \pi_{90}(\tmf)$, $\bar{\kappa}^5 \in \pi_{100}(\tmf)$, $\bar{\kappa}^3 w \in \pi_{105}(\tmf)$, $\bar{\kappa} w^2 \in \pi_{110}(\tmf)$, $\bar{\kappa}^4 w \in \pi_{125}(\tmf)$, and $\bar{\kappa}^2 w^2 \in \pi_{130}(\tmf)$ are in the Hurewicz image.
\end{cor}

\begin{lem}\label{Lem:v28nu2}
The classes  $\{ \nu \Delta^2 \} \nu \in \pi_{54}(\tmf)$, $\{\nu \Delta^2\} \kappa \in \pi_{65}(\tmf)$, and $\{ \eta^2 \Delta^2\} \bar{\kappa} \in \pi_{70}(\tmf)$ are in the Hurewicz image.
\end{lem}

\begin{proof}
See Lemma \ref{Lem:v28Lift}. 
\end{proof}

\begin{cor}\label{Cor:v28nu2}
The classes $\{ \nu \Delta^2 \} \nu^2 \in \pi_{57}(\tmf)$ and $\{ \nu \Delta^2\} \kappa \nu \in \pi_{68}(\tmf)$ are in the Hurewicz image.
\end{cor}



\begin{lem}\label{Lem:v216neekb}
The classes $\{ \nu \Delta^4\} \nu \in \pi_{102}(\tmf)$, $\{ \epsilon \Delta^4\} \in \pi_{104}(\tmf)$, $\{ \kappa \Delta^4 \} \in \pi_{110}(\tmf)$, $2 \Delta^4 \bar{\kappa} \in \pi_{116}(\tmf)$, and $\{\eta \Delta^4 \} \bar{\kappa} \in \pi_{117}(\tmf)$ are in the Hurewicz image.
\end{lem}

\begin{proof}
See Lemmas \ref{Lem:v216neLift} and \ref{Lem:v216kbe}. 
\end{proof}


\begin{cor}\label{Cor:v216low}
The classes $\{\epsilon \Delta^4\} \eta \in \pi_{105}(\tmf)$, $\{ \kappa \Delta^4 \} \eta \in \pi_{111}(\tmf)$, $\{ \kappa \Delta^4 \} \nu \in \pi_{113}(\tmf)$, $\{ \kappa \Delta^4 \} \nu^2 \in \pi_{116}(\tmf)$, $\{\eta \Delta^4 \} \bar{\kappa} \eta \in \pi_{118}(\tmf)$, $\{ \kappa \Delta^4 \} \kappa \in \pi_{124}(\tmf)$, $\{ \kappa \Delta^4 \} \bar{\kappa} \in \pi_{130}(\tmf)$, $\{ \kappa \Delta^4 \} \bar{\kappa} \eta \in \pi_{131}(\tmf)$, $\{\eta \Delta^4 \} \bar{\kappa}^2 \in \pi_{137}(\tmf)$, and $\{ \eta \Delta^4\} \bar{\kappa}^2 \eta \in \pi_{138}(\tmf)$ are in the Hurewicz image.
\end{cor}

\begin{lem}\label{Lem:v216q}
The class $\{q \Delta^4\} \in \pi_{128}(\tmf)$ is in the Hurewicz image. 
\end{lem}

\begin{proof}
See Lemma \ref{Lem:v216qLift}. 
\end{proof}

\begin{cor}\label{Cor:v216q}
The classes $\{ q \Delta^4 \} \eta \in \pi_{129}(\tmf)$, $\{q \Delta^4\}\kappa = w\eta\Delta^4 \in \pi_{142}(\tmf)$, $\{ q \Delta^4 \} \bar{\kappa} \in \pi_{148}(\tmf)$, $\{ q \Delta^4 \} \bar{\kappa} \eta \in \pi_{149}(\tmf)$, $\{ q \Delta^4 \} \bar{\kappa} \eta^2 \in \pi_{150}(\tmf)$ are in the Hurewicz image.
\end{cor}

\begin{lem}\label{Lem:a135}
The class $\Delta^4 u \in \pi_{135}(\tmf)$ is in the Hurewicz image. 
\end{lem}

\begin{proof}
See Lemma \ref{Lem:v216uLift}. 
\end{proof}

\begin{cor}
The classes $\Delta^4 u \eta \in \pi_{136}(\tmf)$ and $\Delta^4 u \bar{\kappa} \in \pi_{155}(\tmf)$ are in the Hurewicz image. 
\end{cor}

%

\begin{lem}\label{Lem:v224nu2}
The classes $\{ \nu \Delta^6 \} \nu \in \pi_{150}(\tmf)$ and $\{\nu \Delta^6 \} \kappa \in \pi_{161}(\tmf)$ are in the Hurewicz image.
\end{lem}

\begin{proof}
See Lemma \ref{Lem:v224nu2Lift}. 
\end{proof}

\begin{cor}\label{Cor:v224nu2}
The classes $\{\nu \Delta^6\} \nu^2 \in \pi_{153}$, $\{ \nu \Delta^6 \} \nu^3 \in \pi_{156}$, $\{\nu \Delta^6 \} \kappa \eta \in \pi_{162}(\tmf)$ and $\{ \nu \Delta^6 \} \kappa \nu \in \pi_{164}(\tmf)$ are in the Hurewicz image.
\end{cor}

%


Thus our calculation of the Hurewicz image up to dimension $192$ has been reduced to showing that the following list of elements is in the Hurewicz image. 

\begin{cor}\label{Cor:Gens}
Up to dimension $192$, the Hurewicz image is generated under multiplication by 
\begin{align*}
\{ & \eta, \nu, \epsilon, \kappa, \bar{\kappa}, q, u, w, \{\nu \Delta^2 \} \nu, \{\nu \Delta^2 \} \kappa, \{\eta^2 \Delta^2 \} \bar{\kappa}, \{\nu \Delta^4 \} \nu, \{ \epsilon \Delta^4\}, \\
& \{ \kappa \Delta^4 \}, 2 \Delta^4 \bar{\kappa}, \{\eta \Delta^4 \} \bar{\kappa}, \{q \Delta^4\},  \Delta^4 u, \{\nu \Delta^6 \} \nu , \{ \nu \Delta^6 \} \kappa \}.
\end{align*}
\end{cor}

\subsection*{Lifting generators}

We will now describe our method for lifting generators.
Given an element $x \in \tmf_*$, we want to lift it to an element $y \in \pi_*^s$.  To this end, we consider the diagram of (M)ASS's:
$$
\xymatrix{
& \Ext_{A(2)_*}(H(8,v_1^8)) \ar@{=>}[rr] \ar[dd]
&& \tmf_{*+18}M(8,v_1^8) \ar[dd] \\
\Ext_{A_*}(H(8,v_1^8)) \ar@{=>}[rr] \ar[ru] \ar[dd]  
&& \pi_{*+18} M(8,v_1^8) \ar[ru] \ar[dd] & \\
& \Ext_{A(2)_*}(\FF_2) \ar@{=>}[rr] 
&& \tmf_* \\
\Ext_{A_*}(\FF_2) \ar@{=>}[rr] \ar[ru]
&& \pi_*^s \ar[ru]
}
$$ 
First, we identify an element
$$ x' \in \Ext_{A(2)_*}(\FF_2) $$
which detects the element $x$ in the ASS for $\tmf_*$, and then we identify an element 
$$ \td{x}' \in \Ext_{A(2)_*}(H(8,v_1^8)) $$
which maps to it.  This element $\td{x}'$ can be regarded as an element of the zero line of the algebraic $\tmf$-resolution for $\Ext_{A_*}(H(8,v_1^8))$.  We will show that the element $\td{x}'$ is a permanent cycle in the algebraic $\tmf$-resolution, and thus lifts to an element
$$ \td{y}' \in \Ext_{A_*}(H(8,v_1^8)). $$
We will then show that the element $\td{y}'$ is a permanent cycle in the MASS for $M(8,v_1^8)$, and hence detects an element
$$ \td{y} \in \pi_*M(8,v_1^8). $$
Let $y \in \pi^s_*$ be the projection of $\td{y}$ to the top cell. 
It then follows that the image of $y$ in $\tmf_*$ equals $x$, modulo terms of higher Adams filtration (AF).  Furthermore, using the $v_2^{32}$-self map on $M(8,v_1^8)$, we deduce that the element
$$ v_2^{32k}\td{y} \in \pi_*M(8,v_1^8) $$
projects on the top cell to an element $v_2^{32k}y \in \pi_*^s$ whose image in $\tmf_*$ is $\Delta^{8k}x$ modulo terms of higher Adams filtration.  Finally, Theorem~\ref{thm:nothi} eliminates the potential ambiguity caused by elements of higher Adams filtration, since the elements of higher Adams filtration are $v_1^4$-periodic.

We will show all of the generators of Corollary~\ref{Cor:Gens} except for $\eta$, $\nu$ and $\epsilon$ actually come from the top cell of $M(8,v_1^8)$, and thus $v_2^{32}$ periodicity extends our work below dimension 192 to all dimensions.  It turns out that $\nu^2$ and $\epsilon$ do not come from the top cell of $M(8,v_1^8)$.  In order to show that the elements  
$$ \Delta^{8k}\nu^2,  \Delta^{8k}\epsilon \in \pi_*\tmf $$
are in the Hurewicz image, for $k > 0$, we will instead show that $\Delta^8\nu^2$ and $\Delta^8\epsilon$ come from the top cell of $M(8,v_1^8)$ (Lemma~\ref{lem:nu2epsilon}). 

\begin{lem}\label{Lem:kkbLift}
The following classes lift to the top cell of $M(8,v_1^8)$:
\begin{enumerate}
\item $\kappa \in \pi_{14}(\tmf)$, 
\item $\bar{\kappa} \in \pi_{20}(\tmf)$. 
\end{enumerate}
\end{lem}

\begin{proof}
We will check that each element lifts using the AHSS:
\begin{enumerate}

\item Since $\kappa$ is $2$-torsion (and thus $8$-torsion), it lifts to $\kappa[1] \in \pi_{15}(M(8))$. Inspection of \cite[Pg. 3]{IsaksenChart} in stems $31$ and $32$ and AF $\geq 12$ reveals that there are no classes which could detect $v_1^8 \kappa[1]$. Therefore $\kappa[1]$ lifts to $\kappa[18] \in \pi_{32}(M(8,v_1^8))$. 
\vspace{10pt}

\item Since $\bar{\kappa}$ is $8$-torsion, it lifts to $\bar{\kappa}[1] \in \pi_{21}(M(8))$. Inspection of \cite[Pg. 3]{IsaksenChart} in stems $36$ and $37$ and AF $\geq 12$ reveals that there are no classes which could detect $v_1^8 \bar{\kappa}[1]$. Therefore $\bar{\kappa}[1]$ lifts to $\bar{\kappa}[18] \in \pi_{38}(M(8,v_1^8))$. 

\end{enumerate}
\end{proof}

\begin{lem}\label{Lem:quwLift}
The following classes lift to the top cell of $M(8,v_1^8)$:
\begin{enumerate}
\item $q \in \pi_{32}(\tmf)$,
\item $u \in \pi_{39}(\tmf)$,  
\item $w \in \pi_{45}(\tmf)$.
\end{enumerate}
\end{lem}

\begin{proof}
We will check that each element lifts using the Atiyah-Hirzebruch spectral sequence (AHSS).

\begin{enumerate}

\item We begin with $q \in \pi_{32}(\tmf)$, which we will define to be the unique non-trivial $c_4$-torsion class detected by the element 
$$ v_2^4 c_0 \in \Ext^{7,7+32}_{A(2)_*}(\FF_2) $$
in the ASS for $\tmf$. The element $v_2^4 c_0$ does \emph{not} lift to $\Ext_{A_*}$.  
Nevertheless, we claim that there is an element $\td{q} \in \pi^s_{32}$\footnote{The element we are calling $\td{q}\in \pi^s_{32}$ is traditionally called $q$, but we add the tilde to distinguish it from the element we are calling $q$ in $\pi_{32}\tmf$.}detected by the element
$$ \Delta h_1 h_3 \in \Ext^{6,6+32}_{A_*}(\FF_2) $$
in the ASS for the sphere, which maps to $q$ under the $\tmf$ Hurewicz homomorphism.  Our strategy will be to argue that $\td{q}$ and $q$ lift to 
$$ \td{q}[18] \in \pi_{50}M(8,v_1^8) \quad \mr{and} \quad q[18] \in \tmf_{50}M(8,v_1^8) $$ respectively, and that the element which detects $\td{q}[18]$ in the MASS for $M(8,v_1^8)$ maps to the element which detects $q[18]$ in the MASS for $\tmf \wedge M(8,v_1^8)$ under the map
\begin{equation}\label{eq:extmap}
 \Ext_{A_*}(H(8,v_1^8)) \rightarrow \Ext_{A(2)_*}(H(8,v_1^8)).
 \end{equation}

Inspection of \cite[Pg. 3]{IsaksenChart} in stem $32$ and AF $ \geq 7$ reveals that $\td{q}$ is $2$-torsion (and thus $8$-torsion), so $\td{q}$ lifts to $\td{q}[1] \in \pi_{33}(M(8))$. Inspection of \cite[Pg. 3]{IsaksenChart} in stems $48$ and $49$ and AF $\geq 14$ reveals that there are no classes which could detect $v_1^8 \td{q}[1]$. Therefore $\td{q}[1]$ lifts to $\td{q}[18] \in \pi_{50}(M(8,v_1^8))$.  A similar but easier analysis reveals that the lift $q[18]$ exists.

The elements $\Delta h_1 h_3 \in \Ext_{A_*}(\FF_2)$ and $v_2^4c_0 \in \Ext_{A(2)_*}(\FF_2)$ are $h_0$-torsion, and hence lift to elements
\begin{align*}
 \Delta h_1 h_3[1] \in \Ext_{A_*}(H(8)), \\
 v_2^4 c_0[1] \in \Ext_{A(2)_*}(H(8))
 \end{align*}
which detect $\td{q}[1] \in \pi_{33}M(8)$ and $q[1] \in \tmf_{33}M(8)$, respectively, in the MASS. To identify the elements which detect $\td{q}[18]$ and $q[18]$ in the MASS, 
we make use of the Geometric Boundary Theorem \cite[Appendix~A]{goodehp}.\footnote{We are specifically using case $(5)$ of the Geometric Boundary Theorem since the relevant class (denoted $p_*(y)$ in the theorem statement) is a permanent cycle.  We will be using this argument repeatedly in subsequent proofs in this section, and for brevity will simply say ``by the Geometric Boundary Theorem...'' in these subsequent instances.}
The differentials
\begin{align*}
 d_3(v_1^2 h_{2,1}g^2[1]) & = v_1^8 \Delta h_3 h_1[1], \\
 d_4(v_1^2 h_{2,1}g^2[1]) & = v_1^8 v_2^4 c_0[1]
 \end{align*}
in the MASS's for $M(8)$ and $\tmf \wedge M(8)$, respectively, imply that $\td{q}[18] \in \pi_{50}M(8,v_1^8)$ and $q[18] \in \tmf_{50}M(8,v_1^8)$ are detected by 
\begin{align*}
v_1^2 h_{2,1}g^2[1] & \in \Ext_{A_*}(H(8,v_1^8)), \\
v_1^2 h_{2,1}g^2[1] & \in \Ext_{A(2)_*}(H(8,v_1^8)),
\end{align*}
in the MASS's for $M(8,v_1^8)$ and $\tmf \wedge M(8,v_1^8)$, respectively, and the former maps to the latter under the map (\ref{eq:extmap}).
\vspace{10pt}

\item Since $u \in \pi_{39}\tmf$ is detected by an element of $\Ext_{A(2)_*}$ in the image of the map
\begin{equation}\label{eq:extmapF2}
 \Ext_{A_*}(\FF_2) \rightarrow \Ext_{A(2)_*}(\FF_2)
\end{equation}
 we immediately see that the element $u \in \pi_{39}(S)$ maps to it. We are left with lifting $u \in \pi^s_{39}$ to the top cell of $M(8,v_1^8)$.  Inspection of \cite[Pg. 3]{IsaksenChart} in stem $39$ and AF $ \geq 10$ reveals that $u$ is $2$-torsion (and thus $8$-torsion), so $u$ lifts to $u[1] \in \pi_{40}(M(8))$. Inspection of \cite[Pg. 3]{IsaksenChart} in stems $55$ and $56$ and AF $\geq 17$ reveals that there are no classes which could detect $v_1^8 u[1]$. Therefore $u[1]$ lifts to $u[18] \in \pi_{57}(M(8,v_1^8))$. 
\vspace{10pt}

\item The element $w \in \pi_{45}\tmf$ is detected by an element which is in the image of the map (\ref{eq:extmapF2}), and thus we deduce that $w \in \pi_{45}(S)$ maps to it.  A similar argument to the case above shows that $w$ lifts to $w[18] \in \pi_{63}(M(8,v_1^8))$. 

\end{enumerate}
\end{proof}

\begin{lem}\label{Lem:v28Lift}
The following classes lift to the top cell of $M(8,v_1^8)$:
\begin{enumerate}
\item $\Delta^2 \nu^2 \in \pi_{54}(\tmf)$, 
\item $\Delta^2 \kappa \nu \in \pi_{65}(\tmf)$,  
\item $\Delta^2 \eta^2 \bar{\kappa} \in \pi_{70}(\tmf)$.
\end{enumerate}
\end{lem}

\begin{proof}
We follow the proof of \cite[Thm. 11.1]{BHHM2} (which builds on \cite[Exm. 9.5]{BHHM2} and \cite[Prop. 10.1]{BHHM2}).

\begin{enumerate}

\item We begin with $\Delta^2 \nu^2 \in \pi_{54}(\tmf)$. This class lifts to an element 
$$\Delta^2 \nu^2[1] \in \tmf_{55}(M(8))$$
which is detected by 
$$v_2^8 h_2^2[1] \in \Ext^{12,55+12}_{A(2)_*}(H(8))$$ 
in the MASS for $\tmf \wedge M(8)$. 
Let 
$$\Delta^2 \nu^2[18] \in \tmf_{72}(M(8,v_1^8)).$$
be a lift of $\Delta^2 \nu^2[1]$. 
In the MASS for $\tmf \wedge M(8)$, there is a differential 
$$d_2(v_2^{10} v_1^4 h_2 h_0[1]) = v_2^8 v_1^8 h_2^2[1].$$
Since $v_2^{10} v_1^4 h_2 h_0[1]$ is a permanent cycle in the MASS for $\tmf \wedge M(8,v_1^8)$, it follows from the Geometric Boundary Theorem that $\Delta^2 \nu^2 [18]$ is detected by $v_2^{10} v_1^4 h_2 h_0[1]$ in the MASS for $\tmf \wedge M(8,v_1^8)$. In particular, we see that $\Delta^2 \nu^2[18]$ has modified Adams filtration (MAF) $18$ and stem $72$.

We now check that $v_2^{10} v_1^4 h_2 h_0[1]$ is a permanent cycle in the algebraic $\tmf$-resolution for $H(8,v_1^8)$. Its relative position\footnote{We will say that an element $x \in \Ext_{A(2)_*}(H(8,v_1^8))$ has \emph{relative position} $(t-s,s)$ in $\Ext_{A(2)_*}(\bou_I \otimes H(8,v_1^8))$ if the image of a differential supported by $x$ in the algebraic tmf-resolution lies in $\Ext^{s+1,t}_{A(2)_*}(\bou_I\otimes H(8,v_1^8))$, and the image of a differential supported by $x$ in the MASS could be detected in the algebraic tmf-resolution by an element in $\Ext^{s+r,t-r+1}_{A(2)_*}(\bou_1 \otimes H(8,v_1^8))$. In other words, if you were to pretend $x$ were an element in $\Ext^{s,t}_{A(2)_*}(\bou_I\otimes H(8,v_1^8))$, then $d_r$-differentials in the algebriac tmf-resolution ``look'' like Adams $d_1$'s, and $d_r$-differentials in the MASS ``look'' like Adams $d_r$'s.} is $t-s = 65$ and $AF = 17$, its relative position in $\Ext_{A(2)_*}(\bou_1^{\otimes 2} \otimes H(8,v_1^8))$ is $t-s = 58$ and $AF = 16$, and its relative position in $\Ext_{A(2)_*}(\bou_1^{\otimes 3} \otimes H(8,v_1^8))$ is $t-s = 51$ and $AF = 15$, the last of which lies above the vanishing line. Inspection of the relevant charts shows that $v_2^{10} v_1^4 h_2 h_0[1]$ cannot support a nontrivial $d_1$-differential since the target bidegrees are zero. Therefore $v_2^{10} v_1^4 h_2 h_0[1]$ is a permanent cycle in the algebraic $\tmf$-resolution for $H(8,v_1^8)$ and therefore it detects an element $\{v_2^{10} v_1^4 h_2 h_0[1]\}$ in $\Ext_{A_*}(H(8,v_1^8))$. 

Finally, inspection of the same algebraic tmf-resolution charts reveals that there are no possible targets for a nontrivial differential supported by $\{v_2^{10} v_1^4 h_2 h_0[1]\}$ in the MASS for $M(8,v_1^8)$. Therefore $\{v_2^{10} v_1^4 h_2 h_0[1]\}$ is a permanent cycle which detects a lift of $\Delta^2 \nu^2$. 
\vspace{10pt}

\item The class $\Delta^2 \kappa \nu \in \pi_{65}(\tmf)$ lifts to an element 
$$\Delta^2 \kappa \nu[1] \in \tmf_{66}(M(8))$$ 
which is detected by 
$$v_2^8 h_2 d_0[1] \in \Ext_{A(2)_*}^{15,66+15}(H(8))$$ 
in the MASS for $\tmf \wedge M(8)$. Lift $\Delta^2 \kappa \nu[1]$ to an element 
$$\Delta^2 \kappa \nu[18] \in \tmf_{83}(M(8,v_1^8)).$$ In the MASS for $\tmf \wedge M(8)$, there is a differential 
$$d_2(v_2^{10} v_1^4 d_0 h_0[1]) = v_2^8 v_1^8 h_2d_0[1].$$
It follows from the Geometric Boundary Theorem that $v_2^8 \kappa \nu[18]$ is detected by $v_2^{10} v_1^4 d_0 h_0[1]$ in the MASS for $\tmf \wedge M(8,v_1^8)$. 
In particular, we see that $\Delta^2 \kappa \nu[18]$ has MAF $21$ and stem $83$. 

We now check that $v_2^{10} v_1^4 d_0 h_0[1]$ is a permanent cycle in the algebraic $\tmf$-resolution for $H(8,v_1^8)$. Its relative position in $\Ext_{A(2)_*}(\bou_1 \otimes H(8,v_1^8))$ is $t-s = 76$ and $AF = 20$, its relative position in $\Ext_{A(2)_*}(\bou_1^{\otimes 2} \otimes H(8,v_1^8))$ is $t-s = 69$ and $AF = 19$, and its relative position in $\Ext_{A(2)_*}(\bou_1^{\otimes 3} \otimes H(8,v_1^8))$ is $t-s = 62$ and $AF = 18$, the last of which has targets only above the vanishing line. Inspection of the relevant charts shows that $v_2^{10} v_1^4 d_0 h_0[1]$ cannot support a nontrivial $d_1$-differential since the target bidegrees are zero. Therefore $v_2^{10} v_1^4 d_0 h_0[1]$ is a permanent cycle in the algebraic $\tmf$-resolution for $H(8,v_1^8)$ and detects an element $\{v_2^{10} v_1^4 d_0 h_0[1]\}$ in $\Ext_{A_*}(H(8,v_1^8))$. 

Finally, inspection of the same charts reveals that there are no possible targets for a nontrivial differential supported by $\{v_2^{10} v_1^4 d_0 h_0[1]\}$ in the MASS for $M(8,v_1^8)$. Therefore $\{v_2^{10} v_1^4 d_0 h_0[1]\}$ is a permanent cycle. 
\vspace{10pt}

\item The class $\Delta^2 \eta^2 \bar{\kappa} \in \pi_{70}(\tmf)$ lifts to an element 
$$\Delta^2 \eta^2 \bar{\kappa} [1] \in \tmf_{71}(M(8))$$
which is detected by 
$$g^2h^6_{2,1}[1] \in \Ext^{16,71+16}_{A(2)_*}(H(8))$$
in the MASS for $\tmf \wedge M(8)$. Lift $\Delta^2 \eta^2 \bar{\kappa}[1]$ to an element 
$$\Delta^2 \eta^2 \bar{\kappa}[18] \in \tmf_{88}(M(8,v_1^8)).$$ 
In the MASS for $\tmf \wedge M(8)$, there is a differential 
$$d_2(v_2^{8} v_1^4 d_0 e_0[1]) = g^2 v_1^8 h_{2,1}^6[1].$$
It follows from the Geometric Boundary Theorem that $\Delta^2 \eta^2 \bar{\kappa}[18]$ is detected by $v_2^{8} v_1^4 d_0 e_0[1]$ in the MASS for $\tmf \wedge M(8,v_1^8)$. In particular, we see that $\Delta^2 \eta^2 \bar{\kappa}[18]$ has MAF $24$ and stem $88$. 

We now check that $v_2^{8} v_1^4 d_0 e_0[1]$ is a permanent cycle in the algebraic $\tmf$-resolution for $H(8,v_1^8)$. Its relative position in $\Ext_{A(2)_*}(\bou_1 \otimes H(8,v_1^8))$ is $t-s = 81$ and $AF = 23$ and its relative position in $\Ext_{A(2)_*}(\bou_1^{\otimes 2} \otimes H(8,v_1^8))$ is $t-s = 74$ and $AF = 22$, the latter of which lies above the vanishing line. Inspection of the relevant charts shows that $v_2^{8} v_1^4 d_0 e_0[1]$ cannot support a nontrivial differential in the algebraic tmf-resolution for $H(8,v_1^8)$ since the target bidegrees are zero. Therefore $v_2^{8} v_1^4 d_0 e_0[1]$ is a permanent cycle in the algebraic $\tmf$-resolution for $H(8,v_1^8)$ and therefore lifts to an element $\{v_2^{8} v_1^4 d_0 e_0[1]\}$ in $\Ext_{A_*}(H(8,v_1^8))$. 

Finally, inspection of the same charts reveals that there are no possible targets for a nontrivial differential supported by $\{v_2^{8} v_1^4 d_0 e_0[1]\}$ in the MASS for $M(8,v_1^8)$. Therefore $\{v_2^{8} v_1^4 d_0 e_0[1]\}$ is a permanent cycle in the MASS for $M(8,v_1^8)$. 

\end{enumerate}

\end{proof}

\begin{lem}\label{Lem:v216neLift}
The following classes lift to the top cell of $M(8,v_1^8)$:
\begin{enumerate}
\item $\Delta^4 \nu^2 \in \pi_{102}(\tmf)$, $\Delta^4 \epsilon \in \pi_{104}(\tmf)$, $\Delta^4 \kappa \in \pi_{110}(\tmf)$, 
\item $\Delta^4 2 \bar{\kappa} \in \pi_{116}(\tmf)$. 

\end{enumerate}
\end{lem}

\begin{proof}$\quad$

\begin{enumerate}

\item These classes were lifted in \cite[Thm. 11.1]{BHHM2}. 
\vspace{10pt}

\item The class $\Delta^4 2 \bar{\kappa} \in \pi_{116}(\tmf)$ lifts to an element 
$$\Delta^4 2 \bar{\kappa}[1] \in \tmf_{117}(M(8))$$
which is detected by 
$$v_2^{16} h_0 g[1] \in \Ext^{23,117+23}_{A(2)_*}(H(8))$$
in the MASS for $\tmf \wedge M(8)$. Lift $\Delta^4 2 \bar{\kappa}[1]$ to an element 
$$\Delta^4 2 \bar{\kappa}[18] \in \tmf_{134}(M(8,v_1^8)).$$
In the MASS for $\tmf \wedge M(8)$, there is a differential 
$$d_2(v_2^{18} v_1^4 d_0 h_2[1]) = v_2^{16} v_1^8 h_0 g [1].$$ 
It follows from the Geometric Boundary Theorem that $\Delta^4 2 \bar{\kappa}[18]$ is detected by $v_2^{18} v_1^4 d_0 h_2[1]$ in the MASS for $\tmf \wedge M(8,v_1^8)$. In particular, we see that $\Delta^4 2\bar{\kappa}[18]$ has MAF $29$ and stem $134$. 

We now check that $v_2^{18} v_1^4 d_0 h_2[1]$ is a permanent cycle in the algebraic $\tmf$-resolution for $H(8,v_1^8)$. Its relative position in $\Ext_{A(2)_*}(\bou_1 \otimes H(8,v_1^8))$ is $t-s = 127$ and $AF = 28$, its relative position in $\Ext_{A(2)_*}(\bou_1^{\otimes 2} \otimes H(8,v_1^8))$ is $t-s = 120$ and $AF = 27$, and its relative position in $\Ext_{A(2)_*}(\bou_1^{\otimes 3} \otimes H(8,v_1^8))$ is $t-s = 113$ and $AF = 26$, the last of which lies above the vanishing line. Inspection of the relevant charts shows that $v_2^{16} 2 \bar{\kappa}[18]$ cannot support a nontrivial $d_1$-differential since the target bidegrees are zero. Therefore $v_2^{16} 2 \bar{\kappa}[18]$ is a permanent cycle in the algebraic $\tmf$-resolution for $H(8,v_1^8)$ and lifts to an element $v_2^{16} 2 \bar{\kappa}[18]$ in $\Ext_{A_*}(H(8,v_1^8))$. 

Finally, inspection of the same charts reveals that there are no possible targets for a nontrivial differential supported by $v_2^{16} 2 \bar{\kappa}[18]$ in the MASS for $M(8,v_1^8)$. Therefore $v_2^{16} 2 \bar{\kappa}[18]$ is a permanent cycle. 


\end{enumerate} 

\end{proof}

Contrary to the previous cases, there are several potential obstructions to lifting $\Delta^4 \bar{\kappa}\eta \in \pi_{117}(\tmf)$ to the top cell of $M(8,v_1^8)$ which are tricky to resolve. 
However, since this element is $2$-torsion and $v_1^4$-torsion, we may instead attempt to lift it to the top cell of the generalized Moore spectrum $M(2,v_1^4)$ of \cite{BHHM}, where the potential obstructions are much simpler to analyze. It then follows from the fact that the composite
$$ \Sigma^{8}M(2,v_1^4) \xrightarrow{\cdot 4v_1^4} M(8,v_1^8) \rightarrow S^{18} $$
is projection onto the top cell of $M(2,v_1^4)$ that $\Delta^4\bar{\kappa}\eta$ does lift to the top cell of $M(8,v_1^8)$.  

\begin{lem}\label{Lem:v216kbe}
The class $\Delta^4 \bar{\kappa}\eta \in \pi_{117}(\tmf)$ lifts to the top cell of $M(2,v_1^4)$.
\end{lem}

\begin{proof}
The class $\Delta^4 \eta \bar{\kappa} \in \pi_{117}(\tmf)$ lifts to an element 
$$\Delta^4 \eta \bar{\kappa}[1] \in \tmf_{118}(M(2))$$ 
which is detected by 
$$v_2^{16} h_1 g[1] \in \Ext^{21, 118+21}(H(2))$$ 
in the MASS for $\tmf \wedge M(2)$. Lift $\Delta^4\eta \bar{\kappa}[1]$ to an element 
$$\Delta^4 \eta \bar{\kappa}[10] \in \tmf_{127}(M(2,v_1^4)).$$ 
In the MASS for $\tmf \wedge M(2)$, there is a differential 
$$d_3(v_2^{20}h_2^2[1]) = v_2^{16} v_1^4 h_1 g[1].$$ 
It follows from the Geometric Boundary Theorem that $\Delta^4 \eta \bar{\kappa}[10]$ is detected by $v_2^{20} h_2^2[1]$ in the MASS for $\tmf \wedge M(2,v_1^4)$. In particular, we see that $\Delta^4 \eta \bar{\kappa}[10]$ has MAF $24$ and stem $127$. 

We now check that $v_2^{20} h_2^2[1]$ is a permanent cycle in the algebraic $\tmf$-resolution for $H(2,v_1^4)$. Its relative position in $\Ext_{A(2)_*}(\bou_1 \otimes H(2,v_1^4))$ is $t-s = 120$ and $AF = 23$, its relative position in $\Ext_{A(2)_*}(\bou_1^{\otimes 2} \otimes H(2,v_1^4))$ is $t-s = 113$ and $AF = 22$, and its relative position in $\Ext_{A(2)_*}(\bou_1^{\otimes 3} \otimes H(2,v_1^4))$ is $t-s = 106$ and $AF = 21$. Inspection of the relevant charts \cite[Figs. 6.4-6.5]{BHHM} shows that there is potentially a nontrivial differential
$$d_1(v_2^{20} h_2^2[1]) = x_{119,24},$$
in the algebraic tmf-resolution, where 
$$ x_{119,24} \in \Ext^{24, 119+24}_{A(2)_*}(\bou_1 \otimes H(2,v_1^4)), $$
but since $v_2^{20} h_2^2[1]$ is $v_2^{16}$-divisible and $x_{119,24}$ is not, this differential cannot occur (compare with the proof of \cite[Prop. 10.1]{BHHM2}). Therefore $v_2^{20} h_2^2[1]$ is a permanent cycle in the algebraic $\tmf$-resolution for $H(2,v_1^4)$ and therefore lifts to an element $\{v_2^{20} h_2^2[1]\}$ in $\Ext_{A_*}(H(2,v_1^4))$. 

Finally, inspection of the same charts reveals that there are no possible nontrivial differentials supported by $\{v_2^{20} h_2^2[1]\}$ in the MASS for $M(2, v_1^4)$. Therefore $\{v_2^{20} h_2^2[1]\}$ is a permanent cycle in the MASS for $M(2,v_1^4)$. 
\end{proof}

\begin{lem}\label{Lem:v216qLift}
The class $\Delta^4 q \in \pi_{128}(\tmf)$  lifts to the top cell of $M(8,v_1^8)$.
\end{lem}

\begin{proof}
The class $\Delta^4 q \in \pi_{128}(\tmf)$ lifts to an element
$$\Delta^4 q [1] \in \tmf_{129}(M(8))$$
which is detected by 
$$v_2^{20} c_0[1] \in \Ext^{23, 129+23}_{A(2)_*}(H(8))$$ 
in the MASS for $\tmf \wedge M(8)$. Lift $\Delta^4 q[1]$ to an element 
$$\Delta^4 q[18] \in \tmf_{146}(M(8,v_1^8)).$$ 
In the MASS for $\tmf \wedge M(8)$, there is a differential 
$$d_4(v_2^{16} g^2 h_{2,1} v_1^2[1]) = v_2^{20}v_1^8c_0[1].$$ 
It follows from the Geometric Boundary Theorem that $\Delta^4 q[18]$ is detected by $v_2^{16} g^2 h_{2,1} v_1^2[1]$ in the MASS for $\tmf \wedge M(8,v_1^8)$. In particular, we see that $\Delta^4 q[18]$ has MAF $29$ and stem $146$. 

We now check that $v_2^{16} g^2 h_{2,1} v_1^2[1]$ is a permanent cycle in the algebraic $\tmf$-resolution for $H(8,v_1^8)$. Its relative position in $\Ext_{A(2)_*}(\bou_1 \otimes H(8,v_1^8))$ is $t-s = 139$ and $AF = 28$, its relative position in $\Ext_{A(2)_*}(\bou_1^{\otimes 2} \otimes H(8,v_1^8))$ is $t-s = 132$ and $AF = 27$, and its relative position in $\Ext_{A(2)_*}(\bou_1^{\otimes 3} \otimes H(8,v_1^8))$ is $t-s = 125$ and $AF = 26$. 

The proof of Lemma~\ref{Lem:quwLift}(1) implies that the element 
$$ g^2 h_{2,1} v_1^2[1] \in \Ext_{A(2)_*}(H(8,v_1^8)) $$
is a permanent cycle in the algebraic tmf-resolution for $H(8,v_1^8)$.  It follows from Lemma~\ref{lem:v2^8} that 
$$ v_2^{16} g^2 h_{2,1} v_1^2[1] $$
is a permanent cycle in the algebraic tmf-resolution for $H(8,v_1^8)$, and detects an element 
$$ v_2^{16} \cdot \{ g^2 h_{2,1} v_1^2[1] \} \in \Ext_{A_*}(H(8,v_1^8)) $$
which persists to the $E_3$-page of the MASS for $M(8,v_1^8)$.

The only possibility for this element to support a non-trivial MASS differential is for it to support a $d_3$-differential whose target to by detected by the element
$$ {v_1 h_{2,1}^{19}(v_0^{-1}v_2^2[\zeta_1^8,\zeta_2^4])[18]} \in \Ext_{A(2)_*}(\bou_1^{\otimes 2} \otimes H(8,v_1^8)) $$
in the algebraic tmf-resolution for $H(8,v_1^8)$.

We wish to use Lemma~\ref{lem:technical} to argue that the element 
$v_1 h_{2,1}^{19}(v_0^{-1}v_2^2[\zeta_1^8,\zeta_2^4])[18]$ detects an element in $\Ext_{A_*}(H(8,v_1^8))$ which is zero in the $E_3$-page of the MASS.
In the MASS for $\bo_1^2 \wedge M(8,v_1^8)$, there is a differential
$$ d_2(v_2^8 h^{10}_{2,1}(v_0^{-1}v_2^2[\zeta_1^8,\zeta_2^4])[18]) = 
v_1 h_{2,1}^{19}(v_0^{-1}v_2^2[\zeta_1^8,\zeta_2^4])[18]. $$
Using the map
$$ \Sigma^{16}\tmf \wedge \bo^2_1 \wedge M(8,v_1^8) \hookrightarrow \tmf \wedge \br{\tmf}^2 \wedge M(8,v_1^8) $$
we get the same differential in the MASS for $\tmf \wedge \br{\tmf}^2 \wedge M(8,v_1^8)$.
By Proposition~\ref{prop:h21detectionH38}, the element $v_2^8 h^{10}_{2,1}(v_0^{-1}v_2^2[\zeta_1^8,\zeta_2^4])[18]$ is a permanent cycle in the algebraic tmf-resolution for $H(8,v_1^8)$, detecting the element
$$ \Delta^2 v_1^6 M(g^2)[1] \in \Ext_{A_*}(H(8,v_1^8)). $$
Therefore the hypotheses of Lemma~\ref{lem:technical} are satisfied, and we deduce that 
$$v_1 h_{2,1}^{19}(v_0^{-1}v_2^2[\zeta_1^8,\zeta_2^4])[18]$$ 
detects an element which is zero in the $E_3$-page of the MASS, and hence cannot be the target of a non-trivial $d_3$-differential in the MASS.
\end{proof}




\begin{lem}\label{Lem:v216uLift}
The class $\Delta^4 u \in \pi_{135}(\tmf)$ lifts to the top cell of $M(8,v_1^8)$. 
\end{lem}

\begin{proof}
The class $\Delta^4 u \in \pi_{135}(\tmf)$ lifts to an element
$$\Delta^4 u[1] \in \tmf_{136}(M(8))$$
which is detected by
$$v_2^{16}v_1^2 x_{35}[1] \in \Ext^{25,136+25}_{A(2)_*}(H(8))$$
in the MASS for $\tmf \wedge M(8)$. Lift $\Delta^4 u[1]$ to an element
$$\Delta^4 u[18] \in \tmf_{153}(M(8,v_1^8)).$$
There is a differential in the MASS for $\tmf \wedge M(8)$
$$d_4(v_2^{16} v_1^3 h_{2,1}^2 g^2[1]) = v_2^{16}v_1^{10} x_{35}[1],$$
so by the Geometric Boundary Theorem, $\Delta^4 u[18]$ is detected by $v_2^{16}v_1^3h^2_{2,1}g^2[1]$ in the MASS for $\tmf \wedge M(8,v_1^8)$. In particular, $\Delta^4 u[18]$ has MAF $31$ and stem $153$. 

We now check that $v_2^{16}v_1^3h_{2,1}^2g^2[1]$ is a permanent cycle in the algebraic $\tmf$-resolution for $H(8,v_1^8)$.  Note that $v_1^3h^2_{2,1}g^2[1]$ detects $u[18]$ in the MASS for $\tmf \wedge M(8,v_1^8)$.  In Lemma~\ref{Lem:quwLift}, we established that $u[18]$ lifts to $M(8,v_1^8)$, and therefore $v_1^3h^2_{2,1}g^2[1]$ is a permanent cycle in the algebraic tmf-resolution, and it detects a permanent cycle in the MASS for $M(8,v_1^8)$.
It follows from Lemma~\ref{lem:v2^8} that
$$ v_2^{16}v_1^3h_{2,1}^2g^2[1] $$
is a permanent cycle in the algebraic tmf-resolution, and detects an element
$$ v_2^{16} \cdot \{v_1^3h_{2,1}^2g^2[1]\} \in \Ext_{A_*}(H(8,v_1^8)). $$
Inspection of the relevant charts shows that the only possible non-trivial MASS differentials supported by this element would be
$$d_2(v_2^{16}\cdot \{v_1^2h_{2,1}^2g^2[1]\}) = \{ v_2^8 h^{15}_{2,1}\zeta_2^4[18] \}. $$
However, we have
$$ d_2(v_2^{16} \cdot \{v_1^3h_{2,1}^2g^2[1]\}) = 0, $$
since it is a product of $d_2$-cycles.
\end{proof}

\begin{lem}\label{Lem:v224nu2Lift}
The following classes lift to the top cell of $M(8,v_1^8)$:
\begin{enumerate}
\item $\Delta^6 \nu^2 \in \pi_{150}(\tmf)$,
\item $\Delta^6 \kappa \nu \in \pi_{161}(\tmf)$. 
\end{enumerate}
\end{lem}

\begin{proof}$\quad$
\begin{enumerate}

\item The class $\Delta^6 \nu^2 \in \pi_{150}(\tmf)$ lifts to an element 
$$\Delta^6 \nu^2 [1] \in \tmf_{151}(M(8))$$
which is detected by 
$$v_2^{24} h_2^2[1] \in \Ext^{28,151+28}_{A(2)_*}(H(8))$$ 
in the MASS for $\tmf \wedge M(8)$. Lift $\Delta^6 \nu^2[1]$ to an element 
$$\Delta^6 \nu^2 [18] \in \tmf_{168}(M(8,v_1^8)).$$ 
In the MASS for $\tmf \wedge M(8)$, there is a differential 
$$d_2(v_2^{26} v_1^4 h_2 h_0[1]) = v_2^{24} v_1^8 h_2^2[1].$$ 
It follows from the Geometric Boundary Theorem that $\Delta^6 \nu^2 [18]$ is detected by $v_2^{26} v_1^4 h_2 h_0[1]$ in the MASS for $\tmf \wedge M(8,v_1^8)$. In particular, we see that $\Delta^6 \nu^2 [18]$ has MAF $34$ and stem $168$. 

In Lemma~\ref{Lem:v28Lift}(1) we showed that $v_2^{10} v_1^4 h_2 h_0[1]$ is a permanent cycle in the algebraic tmf-resolution, detecting an element
$$ \{v_2^{10} v_1^4 h_2 h_0[1]\} \in \Ext_{A_*}(H(8,v_1^8)) $$
in the algebraic tmf-resolution for $H(8,v_1^8)$.
By Lemma~\ref{lem:v2^8}, this is also true of $v_2^{26} v_1^4 h_2 h_0[1]$.

Lemma~\ref{lem:v2^8} implies that $d_2(v_2^{16}) = 0$ in the MASS for $M(8,v_1^8)$.  By Lemma~\ref{Lem:v28Lift}(1), it follows that
$$ d_2(v_2^{16} \cdot \{v_2^{10} v_1^4 h_2 h_0[1]\}) = 0. $$
Inspection of the algebraic tmf-resolution charts reveals that there are no possible targets of a longer MASS differential supported by $v_2^{16} \cdot \{v_2^{10} v_1^4 h_2 h_0[1]\}$.
\vspace{10pt}

\item The class $\Delta^6 \kappa \nu \in \pi_{161}(\tmf)$ lifts to an element 
$$\Delta^6 \kappa \nu[1] \in \tmf_{162}(M(8))$$ 
which is detected by 
$$v_2^{24} d_0 h_2[1] \in \Ext^{31,161+31}_{A(2)_*}(H(8))$$ 
in the MASS for $\tmf \wedge M(8)$. Lift $\Delta^6 \kappa \nu[1]$ to an element 
$$\Delta^6 \kappa \nu[18] \in \tmf_{179}(M(8,v_1^8)).$$ 
In the MASS for $\tmf \wedge M(8)$, there is a differential 
$$d_2(v_2^{26} v_1^4 h_0 d_0[1]) = v_2^{24} v_1^8 h_2 d_0[1].$$
It follows from the Geometric Boundary Theorem that $\Delta^6 \kappa \nu[18]$ is detected by $v_2^{26} v_1^4 h_0 d_0[1]$ in the MASS for $\tmf \wedge M(8,v_1^8)$. In particular, we see that $\Delta^6 \kappa \nu[18]$ has MAF $37$ and stem $179$. 

We showed in Lemma~\ref{Lem:v28Lift} that $v_2^{10}v_1^4 h_0 d_0[1]$ is a permanent cycle in the algebraic tmf-resolution.  By Lemma~\ref{lem:v2^8}, it follows that $v_2^{26} v_1^4 h_0 d_0[1]$ is a permanent cycle in the algebraic $\tmf$-resolution for $H(8,v_1^8)$ and lifts to an element $\{v_2^{26} v_1^4 h_0 d_0[1]\}$ in $\Ext_{A_*}(H(8,v_1^8))$. 

Finally, inspection of the algebraic tmf-resolution charts reveals that there are no possible nontrivial differentials on $\{v_2^{26} v_1^4 h_0 d_0[1]\}$ in the MASS for $M(8,v_1^8)$. Therefore $\{v_2^{26} v_1^4 h_0 d_0[1]\}$ is a permanent cycle. 

\end{enumerate}
\end{proof}

\begin{lem}\label{lem:nu2epsilon}
The classes $\Delta^8 \nu^2 \in \pi_{198}\tmf$ and $\Delta^8 \epsilon \in \pi_{200}\tmf$ lift to the top cell of $M(8,v_1^8)$.
\end{lem}

\begin{proof}
The classes $\Delta^8 \nu^2 \in \pi_{198}(\tmf)$ and $\Delta^8 \epsilon \in \pi_{200}\tmf$ lift to elements 
\begin{align*}
\Delta^8 \nu^2[1] & \in \tmf_{199}(M(8)), \\
\Delta^8 \epsilon[1] & \in \tmf_{201}(M(8))
\end{align*}
which are detected by 
\begin{align*}
v_2^{32}h_2^2[1] & \in \Ext^{36,199+36}_{A(2)_*}(H(8)), \\
v_2^{32}c_0[1]& \in \Ext^{37,201+37}_{A(2)_*}(H(8))
\end{align*}
in the MASS for $\tmf \wedge M(8)$. Lift $\Delta^8 \nu^2[1]$ and $\Delta^8\epsilon[1]$ to elements 
\begin{align*}
\Delta^8 \nu^2[18] & \in \tmf_{210}(M(8,v_1^8)), \\
\Delta^8 \epsilon[18] & \in \tmf_{212}(M(8,v_1^8)).
\end{align*}
In the MASS for $\tmf \wedge M(8)$, there are differentials 
\begin{align*}
d_2(v_2^{32} v_1^4 h_0h_2v_2^2[1]) & = v_2^{32} v_1^8 h_2^2 [1], \\
d_3(v_2^{32} v_1^4 e_0[1]) & = v_2^{32} v_1^8 c_0 [1].
\end{align*} 
It follows from the Geometric Boundary Theorem that $\Delta^8 \nu^2[18]$ is detected by $v_2^{34} v_1^4 h_0h_2[1]$ and $\Delta^8 \epsilon[18]$ is detected by $v_2^{32}v_1^4e_0[1]$ in the MASS for $\tmf \wedge M(8,v_1^8)$. 

In \cite[Thm.~11.1]{BHHM2} the classes $\Delta^4\nu^2[18] \in \pi_{120}M(8,v_1^8)$ and $\Delta^4 \epsilon[18] \in \pi_{122}M(8,v_1^8)$ were produced by showing that the elements 
\begin{align*} 
v_2^{18}v_1^4h_0h_2[1] & \in \Ext^{26, 120+26}_{A(2)_*}(H(8,v_1^8)), \\
v_2^{16}v_1^4e_0[1] & \in \Ext^{26, 122+26}_{A(2)_*}(H(8,v_1^8))
\end{align*}
detect via the algebraic tmf-resolution elements 
\begin{align*} 
\{v_2^{18}v_1^4h_0h_2[1]\} & \in \Ext^{26,120+26}_{A_*}(H(8,v_1^8)), \\
\{v_2^{16}v_1^4e_0[1]\} & \in \Ext^{26,122+26}_{A_*}(H(8,v_1^8)) 
\end{align*} 
which are permanent cycles in the MASS for $M(8,v_1^8)$.

Since the element $v_2^{16} \in \Ext_{A_*}(H(8,v_1^8))$ is the square of the element $v_2^8$, we have $d_2(v_2^{16}) = 0$.  
We deduce that the elements
\begin{align*} 
v_2^{16}\cdot \{v_2^{18}v_1^4h_0h_2[1]\} & \in \Ext^{26,120+26}_{A_*}(H(8,v_1^8)), \\
v_2^{16}\cdot\{v_2^{16}v_1^4e_0[1]\} & \in \Ext^{26,122+26}_{A_*}(H(8,v_1^8)) 
\end{align*} 
persist to the $E_3$ page of the MASS for $M(8,v_1^8)$.  If we can show they are permanent cycles, we are done.

We begin with $\{v_2^{34}v_1^4h_0h_2[1]\}$.
Examination of the algebraic tmf-resolution for $M(8,v_1^8)$ reveals that the only possibility of a non-trivial differential in the MASS supported by this element would be a $d_4(\{v_2^{34} v_1^4h_0h_2[1]\})$ which would be detected by 
$$ h_{2,1}^{33}v_1v_0^{-1}v_2^2[\zeta_1^8,\zeta_2^4][18] \in \Ext_{A(2)_*}(\bou_1^{\otimes 2} \otimes H(8,v_1^8)). $$
In the MASS for $\tmf \wedge \bo_1^{\wedge 2}$ there is a differential
$$ d_2^{mass}(\Delta^2 h_{2,1}^{24} v_0^{-1}v_2^2[\zeta_1^8,\zeta_2^4][18]) = 
h_{2,1}^{33}v_1 v_0^{-1}v_2^2[\zeta_1^8,\zeta_2^4][18]. $$ 
Using the map (\ref{eq:bo1^2map}) we deduce that there is a corresponding differential in the MASS for $\tmf \wedge \br{\tmf}^{\wedge 2}$.  
The elements 
\begin{gather*}
h_{2,1}^{33}v_1v_0^{-1}v_2^2[\zeta_1^8,\zeta_2^4][18], \\
\Delta^2 h_{2,1}^{24}v_1v_0^{-1}v_2^2[\zeta_1^8,\zeta_2^4][18]
\end{gather*}
respectively detect 
\begin{align*}
 v_1^7 h_{2,1}^{23} Mg^2[1] & \in \Ext_{A_*}(H(8,v_1^8)), \\
 \Delta^2 v_1^6 h_{2,1}^{14} Mg^2 [1] & \in \Ext_{A_*}(H(8,v_1^8)) 
\end{align*}
in the algebraic tmf-resolution for  $M(8,v_1^8)$.
We therefore deduce from Lemma~\ref{lem:technical} that $v_1^7 h_{2,1}^{23} Mg^2[1]$ is killed by 
$$ d^{mass}_2 (\Delta^2 v_1^6 h_{2,1}^{14} Mg^2[1]) $$
in the MASS for $M(8,v_1^8)$.
Therefore it cannot be the target of a non-trivial $d^{mass}_4$.

We now consider $\{v_2^{32}v_1^4e_0[1]\}$.
Examination of the algebraic tmf-resolution for $M(8,v_1^8)$ reveals that the only possibility of a non-trivial differential in the MASS supported by this element would be a $d_4(\{v_2^{32} v_1^4e_0[1]\})$ which would be detected by 
$$ \Delta^2 h^{28}_{2,1}\zeta_2^4[18] \in \Ext_{A(2)_*}(\bou_1 \otimes H(8,v_1^8)). $$
in the algebraic tmf-resolution for $M(8,v_1^8)$.
To eliminate this possibility we wish to employ 
Case (1) of Remark~\ref{rmk:technical}, using the differential
$$ d^{mass}_3(\Delta^2 h^{25}_{2,1}v_0^{-2}v_2^2\zeta_1^{16})[18]) = \Delta^2 h_{2,1}^{28} \zeta_2^4[18] $$
in the MASS for $\tmf \wedge \br{\tmf} \wedge M(8,v_1^8)$.
The element $\Delta^2 h^{25}_{2,1}v_0^{-2}v_2^2\zeta_1^{16})[18]$ detects the element
$$ \Delta^2 h^{19}_{2,1} Q_2[18] \in \Ext_{A_*}(H(8,v_1^8)) $$
in the algebraic tmf-resolution for $M(8,v_1^8)$.  We just need to check that there is no possibility for $\Delta^2 h^{19}_{2,1} Q_2[18]$ to support a non-trivial $d^{mass}_2$ in the MASS for $M(8,v_1^8)$.  However, examination of the algebraic tmf-resolution for $M(8,v_1^8)$ reveals there are no classes which could detect the target of such a non-trivial $d^{mass}_2$.
\end{proof}

\bibliographystyle{amsalpha}

\nocite{*}
\bibliography{tmfhi}

\end{document}